\documentclass[12pt,reqno]{article}
\usepackage{amssymb}
\usepackage{amsmath}
\usepackage{mathrsfs}
\usepackage{amsfonts}
\usepackage{bm}
\usepackage[colorlinks,linkcolor=red,anchorcolor=green,citecolor=blue]{hyperref}
\usepackage{graphicx,cite,enumerate,tikz,amsthm,amscd,bm,amssymb,changepage,mathrsfs}
\usepackage{indentfirst}
\newtheorem{proposition}{Proposition}[section]
\newtheorem{lemma}[proposition]{Lemma}
\newtheorem{theorem}[proposition]{Theorem}

\newtheorem{remark}[proposition]{Remark}

\renewcommand{\theequation}{\thesection.\arabic{equation}}
\renewcommand{\oddsidemargin}{5mm}

\begin{document}
	\parindent 15pt
	\renewcommand{\theequation}{\thesection.\arabic{equation}}
	\renewcommand{\baselinestretch}{1.15}
	\renewcommand{\arraystretch}{1.1}
	\renewcommand{\vec}[1]{\bm{#1}}
	\title{Optimal uniform regularity and asymptotic behavior of solutions to Lotka-Volterra type systems with strong competition and asymmetric coefficients  \thanks{This work is supported by National Natural Science Foundation of China(12401136,12471113).}}
	\author{Zexin Zhang\thanks{Email: zhangzexin@amss.ac.cn}\\
		{\small School of Mathematical Science, Jiangsu University, Zhenjiang 212013, P. R. China.}
	}
	\date{}
	\maketitle
	\begin{abstract}
		In this paper, we investigate the uniform regularity and asymptotic behavior of solutions to the following Lotka-Volterra type system of strong competition with Dirichlet boundary conditions:
		\begin{align*}
			\left\{
			\begin{array}{ll}
				-\Delta u_{i,\beta} = f_{i,\beta}(x, u_{i,\beta}) - \beta u_{i,\beta}^{p_i} \sum_{\substack{j=1 \\ j \neq i}}^k a_{ij} u_{j,\beta}^{p_j}, \quad u_{i,\beta} > 0 & \text{in } \Omega, \\
				u_{i,\beta} = \varphi_{i,\beta} & \text{on } \partial\Omega,
			\end{array}\right.
		\end{align*}
		where $N \geq 1$, $1 \leq i \leq k$ with $k \geq 3$, $\beta > 0$, $p_i \geq 1$, $a_{ij} > 0$ for $i \neq j$, and $\Omega$ is a $C^{1,\text{Dini}}$ bounded domain in $\mathbb{R}^N$. First, we prove that the uniform boundedness of the solutions implies their uniform interior and global Lipschitz boundedness as $\beta \to +\infty$. Such uniform results are optimal; partial versions thereof are known in the literature for symmetric coefficients (i.e., $a_{ij} = a_{ji}$ for all $i \neq j$) and homogeneous competition terms (i.e., $p_i = p_j$ for all $i\neq j$). Here, we establish an Alt-Caffarelli-Friedman type monotonicity formula for the system and then employ blow-up analysis to show that these results also hold in the asymmetric or nonhomogeneous case. Next, as consequences of the uniform optimal regularity, we derive sharp quantitative pointwise estimates for the densities near the interface between different components.
	\end{abstract}
	
	\section{Introduction}\label{sec:1}
	\setcounter{section}{1}
	\setcounter{equation}{0}
	
	In this paper, we are concerned with the optimal uniform regularity and asymptotic behavior of a sequence of weak solutions to the following strongly competing systems:
	\begin{align}\label{Sys:main-interior}
		\left\{
		\begin{array}{ll}
			-\Delta u_{i,n}=f_{i,n}(x,u_{i,n})-\beta_n u_{i,n}^{p_i}\sum\limits_{j\neq i}^ka_{ij}u_{j,n}^{p_j}\quad&\text{in}~~\Omega,\\
			u_{i,n}>0\quad&\text{in}~~\Omega
		\end{array}\right.
	\end{align}
	with Dirichlet boundary condition
	\begin{align}\label{Boundary:main}
		u_{i,n}=\varphi_{i,n}\quad\text{on}~~\partial\Omega,
	\end{align}
	where $N\geq 1$, $k\geq 3$, $1\leq i\leq k$, $\beta_n>0$, $p_i\geq 1$, $a_{ij}>0$ for any $i\neq j$, and  $\Omega$ is a  bounded $C^{2}$ domain in $\mathbb{R}^N$. Here, the internal reaction term $\bm{f}_{n}: \Omega \times \mathbb{R}^+ \to \mathbb{R}^k$ and the boundary datum $\bm{\varphi}_{n}: \partial\Omega \to \mathbb{R}^k$ satisfy the following conditions, respectively.\\
	(H1)  $\bm{f}_{n}:\Omega\times \mathbb{R}^+\to \mathbb{R}^k$ is a Carath\'{e}odory (vector) function (i.e., $u\mapsto \bm{f}_{n}(x,u)$ is continuous in $\mathbb{R}^+$ for almost $x\in \Omega$, and $x\mapsto \bm{f}_{n}(x,u)$ is Lebesgue measurable in $\Omega$ for any $u>0$); moreover, for any $m>0$, there exists $d_m>0$ independent of $n$ such that
	$$\|f_{i,n}(\cdot,s)\|_{L^{\infty}(\Omega)}\leq d_m,\quad \forall 0<s\leq m,~~ \forall i=1,\dots,k.$$
	(H2) $\bm{\varphi}_{n}: \partial\Omega \to \mathbb{R}^k$ is a componentwise nonnegative $C^{1,\text{Dini}}$ function satisfying 
	\begin{align}\label{Condi:boundary-disjoint}
		\varphi_{i,n}\cdot \varphi_{j,n}\equiv 0\quad\text{on}~~\partial\Omega,\quad \forall i\neq j,\quad \forall n\geq 1,
	\end{align}
	and 
	$$\|\varphi_{i,n}\|_{C^{1,\mathcal{A}}(\partial\Omega)}:=\|\varphi_{i,n}\|_{C^{1}(\partial\Omega)}+\sum_{j=1}^{N}\left[\partial_{x_j}\varphi_{i,n}\right]_{\mathcal{A};\partial\Omega}\leq \bar{m},\quad \forall 1\leq i\leq k$$
	for some constant $\bar{m}>0$ independent of $n$, where $\mathcal{A}:[0,+\infty)\to [0,+\infty)$ is a fixed (i.e. independent of $n$) non-decreasing function satisfying 
	$$\lim\limits_{t\to 0^+}\mathcal{A}(t)=0,~~\int_{0}^{d}\frac{\mathcal{A}(t)}{t}dt<+\infty,~~\frac{\mathcal{A}(t)}{t^{\alpha}}~\text{is non-increasing in $(0,d)$},~~\liminf_{t\to 0^+}\frac{\mathcal{A}(2t)}{\mathcal{A}(t)}>1$$
	for some fixed constants $d>0$ and $\alpha\in (0,1)$, and 
	$$\left[\partial_{x_j}\varphi_{i,n}\right]_{\mathcal{A};\partial\Omega}:=\sup_{x,y\in \partial\Omega\atop x\neq y}\frac{|\partial_{x_j}\varphi_{i,n}(x)-\partial_{x_j}\varphi_{i,n}(y)|}{\mathcal{A}(|x-y|)}.$$
	\par 
	Here and throughout the remainder of this paper, we use the boldface letter $\bm{w}$ to denote a $k$-component vector-valued function, i.e., $\bm{w} = (w_1, \dots, w_k)$. We denote by $B_r(x_0)$ the open ball in $\mathbb{R}^N$ with radius $r > 0$ and center $x_0$. For notational simplicity, we also write $B_r := B_r(0)$.
	\par 
	When $p_i = 1$ for all $1 \leq i \leq k$, system \eqref{Sys:main-interior} reduces to the well-known Lotka-Volterra system, a fundamental mathematical framework extensively employed in the field of population dynamics, see e.g.  \cite{Conti-Terracini-Verzini-2005adv,Dancer-D1995NA-1,Dancer-D1995NA-2,book:Murray2003}. In this context, $u_{i,n}$ denotes the population density of the $i$-th species; its intrinsic (internal) dynamics are governed by the function $ f_{i,n}$; the positive constants $\beta_n \cdot a_{ij}$ quantify the interaction strength between the $i$-th species $ u_{i,n}$ and the $j$-th species $ u_{j,n}$;  the larger $\beta_n$ is, the stronger the interaction strength becomes; moreover, the limit densities exhibit spatial segregation as $\beta_n \to +\infty$. Notice that the interspecific interaction may be asymmetry (i.e., $a_{ij} \neq a_{ji}$ for some $i\neq j$); in this case, spiraling waves arise in the singular limit of the system when $\beta_n\to+\infty$ as $n\to\infty$ (see e.g. \cite{Terracini-V-Z2019CPAM,Salort-Terracini-VZ2025analPDE}). Spiraling waves also occur in other contexts of reaction-diffusion systems (see e.g. \cite{book:Murray2003}). For general $p_i\geq 1$, system \eqref{Sys:main-interior} appears in the modeling of diffusion flames (\cite{Caffarelli-R2007ARMA}), and becomes the Lotka-Volterra system for the fast-diffusion equation by the change of variable $v_{i,n}:=u_{i,n}^{p_i}$ (\cite{book:Vazquez-2007,Soave-Z2015ARMA}).
	\par 
	The existence of solutions of system \eqref{Sys:main-interior} attracts many scholars to study and has been obtained e.g. in \cite{Eilbeck-F-LG-JDE1994,Conti-Terracini-Verzini-2005adv,Dancer-D1995NA-1,Dancer-D1995NA-2}. The present paper focuses on another object of an intense research in the last two decades: the asymptotic behaviour of solutions $\{\bm{u}_{n}\}$ of system \eqref{Sys:main-interior} as $\beta_n\to+\infty$, where $\{\bm{u}_{n}\}$ is uniformly bounded in $L^{\infty}(\Omega;\mathbb{R}^k)$ in the sense that there exists $m>0$ independent of $n$ such that
	\begin{align}\label{Condi:Bound-u}
		\max_{1\leq i\leq k}\|u_{i,n}\|_{L^{\infty}(\Omega)}\leq m.
	\end{align}
	In many cases, condition \eqref{Condi:Bound-u} can be derived by imposing suitable boundary data on $\partial\Omega$ and applying the maximum principle.
	\par 
	In the case of two species, i.e., $k=2$, by assuming $\bm{f}_{n}\equiv 0$ in $\Omega$ and $p_1=p_2=1$, Conti et al. \cite{Conti-Terracini-Verzini-2005adv} used a blow-up technique to prove that if $\{\bm{u}_{n}\}\subseteq H^1(\Omega;\mathbb{R}^k)$ is a sequence of weak solutions to  \eqref{Sys:main-interior}-\eqref{Boundary:main} with $\beta_n\to +\infty$ (as $n\to+\infty$) and $\bm{\varphi}_{n}$ satisfying (H2), then there exists $M>0$ independent of $n$ such that 
	$$\|\bm{u}_{n}\|_{\text{Lip}(\Omega)}:=\|\bm{u}_{n}\|_{L^{\infty}(\Omega;\mathbb{R}^k)}+\|\nabla\bm{u}_{n}\|_{L^{\infty}(\Omega;\mathbb{R}^{kN})}\leq M,\quad \forall n\geq 1.$$
	Notice that in this case, system \eqref{Sys:main-interior} admits reduction to a single governing equation, significantly simplifying analytical procedures. We remark that this uniform regularity is optimal as the limiting profile $\bm{u}$ is at most Lipschitz continuous in $\Omega$ by the Hopf lemma. Afterwards, for general $p_i\geq 1(i=1,2)$, Caffarelli and Roquejoffre \cite{Caffarelli-R2007ARMA} provided an induction argument to prove that if $\{\bm{u}_{n}\}\subseteq H^1(\Omega;\mathbb{R}^k)$ is a sequence of weak solutions to  \eqref{Sys:main-interior}-\eqref{Condi:Bound-u} with $\beta_n\to +\infty$, then for any compact set $K\Subset \Omega$,  there exist $\alpha\in (0,1)$ such that 
	$\{\bm{u}_{n}\}$ is bounded in the H\"{o}lder space $C^{\alpha}(K;\mathbb{R}^k)$. This uniform interior (or local) regularity result was later refined to the optimal regularity setting by Soave and Zilio \cite{Soave-Z2015ARMA}, whose proof hinges on blow-up analysis and the celebrated almost monotonicity formula of Caffarelli-Jerison-Kenig \cite{Caffarelli-J-K2002annals}. It is worth pointing out that in this two-species case, the asymmetry of the matrix $(a_{ij})$ does not pose any difficulty, since the normalization of each $a_{ij}$ to unity can be achieved via a simple rescaling of the unknowns; moreover, the competition reaction function 
	\begin{align}\label{def:competition-reaction}
		\bm{\hat{g}}(\bm{u}):=\left(u_{1}^{p_1}\sum\limits_{j\neq 1}^ka_{1j}u_{j}^{p_j},\dots,u_{k}^{p_k}\sum\limits_{j\neq k}^ka_{kj}u_{j}^{p_j}\right)
	\end{align}
	is always homogeneous with respect to $\bm{u}$, for any $p_1,p_2\geq 1$, which plays a pivotal role in blow-up analysis.
	\par 
	In the case of multiple species, i.e., $k\geq 3$, the uniform H\"{o}lder regularity of solutions to \eqref{Sys:main-interior}-\eqref{Boundary:main} with $p_i=1$ for all $1\leq i\leq k$ was deduced by Conti et al. \cite{Conti-Terracini-Verzini-2005adv} via blow-up analysis and an Liouville-type result for H\"{o}lder continuous functions (\cite[Theorem 5]{Conti-Terracini-Verzini-2005adv}). Note that this Liouville-type theorem is not applicable to Lipschitz continuous functions; consequently, the procedure cannot be applied to obtain optimal regularity. 
	Later, Wang and Zhang \cite{Wang-Zhang2010poincare} employed the Kato inequality in conjunction with the maximum principle to establish the uniform global Lipschitz regularity of solutions to \eqref{Sys:main-interior}-\eqref{Boundary:main}, under the assumptions that $p_i=1$ for all $i$, the matrix $(a_{ij})$ is symmetric (i.e., $a_{ij}=a_{ji}$ for any $i\neq j$), and suitable conditions are imposed on the internal reaction function $\bm{f}_{n}$ (stronger than (H1)). They also obtained analogous results in the parabolic setting.
	Soave and Zilio \cite{Soave-Z2015ARMA} subsequently investigated the uniform local Lipschitz regularity of solutions to system \eqref{Sys:main-interior} in the general case where $p_i \geq 1$ for all $i$, with the coefficient matrix $(a_{ij})$ being totally symmetric (i.e., $a_{ij} = 1$ for all $i \neq j$). Their analysis relied on the blow-up technique and the almost monotonicity formula of Caffarelli-Jerison-Kenig. However, a critical observation is that their proof implicitly exploited the homogeneity property of the competition term $\bm{\hat{g}}$ as defined in \eqref{def:competition-reaction}. Note that in this case (i.e., $k\geq 3$), the homogeneity of $\bm{\hat{g}}$ implies $p_1=p_2=\cdots=p_k$.
	Therefore, it remains an open problem whether analogous optimal regularity results hold in the scenarios where the coefficient matrix $(a_{ij})$ is asymmetric or the competition term $\bm{\hat{g}}$ is nonhomogeneous. The present paper aims to provide a positive answer to this open question.
	\par 
	Some interesting extensions regarding the uniform regularity of solutions to system \eqref{Sys:main-interior} have been reported in the literature. Specifically, we refer to Wei and Weth \cite{Wei-W2008Nonlinearity} for results on the uniform equicontinuity of solutions to a wider class of elliptic systems (in dimension $N=2$); to Quitalo \cite{Quitalo2013ARMA} for those on the uniform H\"{o}lder regularity associated with extremal Pucci operators; to Verzini and Zilio \cite{Verzini-Z2014CPDE} for those on the uniform H\"{o}lder regularity associated with fractional operators; and more recently, to Soave and Terracini \cite{Soave-Terracini2023jems} for developments concerning the uniform H\"{o}lder regularity related to anisotropic operators.
	We would also like to mention that the problem of uniform regularity of solutions to Gross-Pitaevskii type systems with strong competition has attracted significant attention from many scholars \cite{Conti-Terracini-Verzini-2002poincare,Conti-Terracini-Verzini-2003JFA,Chang-LLL2004PhysD,Caffarelli-L2008JAMS,Berestycki-LCWZ2013ARMA,Dias-T2023na,Dancer-WZ2012JFA,Caffarelli-K-L2009fixpoint,Soave-Z2015ARMA,Noris-TV2010CPAM,Soave-T-T-Z2016-na,Soave-Terracini2023jems,Zhang-Zhang2025}. 
	We remark that if $p_i = 1$ for all $i$ and the matrix $(a_{ij})$ is symmetric, then as shown in \cite{Caffarelli-K-L2009fixpoint, Tavares-T2012CVPDE}, an Almgren type monotonicity formula is available for the limiting profiles of solutions to system \eqref{Sys:main-interior} with $\bm{f}_n\equiv 0$ in $\Omega$ and $\lim_{n\to\infty}\beta_n=+\infty$---analogous to the variational Gross-Pitaevskii system---and one can thus employ standard tools (e.g., dimension estimates) to study the free boundaries of the corresponding limiting profiles. Note that the availability of the Almgren type monotonicity formula allowed Soave and Zilio \cite{Soave-Z2015ARMA} to derive the uniform interior optimal regularity of solutions to the Gross-Pitaevskii system with strong competition. However, such a monotonicity formula appears to be unavailable for system \eqref{Sys:main-interior} when $(a_{ij})$ is asymmetric (see \cite{Terracini-V-Z2019CPAM}), which renders the uniform regularity problem more involved.
	\par 
	Very recently, in cooperation with Zhang, we \cite{Zhang-Zhang2025} proposed a new blow-up argument for deducing the uniform local (and global) Lipschitz regularity of solutions to the Gross-Pitaevskii system. This novel argument enables us to eliminate both the pre-compactness assumption on $\{\bm{f}_{n}\}$ and the dimensional restriction, which constitute technical conditions in the proof of \cite[Theorem 1.3]{Soave-Z2015ARMA}. In this paper, we first demonstrate that the aforementioned method can be adapted to the non-variational system \eqref{Sys:main-interior}, thereby establishing the uniform local optimal regularity of its solutions. Specifically, to achieve this objective, we develop an Alt-Caffarelli-Friedman type monotonicity formula for system \eqref{Sys:main-interior} (see Theorem \ref{thm:ACF-formula}), which constitutes one of the key contributions of the present work. Furthermore, to address the challenge stemming from the lack of homogeneity in the competition term $\bm{\hat{g}}$, we propose an inductive argument to construct a new blow-up sequence (see Lemma \ref{lem:new-u-n-property}); this construction guarantees the applicability of the Alt-Caffarelli-Friedman type monotonicity formula in the current framework. Indeed, we can obtain the following more general result:
	\begin{theorem}\label{thm:interior-lip}
		Let $N \geq 1$ and $\Omega$ be a domain in $\mathbb{R}^N$ (not necessarily proper, bounded, or smooth). Let $k \geq 3$, $p_i \geq 1$, and $a_{ij} > 0$ for all $1 \leq i, j \leq k$ with $i \neq j$. For each $n \in \mathbb{N}$, let $\bm{u}_n \in H^1_{\text{loc}}(\Omega; \mathbb{R}^k) \cap C_{\text{loc}}(\Omega; \mathbb{R}^k)$ be a weak solution to the following system:
		\begin{align}\label{Sys:main-interior-general}
			\left\{
			\begin{array}{ll}
				-\Delta u_{i,n} = f_{i,n}(x, u_{i,n}) - u_{i,n}^{p_i} \sum_{j\neq i}^k a_{ij} \beta_{ij;n} u_{j,n}^{p_j} & \text{in } \Omega, \\
				u_{i,n} > 0 & \text{in } \Omega,
			\end{array}\right. \quad 1 \leq i \leq k,
		\end{align}
		where $\beta_{ij;n} = \beta_{ji;n} > 0$ for all $i \neq j$, and $\bm{f}_n$ satisfies condition (H1). Suppose that $\{\bm{u}_n\}$ is uniformly bounded in $L^{\infty}(\Omega; \mathbb{R}^k)$ (i.e., \eqref{Condi:Bound-u} holds). Then for any compact set $K \Subset \Omega$, there exists a constant $M > 0$ independent of $n$ such that
		$$
		\|\bm{u}_n\|_{\text{Lip}(K)} = \|\bm{u}_n\|_{L^{\infty}(K; \mathbb{R}^k)} + \|\nabla \bm{u}_n\|_{L^{\infty}(K; \mathbb{R}^{kN})} \leq M.
		$$
	\end{theorem}
	\par 
	Next, we investigate the uniform global Lipschitz regularity of solutions to system \eqref{Sys:main-interior-general}-\eqref{Boundary:main}. As mentioned above, the authors in \cite{Wang-Zhang2010poincare} established the uniform global Lipschitz regularity of solutions within their framework. We point out that their proof required differentiating the system, and thus relied on the differentiability of $\bm{f}_n$. In \cite[Theorem 1.6]{Soave-Z2015ARMA}, the authors claimed that, with the aid of the Caffarelli-Jerison-Kenig monotonicity formula for variable-coefficient systems \cite{Matevosyan-P2011CPAM}, they could prove the uniform global Lipschitz regularity of solutions in their setting via the same approach used for the uniform interior case. However, this approach appears to be unfeasible. On the one hand, the monotonicity formula is of local type; to apply it near the domain boundary, one would need to extend the solutions to a larger domain, which seems impossible for general boundary data $\bm{\varphi}_n$. On the other hand, we lack a critical counterpart to Lemma \ref{lem:decay-Mu-p} for analyzing the boundary behavior of the solutions.  
	To address this problem, we first analyze the boundary behavior of the solutions by leveraging condition \eqref{Condi:boundary-disjoint}, and then employ blow-up analysis to derive the desired result. Notably, the blow-up argument differs from that utilized in the proof of Theorem \ref{thm:interior-lip}. Specifically, in the present setting, the blow-up sequence diverges at the origin; however, we can still prove that the gradient of the sequence remains bounded by applying Theorem \ref{thm:interior-lip}, which yields a contradiction.
	We remark that boundary condition \eqref{Condi:boundary-disjoint} plays a crucial role in our proof: it is not only used to analyze the boundary behavior of the solutions, but also enables the application of Lemma \ref{lem:decay-Mu-p} near the boundary by extending the solutions by zero outside the domain. Note that this condition is also adopted in \cite{Conti-Terracini-Verzini-2005adv, Wang-Zhang2010poincare} and ensures the existence of solutions to system \eqref{Sys:main-interior} in many cases (see, e.g., \cite[Theorem 2.1]{Conti-Terracini-Verzini-2005adv}). Our global result is as follows.
	\begin{theorem}\label{thm:global-lip}
		Let $N \geq 1$, and let $\Omega \subset \mathbb{R}^N$ be a bounded $C^{2}$ domain. Let $k \geq 3$, and suppose that $p_i \geq 1$ and $a_{ij} > 0$ for all $1 \leq i, j \leq k$ with $i \neq j$. For each $n \in \mathbb{N}$, let $\bm{u}_n \in H^1_{\text{loc}}(\Omega; \mathbb{R}^k) \cap C(\overline{\Omega}; \mathbb{R}^k)$ be a weak solution to the system \eqref{Sys:main-interior-general} with boundary condition \eqref{Boundary:main}, where $\bm{f}_n$ and $\bm{\varphi}_n$ satisfy assumptions (H1) and (H2) respectively. If $\{\bm{u}_n\}$ is uniformly bounded in $L^{\infty}(\Omega; \mathbb{R}^k)$, then there exists a constant $M > 0$ independent of $n$ such that
		$$
		\|\bm{u}_n\|_{\text{Lip}(\Omega)} = \|\bm{u}_n\|_{L^{\infty}(\Omega; \mathbb{R}^k)} + \|\nabla\bm{u}_n\|_{L^{\infty}(\Omega; \mathbb{R}^{kN})} \leq M.
		$$
	\end{theorem}
	\par 
	Some remarks are in order:
	\begin{remark}
		(i) As mentioned above, these uniform boundedness results are optimal if $\beta_{ij;n}\to+\infty$ as $n\to\infty$ for all $i\neq j$, since the limit profile of $\{\bm{u}_{n}\}$ is at most Lipschitz continuous in $\Omega$ (by the Hopf lemma). \\
		(ii) As will be seen in the proofs of Theorem \ref{thm:interior-lip} and Theorem \ref{thm:global-lip}, if the coefficient matrix $(a_{ij})$ is replaced by a sequence of matrices $(a_{ij;n})$ admitting positive lower and upper bounds, the conclusions of these theorems remain valid.\\
		(iii) In contrast to the existing optimal results in the literature \cite{Wang-Zhang2010poincare, Soave-Z2015ARMA}, our theorems dispense with the two technical conditions, namely the symmetry (or total symmetry) of the matrix $(a_{ij})$ and the homogeneity of $\bm{\hat{g}}$. Furthermore, we allow different pairs of competition coefficients to have different magnitudes: specifically, some $\beta_{ij;n}$ (with $i \neq j$) tend to infinity as $n \to \infty$, while others remain bounded.\\
		(iv) Even in the case where $\beta_n=\beta_{ij;n}>0$, $p_i = 1$ for all $i\neq j$ and the matrix $(a_{ij})$ is symmetric but not totally symmetric, the conclusion of Theorem \ref{thm:interior-lip} remains new: the authors in \cite{Soave-Z2015ARMA} only considered the case of total symmetry. Similarly, Theorem \ref{thm:global-lip} is also new in this specific case, as the condition imposed on $\bm{f}_{n}$ is significantly weaker than that in \cite{Wang-Zhang2010poincare}.\\
		(v) When $\Omega = \mathbb{R}^N$, under the assumptions stated in Theorem \ref{thm:interior-lip}, one can argue by contradiction, combined with a translation argument, and then apply Theorem \ref{thm:interior-lip} to deduce the existence of a constant $M > 0$ (independent of $n$) such that
		$$\|\bm{u}_{n}\|_{\text{Lip}(\mathbb{R}^N)}=\|\bm{u}_{n}\|_{L^{\infty}(\mathbb{R}^N;\mathbb{R}^k)}+\|\nabla\bm{u}_{n}\|_{L^{\infty}(\mathbb{R}^N;\mathbb{R}^{kN})}\leq M.$$
	\end{remark}
	\begin{remark}
		It is straightforward to see that the regularity assumption on $\partial\Omega$ in Theorem \ref{thm:global-lip} can be weakened to $C^{1,\text{Dini}}$. On the other hand, in \cite{Conti-Terracini-Verzini-2005adv, Wang-Zhang2010poincare, Soave-Z2015ARMA}, the authors imposed the following uniform regularity condition on the boundary data $\bm{\varphi}_n$ instead of condition (H2):
		\begin{align*}
			\|\varphi_{i,n}\|_{\text{Lip}(\partial\Omega)} := \|\varphi_{i,n}\|_{L^{\infty}(\partial\Omega)} + \sup_{\substack{x,y \in \partial\Omega, \\ x \neq y}} \frac{|\varphi_{i,n}(x) - \varphi_{i,n}(y)|}{|x - y|} \leq \bar{M}, \quad \forall 1 \leq i \leq k,
		\end{align*}
		where $\bar{M} > 0$ is a constant independent of $n$. However, this condition is not sufficiently regular to guarantee the global Lipschitz continuity of solutions to system \eqref{Sys:main-interior}, even in the harmonic setting. For a concrete counterexample involving harmonic functions in the upper half-space of $\mathbb{R}^2$, we refer the reader to \cite{Li-Wang2009JDE}; below, we provide an alternative counterexample on the unit ball for the reader's convenience. Let $N \geq 2$, fix a point $x_0 \in \partial B_1$, and define the boundary datum $\varphi(y) := |y - x_0|$ for all $y \in \partial B_1$. It is straightforward to verify that $\varphi \in \text{Lip}(\partial B_1)$. Let $u \in C^2(B_1) \cap C(\overline{B_1})$ denote the solution to the Dirichlet problem
		\[
		-\Delta u = 0 \quad \text{in } B_1, \qquad u = \varphi \quad \text{on } \partial B_1.
		\]
		We claim that $u \notin \text{Lip}(\overline{B_1})$. To prove this, assume without loss of generality that $x_0 = \vec{e}_1 := (1, 0, \cdots, 0)$. For $t \in (0, 1)$, let $x_t := (t, 0, \cdots, 0) \in B_1$ (note that $x_t$ converges to $x_0$ as $t \to 1^-$). By the classical Poisson integral representation of harmonic functions on $B_1$, we then deduce that
		\begin{align*}
			\limsup_{t \to 1^-} \frac{|u(x_t) - u(x_0)|}{|x_t - x_0|} \geq C(N) \limsup_{t \to 1^-} \int_{\partial B_1 \cap B_{1-t}(x_0)^c} |y - x_0|^{1 - N} d\sigma_y = +\infty,
		\end{align*}
		where $C(N) > 0$ is a constant depending only on $N$. This divergence directly implies that $u \notin \text{Lip}(\overline{B_1})$. 
	\end{remark}
	\par 
	Finally, as applications of the uniform optimal regularity results established in Theorem \ref{thm:interior-lip}, we investigate sharp quantitative pointwise estimates and the asymptotic behavior of the densities, particularly near the interfaces between distinct components.
	\par 
	First, we establish the following sharp pointwise estimates for solutions to system \eqref{Sys:main-interior-general}. An analogous result for the variational Gross-Pitaevskii system has been obtained in \cite{Soave-Z2017Poincare, Zhang-Zhang2025}.
	\begin{proposition}\label{prop:uniform-upper-bdd}
		Under the assumptions of Theorem \ref{thm:interior-lip}, for any compact set $K \Subset \Omega$, there exists a constant $C > 0$ (independent of $n$) such that 
		$$\beta_{i,j;n} u_{i,n}^{p_i + 1} u_{j,n}^{p_j} \leq C \quad \text{in } K, \quad \forall \, i \neq j.$$
	\end{proposition}
	\par 
	Next, following \cite{Conti-Terracini-Verzini-2005indiana,Conti-Terracini-Verzini-2005adv,Terracini-V-Z2019CPAM}, we consider the function class $\mathscr{S}(\Omega)$ defined by 
	\begin{align}
		\mathscr{S}(\Omega):=
		\left\{
		\bm{u}\in H^1_{\text{loc}}(\Omega;\mathbb{R}^k)\cap C_{\text{loc}}(\Omega;\mathbb{R}^k):
		\begin{array}{l}
			\bm{u}\not\equiv 0,~~u_i\geq 0~~\text{in}~~\Omega,~~\forall i;\\
			-\Delta u_i\leq 0,~~-\Delta \widehat{u}_i\geq 0~~\text{in}~~\Omega,~~\forall i;\\
			u_i\cdot u_j\equiv 0~~\text{in}~~\Omega,~~\forall i\neq j.
		\end{array}\right\},
	\end{align}
	where $N \geq 1$, $k\geq 3$, $a_{ij}>0$ for any $i\neq j$, $\Omega$ is a domain in $\mathbb{R}^N$, and the hat operator is defined for $k$-component vector functions $\bm{v} = (v_1, \dots, v_k)$ by $\widehat{\bm{v}} = (\widehat{v}_1, \dots, \widehat{v}_k)$, where 
	\begin{align}\label{def:hat-operator}
		\widehat{v}_i:=v_i-\sum_{j\neq i}\frac{a_{ij}}{a_{ji}}v_j,\quad 1\leq i\leq k.
	\end{align}
	\par 
	Given $\bm{u}\in \mathscr{S}(\Omega)$, the interface (or free boundary) of $\bm{u}$ is defined by 
	\begin{align}
		\Gamma_{\bm{u}}:=\{x\in \Omega:u_i(x)=0,~\forall 1\leq i\leq k\}.
	\end{align}
	Moreover, for any $x \in \Omega$, the multiplicity of $x$ with respect to $\bm{u}$ is defined as
	\begin{align}
		m_{\bm{u}}(x) := \#\left\{ 1 \leq i \leq k : \left| \{ u_i > 0 \} \cap B_r(x) \right| > 0 \text{ for all } r > 0 \right\},
	\end{align}
	where $\#A$ denotes the cardinality of a set $A$, and $|B|$ stands for the Lebesgue measure of a set $B$ in $\mathbb{R}^N$.
	Furthermore, we denote by $\mathcal{L}_{s,\bm{u}}$ the set consisting of all points in $\Gamma_{\bm{u}}$ with multiplicity $0\leq s\leq k$, that is,
	\begin{align}
		\mathcal{L}_{s,\bm{u}}:=\{x\in \Gamma_{\bm{u}}:m_{\bm{u}}(x)=s\}.
	\end{align}
	\par 
	It is straightforward to verify that if $\{\bm{u}_n\}$ is a sequence of weak solutions to system \eqref{Sys:main-interior} where $\beta_n \to +\infty$ and $\bm{f}_n \to 0$ in $L^{\infty}_{\text{loc}}(\Omega; \mathbb{R}^k)$ as $n \to \infty$, and if $\bm{u}_n \to \bm{u}$ in $H^1_{\text{loc}}(\Omega; \mathbb{R}^k) \cap C_{\text{loc}}(\Omega; \mathbb{R}^k)$ as $n \to \infty$, then $\bm{u} \in \mathscr{S}(\Omega)$. Therefore, for simplicity, we focus on the case where $\bm{f}_n\equiv 0$ in $\Omega$ in the following.
	\par 
	When the matrix $(a_{ij})$ is symmetric, the structure of the free boundary $\Gamma_{\bm{u}}$ for $\bm{u} \in \mathscr{S}(\Omega)$ has been studied in \cite{Caffarelli-K-L2009fixpoint, Conti-Terracini-Verzini-2005indiana, Tavares-T2012CVPDE}. It was shown that the Almgren monotonicity formula, a powerful tool, can be applied to deduce that $\Gamma_{\bm{u}}$ is a collection of smooth hypersurfaces (regular subsets) up to a residual set with small Hausdorff dimension. However, if $(a_{ij})$ is asymmetric, this formula is unavailable, and such a result was only derived recently by Terracini et al. \cite{Terracini-V-Z2019CPAM} for $N=2$ under suitable boundary conditions. We analyze the one-dimensional case in the Appendix, while the problem remains open in higher dimensions. In \cite{Caffarelli-K-L2009fixpoint}, the authors established the so-called Clean-Up Lemma, which asserts that if $N \geq 2$ and $x_0 \in \mathcal{R}_{\bm{u}}$, then $m_{\bm{u}}(x_0) = 2$, where 
	\[
	\mathcal{R}_{\bm{u}} := \left\{ x \in \Gamma_{\bm{u}} : \lim_{r \to 0} J_{ij}(\bm{u}, x, r) > 0 \text{ for some } i \neq j \right\},
	\]
	and 
	\begin{align}\label{def:J-ij}
		J_{ij}(\bm{u}, x, r) := \frac{1}{r^4} \int_{B_r} \frac{|\nabla u_i(y)|^2}{|y - x|^{2 - N}} \, dy \cdot \int_{B_r} \frac{|\nabla u_j(y)|^2}{|y - x|^{2 - N}} \, dy.
	\end{align}
	We point out that although they considered only the symmetric case, their proof can be slightly modified to yield the same result for the asymmetric case. This powerful Clean-Up Lemma was later adapted by Dancer et al. \cite{Dancer-WZ2012Trans} to solutions $\{\bm{u}_n\}$ of system \eqref{Sys:main-interior} with $\bm{f}_n \equiv 0$ in $\Omega$ and $p = p_i = 1$ for all $i$. They showed that near a regular point $x_0 \in \mathcal{R}_{\bm{u}}$, $u_{j,n}$ exhibits an exponential decay for all $j \neq i_1, i_2$ as $n \to \infty$, where $i_1$ and $i_2$ are the only two indices such that $u_{i_1}, u_{i_2} \not\equiv 0$ in a neighborhood of $x_0$. They also obtained the same result for $N = 1$. Using a similar argument, we prove analogous results for general $p_i \geq 1$ with $p=\max_{1\leq i\leq k}\{p_i\}>1$, after establishing a strengthened version of Lemma \ref{lem:decay-Mu-p} (see Lemma \ref{lem:decay-Mu-p-general}). Notably, in our setting, if $j\neq i_1,i_2$ and $p_j>1$, then the decay behavior of $u_{j,n}$ is polynomial rather than exponential.
	\par 
	\begin{theorem}\label{thm:ab-near-regular}
		Let $N \geq 1$ and $\Omega$ be a domain in $\mathbb{R}^N$. Let $k \geq 3$, $p_i\geq 1$, and $a_{ij} > 0$ for all $1 \leq i, j \leq k$ with $i \neq j$. Let $\{\bm{u}_n\} \subseteq H^1_{\text{loc}}(\Omega; \mathbb{R}^k) \cap C_{\text{loc}}(\Omega; \mathbb{R}^k)$ be a sequence of weak solutions to system \eqref{Sys:main-interior}-\eqref{Condi:Bound-u}, where $\beta_n \to +\infty$ as $n \to \infty$ and $\bm{f}_n \equiv 0$ in $\Omega$ for all $n \geq 1$. Assume that $p := \max_{1 \leq i \leq k} \{p_i\} > 1$, and assume further that $\bm{u}_n \to \bm{u} \not\equiv 0$ in $C_{\text{loc}}(\Omega; \mathbb{R}^k)$ as $n \to \infty$. Let $x_0 \in \mathcal{R}_{\bm{u}}$, and let $i_1, i_2$ be the only two indices for which $u_{i_1}, u_{i_2} \not\equiv 0$ in a neighborhood of $x_0$. Then there exists a constant $r > 0$ independent of $n$ such that for $n$ large, we have:
		\begin{itemize}
			\item  for any $j\neq i_1,i_2$, if $p_j>1$, then $u_{j,n}$ admits a polynomial decay in $B_r(x_0)$ in the sense that there exists a constant $C>0$ independent of $n$ such that 
			$$u_{j,n} \leq C \beta_n^{-\left( \frac{1}{p_j} + \frac{1}{2p_j(p_j - 1)} \right)} \quad \text{in } B_r(x_0),$$
			whereas if $p_j=1$, then $u_{j,n}$ admits an exponential decay in $B_r(x_0)$ in the sense that there exists  a constant $C>0$ independent of $n$ such that 
			$$u_{j,n}\leq Ce^{-\beta_n^{\frac{1}{16p(p+2)}}}\quad \text{in }B_r(x_0);$$
			\item in $B_r(x_0)$, the system reduces to 
			\begin{align*}
				\left\{
				\begin{aligned}
					-\Delta u_{i_1,n}&=-\beta_na_{i_1i_2}u_{i_1,n}^{p_{i_1}}u_{i_2,n}^{p_{i_2}}+u_{i_1,n}^{p_{i_1}}\sum_{p_j>1,\atop j \neq i_1,i_2}O(\beta_n^{-\frac{1}{2(p_j-1)}})+u_{i_1,n}^{p_{i_1}}\sum_{p_j=1,\atop j \neq i_1,i_2}O(\beta_ne^{-\beta_n^{\frac{1}{16p(p+2)}}}),\\
					-\Delta u_{i_2,n}&=-\beta_na_{i_2i_1}u_{i_1,n}^{p_{i_1}}u_{i_2,n}^{p_{i_2}}+u_{i_2,n}^{p_{i_2}}\sum_{p_j>1,\atop j \neq i_1,i_2}O(\beta_n^{-\frac{1}{2(p_j-1)}})+u_{i_2,n}^{p_{i_2}}\sum_{p_j=1,\atop j \neq i_1,i_2}O(\beta_ne^{-\beta_n^{\frac{1}{16p(p+2)}}}).
				\end{aligned}\right.
			\end{align*}
		\end{itemize}
		As a consequence, up to a subsequence, for any $j \neq i_1, i_2$ and any $\alpha \in (0,1)$, we have
		\[
		a_{i_2i_1}u_{i_1,n} - a_{i_1i_2}u_{i_2,n} \to a_{i_2i_1}u_{i_1} - a_{i_1i_2}u_{i_2} \quad \text{and} \quad u_{j,n} \to 0
		\]
		in $C^{1,\alpha}_{\text{loc}}(B_r(x_0))$ as $n \to \infty$.
	\end{theorem}
	\par 
	As mentioned above, Theorem \ref{thm:ab-near-regular} holds for the case where $p_i = 1$ for all $i$; moreover, the conditions imposed on $\bm{f}_n$ in this theorem appear to be weakenable, which we do not pursue here (see \cite[Section 5]{Dancer-WZ2012Trans}).
	\par 
	Thirdly, we analyze the asymptotic behavior of the solutions $\bm{u}_n$ near the interface $\Gamma_n$ of $\bm{u}_n$, where
	\begin{align}\label{def:gamma-n}
		\Gamma_n := \left\{ x \in \Omega :
		\begin{aligned}
			&u_{i,n}(x) = u_{j,n}(x) \text{ for some } i \neq j,\\
			&[u_{i,n}(x)+u_{j,n}(x)]^{p_i+p_j+1}\geq [u_{l,n}(x)+u_{s,n}(x)]^{p_l+p_s+1} \text{ for all } l\neq s.
		\end{aligned} \right\}.
	\end{align}
	\begin{theorem}\label{thm:ap-near-gamma-n}
		Let $N \geq 1$ and $\Omega$ be a domain in $\mathbb{R}^N$. Let $k \geq 3$, $p_i\geq 1$, and $a_{ij} > 0$ for all $1 \leq i, j \leq k$ with $i \neq j$. Let $\{\bm{u}_n\} \subseteq H^1_{\text{loc}}(\Omega; \mathbb{R}^k) \cap C_{\text{loc}}(\Omega; \mathbb{R}^k)$ be a sequence of weak solutions to system \eqref{Sys:main-interior}-\eqref{Condi:Bound-u}, where $\beta_n \to +\infty$ as $n \to \infty$ and $\bm{f}_n \equiv 0$ in $\Omega$ for all $n \geq 1$. Assume that $\bm{u}_n \to \bm{u} \not\equiv 0$ in $C_{\text{loc}}(\Omega; \mathbb{R}^k)$ as $n \to \infty$. Let $x_n\in \Gamma_n$ and $x_n\to x_0\in \Omega$ as $n\to\infty$. Then $x_0\in \Gamma_{\bm{u}}$ and the following assertions hold:\\
		(1) If $x_0\in \mathcal{R}_{\bm{u}}$, then 
		$$\limsup_{n \to \infty}\beta_n\sum_{i\neq j}u_{i,n}(x_n)^{p_i+1}u_{j,n}(x_n)^{p_j}\in (0,+\infty).$$ 
		(2) If $x_0\in \Gamma_{\bm{u}}\setminus\mathcal{R}_{\bm{u}}$, then 
		$$\limsup_{n \to \infty}\beta_n\sum_{i\neq j}u_{i,n}(x_n)^{p_i+1}u_{j,n}(x_n)^{p_j}=0.$$
	\end{theorem}
	\par 
	In view of Theorem \ref{thm:ap-near-gamma-n}, we conclude that the pointwise estimate in Proposition \ref{prop:uniform-upper-bdd} is sharp.
	\par 
	Theorem \ref{thm:ap-near-gamma-n} (1) follows from Theorem \ref{thm:ab-near-regular} and Proposition \ref{prop:uniform-upper-bdd}, while part (2) of the theorem is established via Theorem \ref{thm:interior-lip} and the Alt-Caffarelli-Friedman type monotonicity formula (Theorem \ref{thm:ACF-formula}). Analogous results for the Gross-Pitaevskii system have been proved in \cite[Theorem 1.5]{Soave-Z2017Poincare} and \cite[Corollary 1.3]{Wang-kelei2017CVPDE}. Note that if $\bm{\hat{g}}$ is homogeneous, then
	$$\Gamma_n = \left\{ x \in \Omega : u_{i,n}(x) = u_{j,n}(x) \text{ for some } i \neq j, \text{ and } u_{i,n}(x) \geq u_{l,n}(x) \text{ for all } 1 \leq l \leq k \right\}.$$
	\par 
	Finally, we prove that if $p_i = 1$ for all $i$ and $x_0 \in \Gamma_{\bm{u}}$ with $m(x_0) \neq 0$, then there exists a sequence $\{x_n\}$ such that $x_n \in \Gamma_n$ for all $n \geq 1$ and $x_n \to x_0$ in $\Omega$ as $n \to \infty$. 
	\begin{proposition}\label{prop:existence-xn}
		Let $\{\bm{u}_n\}$ and $\bm{u}$ be as in Theorem \ref{thm:ap-near-gamma-n}, with $p_i = 1$ for all $1 \leq i \leq k$. Let $x_0 \in \Gamma_{\bm{u}}$ satisfy $m(x_0) \neq 0$. Then there exists a sequence $\{x_n\}$ with $x_n \in \Gamma_n$ for all $n \geq 1$ such that $x_n \to x_0$ in $\Omega$ as $n \to \infty$.
	\end{proposition}
	\par 
	This result is a straightforward application of Lemma \ref{lem:decay-Mu-p-general}. However, we remark that this conclusion is non-trivial in higher dimensions, as we do not know whether the singular subset $\Gamma_{\bm{u}} \setminus \mathcal{R}_{\bm{u}}$ of $\Gamma_{\bm{u}}$ has Hausdorff dimension at most $N-2$. Indeed, if this property holds, then noting that $\mathcal{L}_{1,\bm{u}} = \emptyset$ (see Lemma \ref{lem:multiplicity-one=empty}), we can follow the proof of \cite[Proposition 1.3]{Soave-Z2017Poincare} to establish the desired result. Specifically, if the matrix $(a_{ij})$ is symmetric, this property is known to hold, and hence Proposition \ref{prop:existence-xn} follows immediately.
	\par 
	The paper is organized as follows: In Section \ref{sec:2}, we introduce some notation and derive several properties of subsolutions to system \eqref{Sys:main-interior-general}. In Section \ref{sec:3}, we establish an Alt-Caffarelli-Friedman type monotonicity formula for system \eqref{Sys:main-interior-general} (see Theorem \ref{thm:ACF-formula}). Section \ref{sec:4} is devoted to proving Theorem \ref{thm:interior-lip}, using blow-up analysis and the Alt-Caffarelli-Friedman type monotonicity formula. In particular, we employ an induction argument to overcome the challenge arising from the nonhomogeneity of the competition term $\bm{\hat{g}}$. Theorem \ref{thm:global-lip} is proved in Section \ref{sec:5} via blow-up analysis and Theorem \ref{thm:interior-lip}. In Section \ref{sec:6}, we prove Proposition \ref{prop:uniform-upper-bdd}, Theorems \ref{thm:ab-near-regular}, \ref{thm:ap-near-gamma-n}, and Proposition \ref{prop:existence-xn}. Finally, in the appendix, we analyze the structure of free boundaries for elements of $\mathscr{S}(\Omega)$ in one dimension.
	\par 
	Throughout the paper, we use $C, C_0, C_1, C_2$ and $c_0, c_1, c_2$ to denote positive constants, which may differ from line to line.

	\section{Preliminaries}\label{sec:2}
	\setcounter{section}{2}
	\setcounter{equation}{0}
	In this section, we introduce some notation and deduce some properties of subsolutions to system \eqref{Sys:main-interior-general}, which will be used mainly in the proof of Theorem \ref{thm:interior-lip}. We also recall a useful lemma which is obtained in \cite{Soave-Z2015ARMA}. 
	\par 
	Let $N\geq 3$, $\Omega'$ be a domain in $\mathbb{R}^N$ and $K\Subset\Omega'$ be a non-empty compact set. For a function $u\in C_{\text{loc}}(\Omega')$, we denote 
	\begin{align}\label{def:widetilde-H}
		\widetilde{H}(u,x_0,r):=\frac{1}{r^{N-1}}\int_{\partial B_r(x_0)}u,\quad \text{for}~~x_0\in K,~~0<r<\text{dist}(K,\partial\Omega').
	\end{align}
	\par 
	First of all, we derive a monotonic property of  $	\widetilde{H}(u,x_0,\cdot)$.
	\begin{lemma}\label{lem:non-decrease-widetilde-H}
		Let $d'\geq 0$. Suppose that $u\in C^1_{\text{loc}}(\Omega')$ satisfies  $-\Delta u\leq d'$ in $\Omega'$ in the distribution sense. Then the function
		$$r\mapsto \widetilde{H}(u,x_0,r)+|B_1|d'\frac{r^2}{2}$$
		is monotone non-decreasing for $r\in (0,\text{dist}(K,\partial\Omega'))$ and $x_0\in K$. Here $|B_1|$ denotes the Lebesgue measure of $B_1$ in $\mathbb{R}^N$.
	\end{lemma}
	\begin{proof}
		Let
		$$0<r<r'<\text{dist}(K,\partial\Omega'),\quad x_0\in K,\quad  \bar{u}(x):=u(x_0+rx)~~\text{for}~~x\in B_{r'/r}.$$
		Then $\bar{u}\in C^1_{\text{loc}}(B_{r'/r})$ and $-\Delta \bar{u}\leq r^2d'$ in $B_{r'/r}$.
		By integrating in $B_1$, we deduce that $\int_{\partial B_1}\nabla \bar{u}(x)\cdot x\geq -|B_1|r^2d'$.
		This implies that 
		$$\frac{d}{dr}\widetilde{H}(u,x_0,r)=\frac{d}{dr}\left(\int_{\partial B_1}\bar{u}\right)=\int_{\partial B_1}(\nabla u)(x_0+rx)\cdot x\geq -|B_1|rd'.$$
		By integrating, we obtain the desired result.
	\end{proof}
	\par 
	Next, we consider a vector-valued function $\bm{u}\in C_{\text{loc}}^1(\Omega';\mathbb{R}^k)$ that solves
	\begin{align}\label{Sys:main-interior-1}
		\left\{
		\begin{array}{ll}
			-\Delta u_{i}\leq d'-u_{i}^{p_i}\sum\limits_{j\neq i}^ka_{ij}M_{ij}u_{j}^{p_{j}}\quad&\text{in}~~\Omega',\\
			u_{i}\geq 0\quad&\text{in}~~\Omega',
		\end{array}\right.
		\qquad 1\leq i\leq k
	\end{align}
	in the distribution sense, where $k\geq 2$, $d'>0$, $\beta>0$, $p_i\geq 1$, $a_{ij}>0$ and $M_{ij}>0$ for any $1\leq i,j\leq k$ with $i\neq j$. 
	\par
	For any $x_0\in K$, $1\leq i\leq k$, and $0<r<\text{dist}(K,\partial \Omega')$, we denote
	\begin{align}\label{def:H-J-Lambda}
		\begin{split}
			H_i(\bm{u},x_0,r)&:=\frac{1}{r^{N-1}}\int_{\partial B_r(x_0)}u_i^2,\\
			J_i(\bm{u},x_0,r)&:=\int_{B_r(x_0)}\left[|\nabla u_i|^2+u_i^{p_i+1}\sum_{j\neq i}^ka_{ij}M_{ij}u_j^{p_{j}}-d'u_i\right]|x-x_0|^{2-N},\\
			\Lambda_i(\bm{u},x_0,r)&:=\frac{r^2\int_{\partial B_r(x_0)}|\nabla_{\theta} u_i|^2+ u_i^{p_i+1}\sum_{j\neq i}^ka_{ij}M_{ij}u_j^{p_{j}}-d'u_i}{\int_{\partial B_r(x_0)}u_i^2}\quad \text{if}~~\int_{\partial B_r(x_0)}u_i^2\neq 0.
		\end{split}
	\end{align}
	Here $|\nabla_{\theta} u_i|^2:=|\nabla u_i|^2-|\partial_{\nu}u_i|^2$, where $\partial_{\nu}u(x)=\nabla u(x)\cdot \nu(x)$ and $\nu(x)$ is the unit outer normal vector at $x\in \partial B_r(x_0)$.  We also introduce the following function
	\begin{align}\label{def:gamma-function}
		\gamma(t):=\sqrt{\left(\frac{N-2}{2}\right)^2+t}-\frac{N-2}{2}\quad \text{for}~~ t\geq -\left(\frac{N-2}{2}\right)^2.
	\end{align}
	Then we have the following estimates for $J_i(1\leq i\leq k)$.  
	\begin{lemma}\label{lem:est:J-i(r)}
		Let $1\leq i\leq k$, $x_0\in K$ and $0<r<\text{dist}(K,\partial\Omega')$. Then the following assertions hold:\\
		(1) There holds
		$$J_i(\bm{u},x_0,r)\leq \frac{1}{r^{N-2}}\int_{\partial B_{r}(x_0)}u_i\partial_{\nu}u_i+\frac{N-2}{2r^{N-1}}\int_{\partial B_{r}(x_0)}u_i^2.$$
		(2) If there exists $L>0$ such that $\|\nabla u_i\|_{L^{\infty}(\Omega')}\leq L$, then
		$$\frac{J_i(\bm{u},x_0,r)}{r^2}\leq L\sigma_{N-1}^{1/2}\left[\frac{H_i(\bm{u},x_0,r)}{r^2}\right]^{1/2}+\frac{N-2}{2}\frac{H_i(\bm{u},x_0,r)}{r^2},$$
		where $\sigma_{N-1}$ denotes the area of the unit sphere in $\mathbb{R}^N$.\\
		(3) If $\Lambda_i(\bm{u},x_0,r)>0$, then
		$$J_i(\bm{u},x_0,r)\leq \frac{r}{2\gamma(\Lambda_i(\bm{u},x_0,r))}\int_{\partial B_r(x_0)}\left[|\nabla u_i|^2+u_i^{p_i+1}\sum_{j\neq i}a_{ij}M_{ij}u_j^{p_{j}}-d'u_i\right]|x-x_0|^{2-N}.$$
	\end{lemma}
	\begin{proof}
		We omit the proof as it is almost identical to that of Lemma 2.4 in \cite{Zhang-Zhang2025}.
	\end{proof}
	\par 
	Thirdly, we derive some asymptotic properties of solutions with uniform Lipschitz bounds to system \eqref{Sys:main-interior-1}. More precisely, consider a sequence of weak solutions $\{\bm{u}_n\} \subseteq C^{1}_{\text{loc}}(B_2; \mathbb{R}^k)$ to the following system:
	\begin{align}\label{Sys:main-interior-2}
		\left\{
		\begin{array}{ll}
			-\Delta u_{i,n}\leq d_{i,n}'- u_{i,n}^{p_i}\sum\limits_{j\neq i}^ka_{ij}M_{i,j;n}u_{j,n}^{p_{j}}\quad&\text{in}~~B_2,\\
			u_{i,n}\geq 0\quad&\text{in}~~B_2,
		\end{array}\right.
	\end{align}
	where $N\geq 3$, $1\leq i\leq k$ with $k\geq 2$, $p_i\geq 1$, $a_{ij}>0$, $d_{i,n}'\geq 0$ and $M_{i,j;n}>0$ for any $j\neq i$. Suppose that there exist $m'>0$ and $d'>0$ independent of $n$ such that
	\begin{align}\label{bdd:u-n-w-1-infty}
		\|u_{i,n}\|_{W^{1,\infty}(B_2)}\leq m',\quad d_{i,n}'\leq d',\quad\forall i=1,2,\dots,k.
	\end{align}
	\par 
	We have the following properties of $\{\bm{u}_n\}$.
	\begin{lemma}\label{lem:property-global-lip-sol}
		Under the previous notation, up to a subsequence, the following assertions hold:\\
		(1) There exists a globally Lipschitz function $\bm{u}$ in $ B_2$ such that $\bm{u}_{n}\to \bm{u}$ in $C_{\text{loc}}(B_2;\mathbb{R}^k)$ as $n\to\infty$.\\
		(2) There holds $\bm{u}_{n}\to \bm{u}$ in $H^1_{\text{loc}}(B_2;\mathbb{R}^k)$ as $n\to\infty$, and for any $0<r<2$, there exists $C>0$ independent of $n$ such that
		\begin{align}\label{eq:property-bdd-betau-i-u-j-2}
			\int_{B_r}u_{i,n}^{p_i}\sum\limits_{j\neq i}^ka_{ij}M_{i,j;n}u_{j,n}^{p_{j}}\leq C,\quad \forall 1\leq i\leq k,
		\end{align}
		in particular, $u_i\cdot u_j\equiv 0$ in $B_2$ if $M_{i,j;n}\to +\infty$ as $n\to\infty$.\\
		(3) Suppose that there exists $C>0$ independent of $n$ such that $M_{i,j;n}\geq C$ and $M_{j,i;n}\geq C$ for some $i\neq j$. If moreover
		\begin{align}\label{cond:H-in-g-in}
			H_i(\bm{u}_{n},0,1)\geq 1/m',\quad H_j(\bm{u}_{n},0,1)\geq 1/m'\quad\forall n\geq 1,
		\end{align}
		then there exists $\varepsilon_0>0$ independent of $n$ such that for any $n\geq 1$ and $s\in \{i,j\}$, one has
		$$\int_{B_1}\left[|\nabla u_{s,n}|^2+u_{s,n}^{p_s+1}\sum_{l\neq s}^ka_{sl}M_{s,l;n}u_{l,n}^{p_{l}}\right]|x|^{2-N}\geq \varepsilon_0,$$
		and 
		$$\frac{\int_{\partial B_1}|\nabla_{\theta} u_{s,n}|^2+ u_{s,n}^{p_s+1}\sum_{l\neq s}^ka_{sl}M_{s,l;n}u_{l,n}^{p_{l}}}{\int_{\partial B_1}u_{s,n}^2}\geq \varepsilon_0.$$
	\end{lemma}
	\begin{proof}
		Point (1) follows by the Ascoli-Arzel\`{a} theorem. Point (2) can be deduced via a simple calculation; one can argue in a similar way as in \cite[Lemma 3.1]{Zhang-Zhang2025} to obtain the desired result. 
		\par 
		It remains to prove Point (3). By Point (1) and \eqref{cond:H-in-g-in}, we deduce that $H_i(\bm{u},0,1)\geq \frac{1}{m'}$ and $H_j(\bm{u},0,1)\geq \frac{1}{m'}$, which implies $u_i\not\equiv 0$ and $u_j\not\equiv 0$ on $\partial B_1$.
		To prove Point (3), we divide the proof into two cases.
		\par 
		\textbf{Case 1. }$u_i\cdot u_j\not\equiv 0$ on $\partial B_1$. In this case, by continuity, there exists $x_0\in \partial B_1$ and $0<r<1$ such that $u_i>0$ and $u_j>0$ in $B_r(x_0)$. Then for $s\in \{i,j\}$, one has by Fatou Lemma, 
		$$\liminf_{n\to \infty}\int_{B_1}\left[|\nabla u_{s,n}|^2+u_{s,n}^{p_s+1}\sum_{l\neq s}^ka_{sl}M_{s,l;n}u_{l,n}^{p_{l}}\right]|x|^{2-N}\geq C\int_{B_1}u_{s}^{p_s+1}\sum_{l\neq s,l\in \{i,j\}}^ka_{sl}u_{l}^{p_{l}}>0,$$
		and
		\begin{align*}
			\liminf_{n\to \infty}\frac{\int_{\partial B_1}|\nabla_{\theta} u_{s,n}|^2+ u_{s,n}^{p_s+1}\sum_{l\neq s}^ka_{sl}M_{s,l;n}u_{l,n}^{p_{l}}}{\int_{\partial B_1}u_{s,n}^2}
			\geq &\liminf_{n\to \infty}\frac{C\int_{\partial B_1}u_{s,n}^{p_s+1}\sum_{l\neq s,l\in \{i,j\}}^ka_{sl}u_{l,n}^{p_{l}}}{\int_{\partial B_1}u_{s,n}^2}\\
			&\geq \frac{C\int_{\partial B_1}u_{s}^{p_s+1}\sum_{l\neq s,l\in \{i,j\}}^ka_{sl}u_{l}^{p_{l}}}{\int_{\partial B_1}u_{s}^2}>0.
		\end{align*}
		\par 
		\textbf{Case 2. }$u_i\cdot u_j\equiv 0$ on $\partial B_1$. In this case, we have that $u_i$ and $u_j$ are non-trivial and non-constant in $B_1$. Hence, for $s\in \{i,j\}$, one has
		$$\liminf_{n\to \infty}\int_{B_1}\left[|\nabla u_{s,n}|^2+u_{s,n}^{p_s+1}\sum_{l\neq s}^ka_{sl}M_{s,l;n}u_{l,n}^{p_{l}}\right]|x|^{2-N}\geq \int_{B_1}|\nabla u_s|^2>0.$$
		On the other hand, since $u_i$ and $u_j$ are non-trivial on $\partial B_1=\mathbb{S}^{N-1}$, we deduce that $0<\mathcal{H}^{N-1}(E_s)<\sigma_{N-1}=\mathcal{H}^{N-1}(\mathbb{S}^{N-1})$, where $s\in \{i,j\}$ and $E_s:=\{x\in \mathbb{S}^{N-1}:u_s(x)>0\}$. Here, $\mathcal{H}^{N-1}$ denotes the $(N-1)$-dimensional Hausdorff measure in $\mathbb{R}^N$. This implies that 
		\begin{align}\label{eq:lower-bdd-on-sphere}
			\frac{\int_{\partial B_1}|\nabla_{\theta} u_{s}|^2}{\int_{\partial B_1}u_{s}^2}\geq \inf_{v\in H^1_0(E_s)}\frac{\int_{\partial B_1}|\nabla_{\theta} v|^2}{\int_{\partial B_1}v^2}>0,\quad \forall s\in \{i,j\}.
		\end{align}
		Since $\bm{u}_n\in  C^{1}_{\text{loc}}(B_2;\mathbb{R}^k)$, by \eqref{bdd:u-n-w-1-infty}, we see that  $\{\bm{u}_n\}$ is bounded in $H^1(\mathbb{S}^{N-1};\mathbb{R}^k)$.  By the Sobolev compact embedding theorem, up to a subsequence, one has for any $1\leq i\leq k$,
		$$u_{i,n}\rightharpoonup u_i\quad \text{weakly  in}~~H^1(\mathbb{S}^{N-1}),\quad u_{i,n}\to  u_i\quad \text{strongly  in}~~L^2(\mathbb{S}^{N-1})\quad \text{as}~~n\to\infty.$$
		Combining this with \eqref{eq:lower-bdd-on-sphere} yields
		\begin{align*}
			&\liminf_{n\to \infty}\frac{\int_{\partial B_1}|\nabla_{\theta} u_{s,n}|^2+ u_{s,n}^{p_s+1}\sum_{l\neq s}^ka_{sl}M_{s,l;n}u_{l,n}^{p_{l}}}{\int_{\partial B_1}u_{s,n}^2}\\
			\geq &\liminf_{n\to \infty}\frac{\int_{\partial B_1}|\nabla_{\theta} u_{s,n}|^2}{\int_{\partial B_1}u_{s,n}^2}
			\geq \frac{\int_{\partial B_1}|\nabla_{\theta} u_{s}|^2}{\int_{\partial B_1}u_{s}^2}>0,\qquad \forall s\in \{i,j\}.
		\end{align*}
		The proof is finished.
	\end{proof}
	\par 
	\par 
	We end this section by recalling the following useful lemma which is abtained in \cite[Lemma 2.2]{Soave-Z2015ARMA}.
	\begin{lemma}[\cite{Soave-Z2015ARMA}]\label{lem:decay-Mu-p}
		Let $N\geq 1$, $x_0\in \mathbb{R}^N$ and $A,M,\delta,\rho>0$. Let  $u\in H^1(B_{2\rho}(x_0))\cap C(\overline{B_{2\rho}(x_0)})$ be a non-negative subsolution to 
		\begin{align*}
			\left\{
			\begin{array}{ll}
				-\Delta u\leq -Mu^p+\delta \quad&\text{in}~~B_{2\rho}(x_0),\\
				u\leq A\quad&\text{on}~~\partial B_{2\rho}(x_0)
			\end{array}\right.
		\end{align*}
		for some $p\geq 1$. Then there exists a constant $C>0$ depending only on $N$ such that 
		$$Mu(x)^p\leq \frac{CA}{\rho^2}+\delta,\quad \forall x\in B_{\rho}(x_0).$$
	\end{lemma}
	
	\section{An Alt-Caffarelli-Friedman type monotonicity formula}\label{sec:3}
	\setcounter{section}{3}
	\setcounter{equation}{0}
	In this section, we establish an Alt-Caffarelli-Friedman type monotonicity formula for system \eqref{Sys:main-interior-1}. To our knowledge, this perturbed Alt-Caffarelli-Friedman monotonicity formula is new, and it is a key tool for us to estimate the uniform interior Lipschitz bounds of solutions to system \eqref{Sys:main-interior-general}. The primary difficulty arises from the non-variational structure of the system. To address this issue, we establish a non-variational Poincaré-type lemma on $\mathbb{S}^{N-1}$, leveraging suitable permutation arguments, the Moser iteration technique, and the classical variational Poincaré-type lemma (see Lemma \ref{lem:poincare-sphere} below).
	\par 
	The main result in this section is as follows. Recall that the functions $H_i,\Lambda_i$ and $J_i$ are defined in \eqref{def:H-J-Lambda}, $\sigma_{N-1}=\mathcal{H}^{N-1}(\mathbb{S}^{N-1})>0$ and $\kappa=\min\{a_{ij}:1\leq i,j\leq k~\text{with}~i\neq j\}>0$.
	\begin{theorem}\label{thm:ACF-formula}
		Let  $N\geq 3$, $\Omega'$ be a domain in $\mathbb{R}^N$, $k\geq 2$, $d'>0$, $p_i\geq 1$, $a_{ij}>0$, $M_{ij}=M_{ji}>0$ for any $1\leq i,j\leq k$ with $i\neq j$ and $K\Subset\Omega'$ be a compact set. Let $\mu>0$, $R>2$, $x_0\in K$ and $\bm{u}\in C_{\text{loc}}^1(\Omega';\mathbb{R}^k)$ be a weak solution of system \eqref{Sys:main-interior-1}. Assume that\\
		(h0) $B_R(x_0)\subseteq\Omega'$;\\
		(h1) $H_i(\bm{u},x_0,r)\geq \mu$ for any $r\in (2,R)$ and $i=1,2$;\\
		(h2) $J_i(\bm{u},x_0,r)>0$ and $\Lambda_i(\bm{u},x_0,r)>0$ for any $r\in (2,R)$ and $i=1,2$;\\
		(h3) $d'r^2\sigma_{N-1}^{1/2}H_i(\bm{u},0,r)^{-1/2}\leq \left(\frac{N-2}{2}\right)^2$ for any $r\in (2,R)$ and $i=1,2$.\\
		Then there exists a positive constant $C=C(\kappa,\mu,N,p_1,p_2)>0$ (independent of $(M_{ij})$, $x_0$, $d'$ and $R$) such that the function
		$$r\mapsto \frac{J_1(\bm{u},x_0,r)J_2(\bm{u},x_0,r)}{r^4}\exp\{-C(M_{12}\cdot r^2)^{-1/(2p+2)}+Cd' \int_{2}^{r}tH(t)^{-1/2}dt\}$$
		is monotone non-decreasing in $(2,R)$, where $p:=\max\{p_1,p_2\}\geq 1$ and 
		$$H(t):=\min_{i\in \{1,2\}}H_i(\bm{u},x_0,t).$$
	\end{theorem}
	\par 
	We refer the reader to \cite{Alt-Caffarelli-Friedman1984trans} for the original Alt-Caffarelli-Friedman monotonicity formula. 
	\par 
	When the interaction is Gross-Pitaevskii type $\beta u_{i}^p\sum\limits_{j\neq i}^ka_{ij}u_{j}^{p+1}$($i=1,\dots,k$), similar result with an additional condition (i.e., $H_1(\bm{u},x_0,r)/H_2(\bm{u},x_0,r)$ has uniform positive upper and lower bounds in $r\in (1,R)$ ) was first deduced by Wang \cite{Wang-kelei2014CPDE} for $p=1$, later by Soave-Zilio \cite{Soave-Z2015ARMA} for general $1\leq p\leq p_S-1$, and recently by Dias-Tavares \cite{Dias-T2023na} for general variational elliptic operators. Here, $p_S=p_S(N-1):=\frac{2(N-1)}{(N-3)^+}$ is the Sobolev critical exponent for $H^1(\mathbb{S}^{N-1})$, where $\frac{2(N-1)}{(N-3)^+}:=+\infty$ for $N=2,3$. We \cite{Zhang-Zhang2025} recently observe that this condition can be droped when dealing with the Gross-Pitaevskii  systems. The main ingredient of the proofs  is a variational  Poincar\'{e}-type lemma on $\mathbb{S}^{N-1}$ (see \cite[Lemma 4.2]{Wang-kelei2014CPDE} and \cite[Lemma 3.16, Remark 3.17]{Soave-Z2015ARMA}), which was proved in the subcritical or critical case, i.e., $1\leq p\leq p_S-1$. We remark that in the supercritical case, the validity of the Poincar\'{e}-type lemma remains unknown (for further details, see the opening part of the proof of Lemma \ref{lem:poincare-sphere} below). 
	\par 
	To establish Theorem \ref{thm:ACF-formula}, we begin by proving the variational  Poincar\'{e}-type lemma under supercritical conditions. The main challenge stems from the lack of regularity of the minimizers of the corresponding energy functional, which is caused by the supercriticality. To address this issue, we conduct a perturbation argument to find a minimizer that belongs to $(L^{\infty}(\mathbb{S}^{N - 1}))^2$. Subsequently, we can follow the proof of \cite[Lemma 3.16]{Soave-Z2015ARMA} to attain the desired result. For convenience, we state it as follows, including the subcritical and critical cases that are already known as mensioned above.
	\begin{lemma}\label{lem:poincare-sphere}
		Let $N\geq 3$, $p\geq 1$ and
		\begin{align}\label{def:H1-space}
			H_1:=\{(u,v)\in \left(H^1(\mathbb{S}^{N-1})\right)^2:\int_{\mathbb{S}^{N-1}}u^2=\int_{\mathbb{S}^{N-1}}v^2=1\}.
		\end{align} 
		Then there exists $C=C(N,p)>0$ such that if
		$$\sigma>0,\quad 0\leq \varepsilon\leq \left(\frac{N-2}{2}\right)^2,$$
		then
		\begin{align*}
			&\min_{(u,v)\in H_1}\gamma\left(\int_{\mathbb{S}^{N-1}}|\nabla_{\theta}u|^2+\sigma |u|^{p+1}|v|^{p+1}-\varepsilon u^2\right)+\gamma\left(\int_{\mathbb{S}^{N-1}}|\nabla_{\theta}v|^2+\sigma |u|^{p+1}|v|^{p+1}-\varepsilon v^2\right)\\
			\geq &2-C(\varepsilon+\sigma^{-1/(2p+2)}),
		\end{align*}
		where the function $\gamma$ is defined in \eqref{def:gamma-function}.
	\end{lemma}
	\begin{proof}
		Note that it suffices to consider the problem for $\varepsilon=0$ due to the fact that
		$$\gamma(t-\varepsilon)\geq \gamma(t)-\frac{2}{N-2}\varepsilon,\quad \forall t\geq 0,~\forall 0\leq\varepsilon\leq \left(\frac{N-2}{2}\right)^2.$$
		\par 
		If $p = 1$, this result was established in \cite[Lemma 4.2]{Wang-kelei2014CPDE} and \cite[Lemma 3.16]{Soave-Z2015ARMA}. For general $p \geq 1$, the authors in \cite[Remark 3.17]{Soave-Z2015ARMA} claimed that it can be proven following the same approach as in the proof of \cite[Lemma 3.16]{Soave-Z2015ARMA}, with the aid of \cite[Lemma 2.2]{Soave-Z2015ARMA} (i.e., Lemma \ref{lem:decay-Mu-p}). However, such a proof is only valid for $N = 3$ or for $N \geq 4$ with $1 \leq p \leq p_S - 1$. This restriction arises because applying the Lagrange multiplier rule as in \cite[(3.7), p. 673]{Soave-Z2015ARMA} requires that for any $\eta \in C^{\infty}(\mathbb{S}^{N-1})$, the functions
		\[
		t \mapsto J\left( \frac{u_0 + t\eta}{\|u_0 + t\eta\|_{L^2(\mathbb{S}^{N-1})}}, v_0 \right) \quad \text{and} \quad t \mapsto J\left( u_0, \frac{v_0 + t\eta}{\|v_0 + t\eta\|_{L^2(\mathbb{S}^{N-1})}} \right)
		\]
		are well-defined in a neighborhood of $0 \in \mathbb{R}$ and differentiable at $0$. Here,
		\[
		J(u, v) := \gamma \left( \int_{\mathbb{S}^{N-1}} |\nabla_{\theta}u|^2 + \sigma |u|^{p+1}|v|^{p+1} \right) + \gamma \left( \int_{\mathbb{S}^{N-1}} |\nabla_{\theta}v|^2 + \sigma |u|^{p+1}|v|^{p+1} \right),
		\]
		and $(u_0, v_0) \in H_1$ denotes a minimizer of $J|_{H_1}$. This condition is equivalent to $u_0, v_0 \in L^{p+1}(\mathbb{S}^{N-1})$, which follows from the Sobolev embedding $H^1(\mathbb{S}^{N-1}) \hookrightarrow L^{p+1}(\mathbb{S}^{N-1})$ provided that $p \leq p_S - 1$.
		\par 
		It remains to consider the case when $N\geq 4$ and $p>p_S-1$. In this case, the challenge arises because $u_0, v_0 \in L^{p+1}(\mathbb{S}^{N-1})$ are not known a priori, making the differential technique inapplicable. To overcome this, we aim to find a minimizer $(u,v)\in H_1$ of $J|_{H_1}$ with $u,v\in L^{\infty}(\mathbb{S}^{N-1})$ via a perturbation argument. Then the Lagrange multipliers rule can be applied and hence one can proceed as in the proof of \cite[Lemma 3.16]{Soave-Z2015ARMA} to obtain the desired result. To this aim, we consider the "truncated" functional defined by 
		\begin{align*}
			J_{m,s}(u,v):=
			&\gamma\left(\int_{\mathbb{S}^{N-1}}|\nabla_{\theta}u|^2+\sigma |\zeta_{m,s}(u)|^{p+1}|\zeta_{m,s}(v)|^{p+1}\right)\\
			&+\gamma\left(\int_{\mathbb{S}^{N-1}}|\nabla_{\theta}v|^2+\sigma |\zeta_{m,s}(u)|^{p+1}|\zeta_{m,s}(v)|^{p+1}\right)
		\end{align*}
		for $(u,v)\in (H^1(\mathbb{S}^{N-1}))^2$, where $m>0$, $0\leq s\leq 1$ and $\zeta_{m,s}\in C(\mathbb{R})$ is an odd function defined as 
		\begin{align}\label{def:zeta-ms}
			\zeta_{m,s}(t):=\left\{
			\begin{array}{ll}
				t,&\text{if}~~0\leq t\leq m,\\
				s/2+m,&\text{if}~~t\geq m+s,\\
				-\frac{1}{2s}(t-m-s)^2+s/2+m,&\text{if}~~s>0~\text{and}~m<t<m+s.
			\end{array}\right.
		\end{align} 
		Note that $\zeta_{m,s}\in C^1(\mathbb{R})$ if $s>0$ and $\|\zeta_{m,s}\|_{W^{1,\infty}(\mathbb{R})}\leq m+2$ for any $s\in [0,1]$.
		\par 
		First of all, for $0\leq s\leq 1\leq m$, we consider the following minimization problem
		$$M_{m,s}:=\inf_{(u,v)\in H_1}J_{m,s}(u,v)\geq 0.$$
		The direct method in Variational Calculus yields the existence of  a minimizer $(u_{m,s},v_{m,s})\in H_1$ of $J_{m,s}$ for $M_{m,s}$. Since $J_{m,s}(|u_{m,s}|,|v_{m,s}|)=J_{m,s}(u_{m,s},v_{m,s})$, we can asume that $u_{m,s},u_{m,s}\geq 0$ in $\mathbb{S}^{N-1}$. In addition, it is easy to see that there exists a constant $C>0$ independent of $\sigma,m,s$ such that 
		\begin{align}\label{eq:s-1-2}
			\begin{split}
				0\leq s_1:=\int_{\mathbb{S}^{N-1}}|\nabla_{\theta}u_{m,s}|^2+\sigma |\zeta_{m,s}(u_{m,s})|^{p+1}|\zeta_{m,s}(v_{m,s})|^{p+1}\leq C,\\
				0\leq s_2:=\int_{\mathbb{S}^{N-1}}|\nabla_{\theta}v_{m,s}|^2+\sigma |\zeta_{m,s}(u_{m,s})|^{p+1}|\zeta_{m,s}(v_{m,s})|^{p+1}\leq C.
			\end{split}
		\end{align}
		\par 
		Next, we assume that $0<s\leq 1\leq m$. Keeping $\zeta_{m,s}\in C^1(\mathbb{R})$ and $\|\zeta_{m,s}\|_{W^{1,\infty}(\mathbb{R})}\leq m+2$ in mind, we can apply the Lagrange multipliers rule to obain there exist $\lambda_{1,m,s},\lambda_{2,m,s}\in \mathbb{R}$ such that $(u_{m,s},v_{m,s})$ solves
		\begin{align}\label{Eq:GP-sphere-s}
			\left\{
			\begin{array}{l}
				-\Delta_{\theta}u=-\sigma\left[\frac{p+1}{2}+\frac{p+1}{2}\frac{\gamma'(s_2)}{\gamma'(s_1)}\right]|\zeta_{m,s}(u)|^{p-1}\zeta_{m,s}(u)\zeta_{m,s}'(u)|\zeta_{m,s}(v)|^{p+1}+\frac{\lambda_{1,m,s}}{\gamma'(s_1)}u,\\
				-\Delta_{\theta}v=-\sigma\left[\frac{p+1}{2}+\frac{p+1}{2}\frac{\gamma'(s_1)}{\gamma'(s_2)}\right]|\zeta_{m,s}(v)|^{p-1}\zeta_{m,s}(v)\zeta_{m,s}'(v)|\zeta_{m,s}(u)|^{p+1}+\frac{\lambda_{2,m,s}}{\gamma'(s_2)}v
			\end{array}\right.
			~\text{in}~\mathbb{S}^{N-1}
		\end{align}
		in the distribution sense, where $\Delta_{\theta}$ is the Laplace-Beltrami operator on the sphere $\mathbb{S}^{N-1}$.
		\par 
		We  claim that there exists a constant $C_1>0$ independent of $\sigma,m,s$ such that 
		\begin{align}
			0\leq \frac{\lambda_{1,m,s}}{\gamma'(s_1)}\leq C_1,\quad 	0\leq \frac{\lambda_{2,m,s}}{\gamma'(s_2)}\leq C_1.
		\end{align}
		\par 
		Indeed, by testing the equation for $u_{m,s}$ against $u_{m,s}$ in $\mathbb{S}^{N-1}$, we obtain
		\begin{align*}
			&\frac{\lambda_{1,m,s}}{\gamma'(s_1)}
			=\int_{ \mathbb{S}^{N-1}}|\nabla_{\theta}u_{m,s}|^2\\
			&+\sigma\left[\frac{p+1}{2}+\frac{p+1}{2}\frac{\gamma'(s_2)}{\gamma'(s_1)}\right]\int_{ \mathbb{S}^{N-1}}|\zeta_{m,s}(u_{m,s})|^{p-1}\zeta_{m,s}(u_{m,s})\zeta_{m,s}'(u_{m,s})|\zeta_{m,s}(v_{m,s})|^{p+1}u_{m,s}.
		\end{align*}
		By the definition of $\zeta_{m,s}$ and note that $s\leq m$, we have $\zeta_{m,s}(t)\geq \frac{1}{2}t$ for any $t\in [0,m+s]$. Combining this with $u_{m,s}\geq 0$ and $\zeta_{m,s}'(t)=0$ for $t\geq m+s$, we then deduce that 
		\begin{align*}
			0\leq& \int_{ \mathbb{S}^{N-1}}|\zeta_{m,s}(u_{m,s})|^{p-1}\zeta_{m,s}(u_{m,s})\zeta_{m,s}'(u_{m,s})|\zeta_{m,s}(v_{m,s})|^{p+1}u_{m,s}\\
			&\leq 2\int_{ \mathbb{S}^{N-1}}|\zeta_{m,s}(u_{m,s})|^{p+1}|\zeta_{m,s}(v_{m,s})|^{p+1}.
		\end{align*}
		This together with \eqref{eq:s-1-2} yields the desired estimate for  $\frac{\lambda_{1,m,s}}{\gamma'(s_1)}$. The estimate for $\frac{\lambda_{2,m,s}}{\gamma'(s_2)}$ is analogous.
		\par 
		With the claim in hand, we deduce from \eqref{Eq:GP-sphere-s} that $(u_{m,s},v_{m,s})\in H_1$ is a non-negative weak solution to 
		$$-\Delta_{\theta}u_{m,s}\leq C_1u_{m,s},\quad -\Delta_{\theta}v_{m,s}\leq C_1v_{m,s}\quad \text{in}~~\mathbb{S}^{N-1}.$$
		This allows us to deduce 
		\begin{align}\label{eq:u-ms-v-ms-bdd}
			0\leq u_{m,s}(x)\leq C_2,\quad 0\leq v_{m,s}(x)\leq C_2\quad \text{for}~~x\in \mathbb{S}^{N-1}
		\end{align}
		via Moser iteration, where $C_2>0$ is a constant independent of $\sigma,m,s$.
		\par 
		Finally, we fix $m>C_2+1$ and choose a sequence $\{s_n\}\subseteq (0,1)$ with $\lim\limits_{n\to\infty}s_n=0$. By the Sobolev compact embedding theorem, there exsits $(u_m,v_m)\in H_1$ such that up to a subsequence, 
		$$(u_{m,s_n},v_{m,s_n})\to (u_m,v_m)\quad \text{$\mathcal{H}^{N-1}$-a.e. in}~~\mathbb{S}^{N-1}\quad \text{as}~~n\to\infty.$$ 
		Then by the definition of $\zeta_{m,s}$ and the Lebesgue dominated convergence theorem, we deduce that
		\begin{align*}
			M:=\inf_{(u,v)\in H_1}J(u,v)\geq& \inf_{(u,v)\in H_1}J_{m,0}(u,v)=J_{m,0}(u_{m,0},v_{m,0})=\lim_{n\to\infty}J_{m,s_n}(u_{m,0},v_{m,0})\\
			\geq &\lim_{n\to\infty}J_{m,s_n}(u_{m,s_n},v_{m,s_n})=\lim_{n\to\infty}J(u_{m,s_n},v_{m,s_n})\geq J(u_m,v_m)\geq M.
		\end{align*}
		Hence, $(u_m,v_m)\in H_1$ is a minimizer of $J|_{H_1}$ satisfying $u_m,v_m\in L^{\infty}(\mathbb{S}^{N-1})$ (by \eqref{eq:u-ms-v-ms-bdd}), as desired.
	\end{proof}
	\par 
	Since system \eqref{Sys:main-interior-1} is non-variational,  Lemma \ref{lem:poincare-sphere} can not be used directly. 
	Therefore, we need to extend the above Poincar\'{e}-type lemma to the non-variational case. At this stage, difficulties arise due to the non-variational structure and the lack of regularity of the corresponding functional when  some of $p_i$ ($1\leq i\leq k$) is equal to one. Notably, the variational structure and the regularity of the corresponding functional are crucially relied upon in the proof of Lemma \ref{lem:poincare-sphere}. To overcome these challenges, our approach remains to employ a perturbation argument to construct a non-negative minimizer with a universal bound in $(L^{\infty}(\mathbb{S}^{N-1}))^2$. Then the desired result can be deduced via Lemma \ref{lem:poincare-sphere}.
	\begin{lemma}\label{lem:poincare-sphere-LV}
		Let $N\geq 3$, $p_i,q_i\geq 1$ for $i=1,2$. Then there exists a cosntant $C=C(N,p_1,p_2,q_1,q_2)>0$ such that if
		$$\sigma>0,\quad 0\leq \varepsilon\leq \left(\frac{N-2}{2}\right)^2,$$
		then
		\begin{align*}
			&\inf_{(u,v)\in H_1}\gamma\left(\int_{\mathbb{S}^{N-1}}|\nabla_{\theta}u|^2+\sigma |u|^{p_1}|v|^{q_1}-\varepsilon u^2\right)+\gamma\left(\int_{\mathbb{S}^{N-1}}|\nabla_{\theta}v|^2+\sigma |u|^{p_2}|v|^{q_2}-\varepsilon v^2\right)\\
			\geq &2-C(\varepsilon+\sigma^{-1/(2p+2)}),
		\end{align*}
		where $p:=\max\{1,\max\{p_1,q_1,p_2,q_2\}-1\}\geq 1$ and the space $H_1$ is defined in \eqref{def:H1-space}.
	\end{lemma}
	\begin{proof}
		As before, it suffices to consider the problem for $\varepsilon=0$.
		Let 
		$$\widetilde{J}(u,v):=\gamma\left(\int_{\mathbb{S}^{N-1}}|\nabla_{\theta}u|^2+\sigma |u|^{p_1}|v|^{q_1} \right)+\gamma\left(\int_{\mathbb{S}^{N-1}}|\nabla_{\theta}v|^2+\sigma |u|^{p_2}|v|^{q_2}\right)$$
		for $(u,v)\in (H^1(\mathbb{S}^{N-1}))^2$. We aim to find a minimizer of $\widetilde{J}|_{H_1}$ which has a uniform bound in $(L^{\infty}(\mathbb{S}^{N-1}))^2$ with respect to $\sigma>0$. The proof idea is similar to that of the proof of Lemma \ref{lem:poincare-sphere}. The difference is that here some of the exponents (i.e., $p_i,q_i$) may equal one, causing the lack of regularity of the functional $\widetilde{J}$. Therefore, we consider the following three-parameter functional defined by 
		\begin{align*}
			\widetilde{J}_{m,s,\varepsilon}(u,v)
			:=&\gamma\left(\int_{\mathbb{S}^{N-1}}|\nabla_{\theta}u|^2+\sigma |\zeta_{m,s}(u)|^{p_1+\varepsilon}|\zeta_{m,s}(v)|^{q_1+\varepsilon} \right)\\
			&+\gamma\left(\int_{\mathbb{S}^{N-1}}|\nabla_{\theta}v|^2+\sigma |\zeta_{m,s}(u)|^{p_2+\varepsilon}|\zeta_{m,s}(v)|^{q_2+\varepsilon}\right),
		\end{align*}
		where $m\geq 1$, $s,\varepsilon\in [0,1]$ and $\zeta_{m,s}\in C(\mathbb{R})$ is the odd function defined in \eqref{def:zeta-ms}. 
		\par 
		We divide the remaining proof into four steps.
		\par 
		\textbf{Step 1.} For $m\geq 1$ and $s,\varepsilon\in [0,1]$,  we consider the following minimization problem
		$$\widetilde{M}_{m,s,\varepsilon}:=\inf_{(u,v)\in H_1}\widetilde{J}_{m,s,\varepsilon}(u,v)\geq 0.$$
		By the direct method and the fact that $\widetilde{J}_{m,s,\varepsilon}(|u|,|v|)=\widetilde{J}_{m,s,\varepsilon}(u,v)$,  one can find a non-negative minimizer $(u_{m,s,\varepsilon},v_{m,s,\varepsilon})\in H_1$ of $\widetilde{J}_{m,s,\varepsilon}|_{H_1}$ for $\widetilde{M}_{m,s,\varepsilon}$. Moreover, there exists a constant $C>0$ independent of $\sigma,m,s,\varepsilon$ such that 
		\begin{align}\label{eq:tilde-s-1-2}
			\begin{split}
				0\leq \tilde{s}_1:=\int_{\mathbb{S}^{N-1}}|\nabla_{\theta}u_{m,s,\varepsilon}|^2+\sigma |\zeta_{m,s}(u_{m,s,\varepsilon})|^{p_1+\varepsilon}|\zeta_{m,s}(v_{m,s,\varepsilon})|^{q_1+\varepsilon}\leq C,\\
				0\leq \tilde{s}_2:=\int_{\mathbb{S}^{N-1}}|\nabla_{\theta}v_{m,s,\varepsilon}|^2+\sigma |\zeta_{m,s}(u_{m,s,\varepsilon})|^{p_2+\varepsilon}|\zeta_{m,s}(v_{m,s,\varepsilon})|^{q_2+\varepsilon}\leq C.
			\end{split}
		\end{align}
		\par 
		\textbf{Step 2.} Assume that $m\geq 1$ and $0<s,\varepsilon<1$. By the Lagrange multipliers rule, there exist $\lambda_{1,m,s,\varepsilon},\lambda_{2,m,s,\varepsilon}\in \mathbb{R}$ such that $(u_{m,s,\varepsilon},v_{m,s,\varepsilon})$ solves
		\begin{align}\label{Eq:LV-sphere-s}
			\left\{
			\begin{array}{l}
				-\Delta_{\theta}u=-\sigma h_{1,m,s,\varepsilon}(u,v)+\frac{\lambda_{1,m,s,\varepsilon}}{\gamma'(\tilde{s}_1)}u,\\
				-\Delta_{\theta}v=-\sigma h_{2,m,s,\varepsilon}(u,v)+\frac{\lambda_{2,m,s,\varepsilon}}{\gamma'(\tilde{s}_2)}v
			\end{array}\right.
			\quad\text{in}~~\mathbb{S}^{N-1}
		\end{align}
		in the distribution sense, where 
		\begin{align*}
			\begin{split}
				h_{1,m,s,\varepsilon}(u,v):=&\frac{p_1+\varepsilon}{2}|\zeta_{m,s}(u)|^{p_1+\varepsilon-2}\zeta_{m,s}(u)\zeta_{m,s}'(u)|\zeta_{m,s}(v)|^{q_1+\varepsilon}\\
				&+\frac{p_2+\varepsilon}{2}\frac{\gamma'(\tilde{s}_2)}{\gamma'(\tilde{s}_1)}|\zeta_{m,s}(u)|^{p_2+\varepsilon-2}\zeta_{m,s}(u)\zeta_{m,s}'(u)|\zeta_{m,s}(v)|^{q_2+\varepsilon},\\
				h_{2,m,s,\varepsilon}(u,v):=&\frac{q_2+\varepsilon}{2}|\zeta_{m,s}(v)|^{q_2+\varepsilon-2}\zeta_{m,s}(v)\zeta_{m,s}'(v)|\zeta_{m,s}(u)|^{p_2+\varepsilon}\\
				&+\frac{q_1+\varepsilon}{2}\frac{\gamma'(\tilde{s}_1)}{\gamma'(\tilde{s}_2)}|\zeta_{m,s}(v)|^{q_1+\varepsilon-2}\zeta_{m,s}(v)\zeta_{m,s}'(v)|\zeta_{m,s}(u)|^{p_1+\varepsilon}.
			\end{split}
		\end{align*}
		As in the proof of Lemma \ref{lem:poincare-sphere}, we test the first and second equations of \eqref{Eq:LV-sphere-s} against $u_{m,s,\varepsilon}$ and $v_{m,s,\varepsilon}$ in $\mathbb{S}^{N-1}$, respectively, and then use \eqref{eq:tilde-s-1-2} together with the properties of $\zeta_{m,s}$ to deduce that 
		\begin{align*}
			0\leq \frac{\lambda_{1,m,s,\varepsilon}}{\gamma'(\tilde{s}_1)}\leq C_1,\quad 	0\leq \frac{\lambda_{2,m,s,\varepsilon}}{\gamma'(\tilde{s}_2)}\leq C_1,
		\end{align*}
		where $C_1>0$ is a constant independent of $\sigma,m,s,\varepsilon$. 
		Combining this with \eqref{Eq:LV-sphere-s}, we can apply Moser iteration to obtain a constant $C_2>0$ independent of $\sigma,m,s,\varepsilon$ such that 
		\begin{align}\label{eq:bdd-uv-msvarepsilon}
			\|u_{m,s,\varepsilon}\|_{L^{\infty}(\mathbb{S}^{N-1})}\leq C_2,\quad 	\|v_{m,s,\varepsilon}\|_{L^{\infty}(\mathbb{S}^{N-1})}\leq C_2.
		\end{align}
		\par 
		\textbf{Step 3.} We fix $m>C_2+1$ and choose a sequence $\{s_n\}\subseteq (0,1)$ with $\lim\limits_{n\to\infty}s_n=0$. The Sobolev compact embedding theorem yields that there exsits $(\hat{u},\hat{v})\in H_1$ such that up to a subsequence, 
		$$(u_{m,s_n,s_n},v_{m,s_n,s_n})\to (\hat{u},\hat{v})\quad \text{$\mathcal{H}^{N-1}$-a.e. in}~~\mathbb{S}^{N-1}\quad \text{as}~~n\to\infty.$$ 
		Then by the definition of $\zeta_{m,s}$ and the Lebesgue dominated convergence theorem, we deduce that 
		\begin{align*}
			\widetilde{M}:=&\inf_{(u,v)\in H_1}\widetilde{J}(u,v)\geq \inf_{(u,v)\in H_1}\widetilde{J}_{m,0,0}(u,v)\\
			=&\widetilde{J}_{m,0,0}(u_{m,0,0},v_{m,0,0})=\lim_{n\to\infty}\widetilde{J}_{m,s_n,s_n}(u_{m,0,0},v_{m,0,0})\\
			\geq &\limsup_{n\to\infty}\widetilde{J}_{m,s_n,s_n}(u_{m,s_n,s_n},v_{m,s_n,s_n})=\limsup_{n\to\infty}\widetilde{J}_{m,s_n,0}(u_{m,s_n,s_n},v_{m,s_n,s_n})\\
			=&\limsup_{n\to\infty}\widetilde{J}(u_{m,s_n,s_n},v_{m,s_n,s_n})\geq \widetilde{J}(\hat{u},\hat{v})\geq \widetilde{M}.
		\end{align*}
		Here we have used the fact that the function $g_i(t_1,t_2,\varepsilon):=|t_1|^{p_i+\varepsilon}|t_2|^{q_i+\varepsilon}$ is uniformly continuous in $[0,C_2]\times[0,C_2]\times [0,1]$ for $i=1,2$.
		Hence, $(\hat{u},\hat{v})\in H_1$ is a minimizer of $\widetilde{J}|_{H_1}$. Moreover, it follows from \eqref{eq:bdd-uv-msvarepsilon} that 
		\begin{align}\label{eq:bdd-um-vm}
			\|\hat{u}\|_{L^{\infty}(\mathbb{S}^{N-1})}\leq C_2,\quad 	\|\hat{v}\|_{L^{\infty}(\mathbb{S}^{N-1})}\leq C_2.
		\end{align}
		\par 
		\textbf{Step 4. } By  \eqref{eq:bdd-um-vm} and Lemma \ref{lem:poincare-sphere}, we then obtain
		\begin{align*}
			\inf_{(u,v)\in H_1}\widetilde{J}(u,v)=\widetilde{J}(\hat{u},\hat{v})
			=& \gamma\left(\int_{\mathbb{S}^{N-1}}|\nabla_{\theta}\hat{u}|^2+\sigma |\hat{u}|^{p_1}|\hat{v}|^{q_1} \right)+\gamma\left(\int_{\mathbb{S}^{N-1}}|\nabla_{\theta}\hat{v}|^2+\sigma |\hat{u}|^{p_2}|\hat{v}|^{q_2}\right)\\
			\geq& \gamma\left(\int_{\mathbb{S}^{N-1}}|\nabla_{\theta}\hat{u}|^2+\sigma (C_2)^{p_1+q_1-2(p+1)}|\hat{u}|^{p+1}|\hat{v}|^{p+1} \right)\\
			&+\gamma\left(\int_{\mathbb{S}^{N-1}}|\nabla_{\theta}\hat{v}|^2+\sigma (C_2)^{p_2+q_2-2(p+1)}|\hat{u}|^{p+1}|\hat{v}|^{p+1}\right)\\
			\geq& 2-C_4[\sigma C_3]^{-1/(2p+2)}=2-C_5\sigma^{-1/(2p+2)},
		\end{align*}
		where $p=\max\{1,\max\{p_1,q_1,p_2,q_2\}-1\}\geq 1$ and $C_3,C_4,C_5>0$ are constants independent of $\sigma$. The proof is finished.
	\end{proof}
	\par 
	Now with Lemma \ref{lem:poincare-sphere-LV} in hand, we are able to prove Theorem \ref{thm:ACF-formula}. 
	\begin{proof}[Proof of Theorem \ref{thm:ACF-formula}]
		For simplicity, we assume that $x_0=0$ and rewrite $J_i(\bm{u},x_0,r)$ and $\Lambda_i(\bm{u},x_0,r)$ as $J_i(r)$ and $\Lambda_i(r)$ respectively for $i=1,2$. By (h2) and Lemma \ref{lem:est:J-i(r)} (3), we deduce that for $r\in (2,R)$,
		\begin{align*}
			\frac{d}{dr}\ln \left(\frac{J_1(r)J_2(r)}{r^4}\right)
			=&-\frac{4}{r}+\frac{\frac{d}{dr}J_1(r)}{J_1(r)}+\frac{\frac{d}{dr}J_2(r)}{J_2(r)}\\
			=&-\frac{4}{r}+\frac{\int_{\partial B_r}\left[|\nabla u_1|^2+ u_1^{p_1+1}\sum_{j\neq 1}^ka_{1j}M_{1j}u_j^{p_{j}}-d'u_1\right]|x|^{2-N}}{\int_{B_r}\left[|\nabla u_1|^2+u_1^{p_1+1}\sum_{j\neq 1}^ka_{1j}M_{1j}u_j^{p_{j}}-d'u_1\right]|x|^{2-N}}\\
			&+\frac{\int_{\partial B_r}\left[|\nabla u_2|^2+u_2^{p_2+1}\sum_{j\neq 2}^ka_{2j}M_{2j}u_j^{p_{j}}-d'u_2\right]|x|^{2-N}}{\int_{B_r}\left[|\nabla u_2|^2+ u_2^{p_2+1}\sum_{j\neq 2}^ka_{2j}M_{2j}u_j^{p_{j}}-d'u_2\right]|x|^{2-N}}\\
			\geq &\frac{2}{r}\left[\gamma(\Lambda_1(r))+\gamma(\Lambda_2(r))-2\right]\\
			\geq& -\frac{4}{r}+\frac{2}{r}\gamma\left(\frac{r^2\int_{\partial B_r}|\nabla_{\theta} u_1|^2+M_{12}\kappa u_1^{p_1+1}u_2^{p_{2}}-d'u_1}{\int_{\partial B_r}u_1^2}\right)\\
			&+\frac{2}{r}\gamma\left(\frac{r^2\int_{\partial B_r}|\nabla_{\theta} u_2|^2+M_{12}\kappa u_2^{p_2+1}u_1^{p_{1}}-d'u_2}{\int_{\partial B_r}u_2^2}\right).
		\end{align*}
		By H\"{o}lder's inequality, we have 
		$$\frac{r^2\int_{\partial B_r}d'u_i}{\int_{\partial B_r}u_i^2}\leq d'r^2\sigma_{N-1}^{1/2}H_i(\bm{u},0,r)^{-1/2},\quad i=1,2.$$
		On the other hand, for $i=1,2$,
		let
		$$u_{i,r}(x):=\frac{u_i(rx)}{\left(\frac{1}{r^{N-1}}\int_{\partial B_r}u_i^2\right)^{1/2}},\quad \forall x\in \mathbb{S}^{N-1}.$$
		Then by (h1), we obtain
		\begin{align*}
			&\frac{r^2\int_{\partial B_r}|\nabla_{\theta} u_1|^2+M_{12}\kappa u_1^{p_1+1}u_2^{p_{2}}}{\int_{\partial B_r}u_1^2}\\
			=&\int_{\mathbb{S}^{N-1}}|\nabla_{\theta} u_{1,r}|^2+r^2\left(\frac{1}{r^{N-1}}\int_{\partial B_r}u_1^2\right)^{(p_1+1)/2-1}\left(\frac{1}{r^{N-1}}\int_{\partial B_r}u_2^2\right)^{p_{2}/2}M_{12} \kappa u_{1,r}^{p_1+1}u_{2,r}^{p_{2}}\\
			\geq& \int_{\mathbb{S}^{N-1}}|\nabla_{\theta} u_{1,r}|^2+r^2\mu^{(p_1+p_{2}-1)/2}M_{12} \kappa u_{1,r}^{p_1+1}u_{2,r}^{p_{2}}.
		\end{align*}
		Similarly,
		\begin{align*}
			\frac{r^2\int_{\partial B_r}|\nabla_{\theta} u_2|^2+M_{12}\kappa u_2^{p_2+1}u_1^{p_{1}}}{\int_{\partial B_r}u_2^2}
			\geq  \int_{\mathbb{S}^{N-1}}|\nabla_{\theta} u_{2,r}|^2+r^2\mu^{(p_2+p_{1}-1)/2}M_{12} \kappa u_{1,r}^{p_1}u_{2,r}^{p_{2}+1}.
		\end{align*}
		Hence, by (h3), we can apply Lemma \ref{lem:poincare-sphere-LV} with
		$$\sigma=r^2\mu^{(p_1+p_{2}-1)/2}M_{12} \kappa>0\quad \text{and}\quad \varepsilon=d'r^2\sigma_{N-1}^{1/2}\left[\min_{i\in \{1,2\}}H_i(\bm{u},0,r)\right]^{-1/2}>0$$ 
		to deduce that there exists $C_1=C_1(N,p_1,p_2)>0$ such that
		\begin{align*}
			\frac{d}{dr}\ln \left(\frac{J_1(r)J_2(r)}{r^4}\right)
			\geq& -\frac{4}{r}+\frac{2}{r}\gamma\left(\int_{\mathbb{S}^{N-1}}|\nabla_{\theta} u_{1,r}|^2+\sigma u_{1,r}^{p_1+1}u_{2,r}^{p_{2}}d\mathcal{H}^{N-1}-\varepsilon\right)\\
			&+\frac{2}{r}\gamma\left(\int_{\mathbb{S}^{N-1}}|\nabla_{\theta} u_{2,r}|^2+\sigma u_{1,r}^{p_1}u_{2,r}^{p_{2}+1}d\mathcal{H}^{N-1}-\varepsilon\right)\\
			&\geq -\frac{C_1}{r}(\varepsilon+\sigma^{-1/(2p+2)})\\
			=& -CM_{12}^{-1/(2p+2)}r^{-(p+2)/(p+1)}-Cd'rH(r)^{-1/2},
		\end{align*}
		where $C=C(\kappa,\mu,N,p_1,p_{2})>0$ is a constant, $p=\max\{p_1,p_2\}\geq 1$ and $H(r)=\min_{i\in \{1,2\}}H_i(\bm{u},0,r)$. By integrating, the desired result follows.
	\end{proof}
	
	\section{Uniform interior Lipschitz bounds}\label{sec:4}
	\setcounter{section}{4}
	\setcounter{equation}{0}
	
	This section is devoted to establishing uniform interior Lipschitz bounds for solutions to system \eqref{Sys:main-interior}; specifically, we prove Theorem \ref{thm:interior-lip} herein. The proof relies on blow-up analysis and an Alt-Caffarelli-Friedman type monotonicity formula (see Theorem \ref{thm:ACF-formula}, established in Section \ref{sec:3}). In particular, we follow the proof strategy developed in \cite{Zhang-Zhang2025} to prove Theorem \ref{thm:interior-lip}. As noted earlier, with the aid of Theorem \ref{thm:ACF-formula}, the remaining challenge we need to address stems from the nonhomogeneity of the competition-reaction function $\bm{\hat{g}}$ defined in \eqref{def:competition-reaction}. To overcome this challenge, we employ an induction argument on the index $k$.
	\par 
	Following \cite{Dias-T2023na, Soave-Z2015ARMA, Zhang-Zhang2025}, we may assume without loss of generality that $N \geq 3$, as the case $N \leq 2$ can be handled by extending the solutions to spatial ones.
	\par
	Let us consider a sequence of weak solutions $\{\bm{u}_n\}$ to system \eqref{Sys:main-interior-general} that satisfy the assumptions of Theorem \ref{thm:interior-lip}. Without loss of generality, we assume that $B_3 \subseteq \Omega$ and seek to prove the uniform Lipschitz bound in $B_1$.
	As in \cite{Zhang-Zhang2025}, we select a radially non-increasing cut-off function $\eta \in C_c^1(\mathbb{R}^N)$ such that
	$$\eta\equiv 1~~\text{in}~~B_1,\quad \eta(x)=(2-|x|)^2~~ \text{for}~~x\in B_2\setminus B_{3/2},\quad \eta\equiv 0~~\text{in}~~\mathbb{R}^N\setminus B_2,$$
	which possesses the following property:
	\begin{align}\label{property:eta}
		\sup_{x\in B_2}\sup_{0<\rho<\frac{1}{2}\text{dist}(x,\partial B_2)}\frac{\sup_{B_{\rho}(x)}\eta}{\inf_{B_{\rho}(x)}\eta}\leq 64.
	\end{align}
	By definition, it then suffices to derive a uniform bound on the gradients of $\{\eta \bm{u}_n\}$; specifically, we aim to prove that there exists a constant $C > 0$ (independent of $n$) such that
	\begin{align}\label{bdd:interior}
		\max_{1 \leq i \leq k} \max_{x \in \overline{B_2}} \left| \nabla \left( \eta u_{i,n} \right) (x) \right| \leq C.
	\end{align}
	If $\{\beta_{ij;n}\}$ is bounded for all $i \neq j$, then the desired bound \eqref{bdd:interior} follows from the regularity theory for elliptic equations (see, e.g., \cite[Theorem 8.32]{Gilbarg-Trudinger2001-book}). Hence, we only need to consider the case where $\beta_{ij;n} \to +\infty$ as $n \to \infty$ for some $i \neq j$.
	\par
	Assume by contradiction that \eqref{bdd:interior} fails to hold. Then there exists a subsequence of $\{\bm{u}_n\}$ (still denoted by $\{\bm{u}_n\}$) such that
	\begin{align}\label{def:L-n}
		L_n := \max_{1 \leq i \leq k} \max_{x \in \overline{B_2}} \left| \nabla \left( \eta u_{i,n} \right) (x) \right| \to +\infty \quad \text{as } n \to \infty.
	\end{align}
	Since $\nabla(\eta u_{i,n}) \equiv 0$ in $B_3 \setminus B_2$, we may assume, up to a relabelling of indices, that the maximum in \eqref{def:L-n} is achieved for $i = 1$ and at a point $x_n \in B_2$. We introduce two blow-up sequences as follows:
	\begin{align}\label{def:blow-up-seq}
		v_{i,n}(x) := \eta(x_n) \cdot \frac{u_{i,n}(x_n + r_n x)}{L_n r_n}, \quad \bar{v}_{i,n}(x) := \frac{(\eta u_{i,n})(x_n + r_n x)}{L_n r_n}, \quad 1 \leq i \leq k,
	\end{align}
	which are defined on the scaled domain $\Omega_n := (B_3 - x_n)/r_n \subseteq (\Omega - x_n)/r_n$. Here, $r_n > 0$ is defined by
	\begin{align*}
		r_n := \sum_{i=1}^k \frac{(\eta u_{i,n})(x_n)}{L_n} \to 0 \quad \text{as } n \to \infty.
	\end{align*}
	This limit $r_n \to 0$ follows from \eqref{def:L-n} and the fact that $\{\bm{u}_{n}\}$ is bounded in $L^{\infty}(\Omega; \mathbb{R}^k)$. We also denote the scaled nonlinear term as
	\begin{align*}
		\hat{f}_{i,n}(x) := r_n \cdot \frac{\eta(x_n)}{L_n} f_{i,n}\left( x_n + r_n x, u_{i,n}(x_n + r_n x) \right), \quad \forall x \in \Omega_n, ~\forall 1 \leq i \leq k.
	\end{align*}
	
	\subsection{Asymptotics of the blow-up sequence}
	
	First of all, we derive some preliminary properties of $\{\bm{v}_n\}$ and $\{\bm{\hat{f}}_{n}\}$.
	\begin{lemma}\label{lem:pre-property-v-n}
		Under the previous notation, up to a subsequence, the following assertions hold:\\
		(1) $\sum_{i=1}^kv_{i,n}(0)=1$ for any $n\geq 1$.\\
		(2) $|\hat{f}_{i,n}(x)|\leq d_m\frac{\eta(x_n)}{L_n}r_n$ a.e. for $x\in \Omega_n$ and $\|\hat{f}_{i,n}\|_{L^{\infty}(\Omega_n)}\to 0$ as $n\to\infty$, $\forall 1\leq i\leq k$.\\
		(3) There holds $\Omega_n':=B_{\bar{r}_n}\subseteq B_{1/r_n}\subseteq\Omega_n$ for every $n\geq 1$, where
		\begin{align}\label{def:bar-r-n}
			\bar{r}_n:=\frac{1}{2r_n}\text{dist}(x_n,\partial B_2)\to +\infty\quad \text{as}~~n\to\infty.
		\end{align}
		(4) The sequence $\{\bm{v}_n\}$ satisfies
		\begin{align}\label{sys:v-n}
			-\Delta v_{i,n}=\hat{f}_{i,n}- v_{i,n}^{p_i}\sum_{j\neq i}a_{ij}M_{i,j;n}v_{j,n}^{p_j}\quad \text{in}~~\Omega_n,
		\end{align}
		where
		\begin{align}\label{def:M-ijn}
			M_{i,j,n}:=\eta(x_n)\frac{r_n}{L_n}\beta_{ij;n} \left(\frac{\eta(x_n)}{L_nr_n}\right)^{-p_i-p_j}.
		\end{align}
		(5) The sequence $\{\bm{v}_n\}$ has uniformly bounded Lip-seminorm in $\Omega_n'$, that is,
		$$\max_{1\leq i\leq k}\sup_{x,y\in \Omega_n',\atop x\neq y}\frac{|v_{i,n}(x)-v_{i,n}(y)|}{|x-y|}\leq 128,\quad \forall n\geq 1;$$
		moreover, $|\nabla v_{1,n}(0)|\to 1$ as $n\to\infty$.\\
		(6) There exists a globally Lipschitz continuous non-negative function $\bm{v}$ in $\mathbb{R}^N$ such that $\bm{v}_n\to \bm{v}$ in $C_{\text{loc}}(\mathbb{R}^N;\mathbb{R}^k)$ as $n\to\infty$; moreover,
		$$\max_{1\leq i\leq k}\sup_{x,y\in \mathbb{R}^N,\atop x\neq y}\frac{|v_{i}(x)-v_{i}(y)|}{|x-y|}\leq 1.$$
	\end{lemma}
	\begin{proof}
		Points (1), (2) and (4) follow by definitions.
		\par
		(3) The fact that $\Omega_n'\subseteq B_{1/r_n}\subseteq \Omega_n$ follows by definition. On the other hand, since
		\begin{align}\label{eq:r-n-leq-dist}
			r_n=\sum_{i=1}^k\frac{(\eta u_{i,n})(x_n)}{L_n}\leq \frac{\|\bm{u}_{n}\|_{L^{\infty}(\Omega;\mathbb{R}^k)}}{L_n}\eta(x_n)\leq \frac{m\bar{l}}{L_n}\text{dist}(x_n,\partial B_2),\quad \forall n\geq 1,
		\end{align}
		where $\bar{l}:=\|\nabla\eta\|_{L^{\infty}(B_3)}>0$ and $\|\bm{u}_{n}\|_{L^{\infty}(\Omega;\mathbb{R}^k)}=\sum_{i=1}^k\|u_{i,n}\|_{L^{\infty}(\Omega)}$, we deduce from \eqref{def:L-n} that
		$$\bar{r}_n=\frac{1}{2r_n}\text{dist}(x_n,\partial B_2)\geq \frac{L_n}{2m\bar{l}}\to+\infty,\quad \text{as}~~n\to\infty.$$
		\par
		(5) For any $1\leq i\leq k$ and any $x\in \Omega_n'$, one has $|x_n+r_nx-x_n|<\frac{1}{2}\text{dist}(x_n,\partial B_2)$. It follows from \eqref{property:eta} that
		$\frac{\eta(x_n)}{\eta(x_n+r_nx)}\leq 64$. Hence,
		\begin{align*}
			|\nabla v_{i,n}(x)|
			&=\frac{\eta(x_n)}{\eta(x_n+r_nx)}\left|\frac{\eta(x_n+r_nx)(\nabla u_{i,n})(x_n+r_nx)}{L_n}\right|\\
			&=\frac{\eta(x_n)}{\eta(x_n+r_nx)}\left|\nabla \bar{v}_{i,n}(x)-\frac{u_{i,n}(x_n+r_nx)(\nabla \eta)(x_n+r_nx)}{L_n}\right|\\
			&\leq 64\left[|\nabla \bar{v}_{i,n}(x)|+\frac{m\bar{l}}{L_n}\right]\leq 128,\quad \forall n\gg 1.
		\end{align*}
		On the other hand, note that $|\nabla\bar{v}_{1,n}(0)|=1$ and
		$$\nabla v_{1,n}(0)=\nabla\bar{v}_{1,n}(0)-\frac{u_{i,n}(x_n)(\nabla \eta)(x_n)}{L_n}=\nabla\bar{v}_{1,n}(0)+o(1).$$
		It follows that $|\nabla v_{1,n}(0)|\to 1$ as $n\to\infty$.
		\par
		(6) By Point (3), we see that $\Omega_n'$ exhausts $\mathbb{R}^N$. Thus, it follows from Points (1), (5) and the Ascoli-Arzel\`{a} theorem that there exists $\bm{v}\in C_{\text{loc}}(\mathbb{R}^N;\mathbb{R}^k)$ such that $\bm{v}_n\to \bm{v}$ in $C_{\text{loc}}(\mathbb{R}^N;\mathbb{R}^k)$ as $n\to\infty$. On the other hand, since for any compact set $K'\Subset \mathbb{R}^N$,
		$$\max_{x\in K'}|\bar{v}_{i,n}(x)-v_{i,n}(x)|\leq \max_{x\in K'}(\bar{l}|x|u_{i,n}(x_n+r_nx)/L_n)\leq \frac{m\bar{l}\max_{x\in K'}|x|}{L_n}\to0\quad \text{as}~~n\to\infty,$$
		it follows that $\bar{\bm{v}}_n\to \bm{v}$ in $C_{\text{loc}}(\mathbb{R}^N;\mathbb{R}^k)$ as $n\to\infty$. Note that $|\nabla \bar{\bm{v}}_n|\leq 1$ in $\Omega_n$ for any $n\geq 1$, we then deduce that
		$$\max_{1\leq i\leq k}\sup_{x,y\in \mathbb{R}^N,\atop x\neq y}\frac{|v_{i}(x)-v_{i}(y)|}{|x-y|}\leq 1.$$
		The proof is finished.
	\end{proof}
	\par 
	Next, note that since $k \geq 3$, the quantities $M_{i,j;n}$ ($i \neq j$) defined in \eqref{def:M-ijn} may not lie on the same scale as $n \to \infty$. Specifically, for some pairs $(i, j)$ with $i \neq j$, the sequences $\{M_{i,j;n}\}_{n=1}^{\infty}$ could be unbounded, while for other such pairs, they remain bounded. Consequently, proving the nontriviality or non-constancy of $v_1$ as in \cite{Soave-Z2015ARMA, Zhang-Zhang2025} appears to be impossible. To tackle this problem, we rescale $\{\bm{v}_n\}$ in a suitable way via an induction method and then prove that this new sequence satisfies all the properties that we need in the next subsection. Actually, we can deduce the following result.
	\begin{lemma}\label{lem:new-u-n-property}
		Up to a relabeling of the subscript set $\{2,3,\dots,k\}$, there exist an integer $2\leq \bar{k}\leq k$ and a sequence $\{R_n\}\subseteq (0,1]$ such that the functions $\{(\bar{u}_{1,n},\dots,\bar{u}_{\bar{k},n})\}$ defined by 
		\begin{align}
			\bar{u}_{i,n}:=\frac{1}{R_n}v_{i,n}(R_nx),\quad x\in U_n:=\Omega_n/R_n,\quad 1\leq i\leq \bar{k},~~n\geq 1,
		\end{align}
		weakly solve
		\begin{align}\label{sys:bar-u-in}
			-\Delta \bar{u}_{i,n}=g_{i,n}(x) -\bar{u}_{i,n}^{p_i}\sum_{j=1,j\neq i}^{\bar{k}}a_{ij}\overline{M}_{i,j,;n}\bar{u}_{j,n}^{p_j}\quad \text{in}~~U_n,\quad  1\leq i\leq \bar{k},~~n\geq 1,
		\end{align}
		where $\overline{M}_{i,j,;n}:=R_n^{p_i+p_j+1}M_{i,j;n}$ and 
		$$ g_{i,n}(x):=R_n\hat{f}_{i,n}(R_nx)-R_n[v_{i,n}(R_nx)]^{p_i}\sum_{j=\bar{k}+1,j\neq i}^{k}a_{ij}M_{i,j,;n}[v_{j,n}(R_nx)]^{p_j}.$$
		Moreover, up to a subsequence, the following assertions hold:\\
		(1) $\sum_{i=1}^{\bar{k}}\bar{u}_{i,n}(0)=1$ for any $n\geq 1$.\\
		(2) $g_{i,n}(x)\leq d_m\frac{\eta(x_n)}{L_n}R_nr_n$ a.e. for $x\in U_n$ and $g_{i,n}\to 0$ in $L^{\infty}_{\text{loc}}(\mathbb{R}^N)$ as $n\to\infty$, $\forall 1\leq i\leq k$.\\
		(3) The sequence $\{(\bar{u}_{1,n},\dots,\bar{u}_{\bar{k},n})\}$ has uniformly bounded Lip-seminorm in $U_n':=\Omega_n'/R_n$, that is,
		$$\max_{1\leq i\leq \bar{k}}\sup_{x,y\in U_n',\atop x\neq y}\frac{|\bar{u}_{i,n}(x)-\bar{u}_{i,n}(y)|}{|x-y|}\leq 128,\quad \forall n\geq 1.$$\\
		(4) There exists a globally Lipschitz continuous non-negative function $(\bar{u}_1,\dots,\bar{u}_{\bar{k}})$ in $\mathbb{R}^N$ such that
		$$(\bar{u}_{1,n},\dots,\bar{u}_{\bar{k},n})\to (\bar{u}_1,\dots,\bar{u}_{\bar{k}})\quad \text{in}~~C_{\text{loc}}(\mathbb{R}^N;\mathbb{R}^{\bar{k}})\quad \text{as}~~n\to\infty.$$\\
		(5) $\bar{u}_1(0)>0$ and $|\nabla \bar{u}_1(0)|>0$.\\
		(6) There exist $2\leq j_0\leq \bar{k}$ and $C>0$ independent of $n$ such that
		$$\overline{M}_{1,j_0;n}\geq C,\quad \frac{1}{2^{N-1}}\int_{\partial B_2}\bar{u}_{1,n}\geq C,\quad \frac{1}{2^{N-1}}\int_{\partial B_2}\bar{u}_{j_0,n}\geq C,\quad \forall n\geq 1.$$
	\end{lemma}
	\begin{proof}
		We divide the proof according to the value of $v_1(0)$.
		\par 
		\textbf{Case 1.} $v_1(0)>0$. In this case, we take $\bar{k}=k$ and $R_n=1,\forall n\geq 1$. Then by definition, $\bar{u}_{i,n}=v_{i,n}$, $g_{i,n}=\hat{f}_{i,n}$ and $\overline{M}_{i,j;n}=M_{i,j;n}$ for any $i\neq j$ and $n\geq 1$. Hence, by Lemma \ref{lem:pre-property-v-n}, $(\bar{u}_{1,n},\dots,\bar{u}_{\bar{k},n})$ solves \eqref{sys:bar-u-in}, and Points (1)--(4) follow. It remains to prove Points (5) and (6). We establish these via the following two claims.
		\par 
		\textbf{Claim (I).} $|\nabla v_1(0)|>0$.
		\par 
		Indeed, since $v_1(0)>0$, by Lemma \ref{lem:pre-property-v-n} (6), there exist $C>0$ and $0<r<1$ such that $v_{1,n}\geq C$ in $B_{2r}$ for any $n\geq1$. Combining this with \eqref{sys:v-n} yields that 
		$$-\Delta v_{j,n}\leq - a_{j1}C^{p_1}M_{j,1;n}v_{j,n}^{p_j}+1\quad \text{in}~~B_{2r},\quad 0\leq v_{j,n}\leq 256r+1\quad \text{on}~~\partial B_{2r}$$
		for any $2\leq j\leq k$ and $n\geq 1$. Then by Lemma \ref{lem:decay-Mu-p}, we obtain 
		\begin{align}\label{eq:bdd-M-j1n-v-jn}
			M_{1,j;n}v_{j,n}^{p_j}=M_{j,1;n}v_{j,n}^{p_j}\leq C_1\quad \text{in}~~B_r,\quad \forall 2\leq j\leq k,~~\forall n\geq 1,
		\end{align}
		where $C_1>0$ is a constant independent of $n$. Considering the equation for $v_{1,n}$, this provides the uniform boundedness of $\{\Delta v_{1,n}\}$ in $L^{\infty}(B_r)$. The standard elliptic estimates then yields that $\{v_{1,n}\}$ is bounded in $C^{1,\alpha}(\overline{B_{r/2}})$ for any $0<\alpha<1$. Thus, by the Ascoli-Arzel\`{a} theorem, up to a subsequence, $v_{1,n}\to v_1$ in $C^1(\overline{B_{r/2}})$ as $n\to\infty$. It then follows from Lemma \ref{lem:pre-property-v-n} (5) that $|\nabla v_1(0)|=\lim\limits_{n\to\infty}|\nabla v_{1,n}(0)|=1$. The claim is proved.
		\par 
		\textbf{Claim (II).} Point (6) holds true.
		\par 
		Indeed, since $\hat{f}_{i,n}\leq d_m\frac{\eta(x_n)}{L_n}r_n$, by Lemma \ref{lem:non-decrease-widetilde-H}, the fucntion 
		$$t\mapsto \widetilde{H}(v_{1,n},0,t)+|B_1|d_mr_n\frac{\eta(x_n)}{L_n}\frac{t^2}{2}\quad\text{is non-decreasing for }t\in (0,2].$$
		Combining this with $\lim\limits_{n\to\infty}r_n\frac{\eta(x_n)}{L_n}=0$ and $v_{1,n}\geq C$ in $B_{2r}$ for any $n\geq 1$, it follows that $\widetilde{H}(v_{1,n},0,2)\geq C_1$ for any $n\geq 1$, where $C_1>0$ is independent of $n$. Therefore, it suffices to prove the existence of $2\leq j_0\leq k$ such that 
		$$\limsup_{n\to\infty} \widetilde{H}(v_{j_0,n},0,2)>0,\quad \limsup_{n\to\infty} M_{1,j_0;n}>0.$$
		To this aim, we consider the set 
		$$\mathcal{J}:=\{2\leq j\leq k: \limsup_{n\to\infty} M_{1,j;n}>0\}.$$
		If $\mathcal{J}=\emptyset$, then by Lemma \ref{lem:pre-property-v-n} and Claim (I), we can deduce that $v_1$ is a nontrivial nonconstant harmonic function in $\mathbb{R}^N$, contradicting to the classical Liouville theorem. Hence, $\mathcal{J}\neq \emptyset$. Write $j_1=1$ and $\mathcal{J}=\{j_2,j_3,\dots,j_l\}$ for some $2\leq l\leq k$. Then $(v_{j_1,n},\dots, v_{j_l,n})$ solves 
		$$-\Delta v_{j_i,n}=h_{j_i,n}-v_{j_i,n}^{p_{j_i}}\sum_{s\neq i}^{l}a_{j_{i}j_{s}}M_{j_{i},j_{s};n}v_{j_{s},n}^{p_{j_s}}\quad \text{in}~~\Omega_n,\quad \forall 1\leq i\leq l,$$
		where $h_{j_i,n}:=\hat{f}_{j_{i},n}-v_{j_i,n}^{p_{j_i}}\sum_{s\notin \{1\}\cup \mathcal{J}}a_{j_{i}s}M_{j_{i},s;n}v_{s,n}^{p_{s}}$. Since $h_{j_1,n}\to 0$ in $L^{\infty}_{\text{loc}}(\mathbb{R}^N)$ as $n\to\infty$, we can deduce that $-\Delta v_{j_i}\leq  0$ and $-\Delta \widetilde{v}_{j_i}\geq 0$ in $\mathbb{R}^N$, where 
		$$\widetilde{v}_{j_i}:=v_{j_i}-\sum_{s\neq i}^l\frac{a_{j_{i}j_{s}}}{a_{j_{s}j_{i}}}v_{j_s}.$$
		\par 
		We assert that at least one of $2\leq i\leq l$ satisfies $\widetilde{H}(v_{j_i},0,2)>0$. 
		\par 
		Indeed, if not, then for any $2\leq i\leq l$, $\widetilde{H}(v_{j_i},0,2)=0$. By the monotonic property of $\widetilde{H}(v_{j_i,n},0,\cdot)$ (Lemma \ref{lem:non-decrease-widetilde-H}), it follows that $v_{j_i}\equiv 0$ in $\overline{B_2}$. This together with $-\Delta v_{j_1}\leq  0$ and $-\Delta \widetilde{v}_{j_1}\geq 0$ in $\mathbb{R}^N$ implies that $v_1=v_{j_1}$ is a non-negative harmonic function in $B_2$. By the strong maximum principle and $v_1(0)>0$, it follows that $v_1>0$ in $B_2$. By continuity, there exist $\delta>0$ and $1<r<2$ such that $v_{1,n}\geq \delta $ in $B_r$ for $n$ large. Hence, for any $2\leq i\leq l$ and any $x\in B_{(1+r)/2}$, the equation for $v_{j_i,n}$ yields
		$$-\Delta v_{j_i,n}\leq -\delta^{p_1}a_{j_ij_1} M_{j_{i},j_{1};n}v_{j_i,n}^{p_{j_i}}+\delta_n\quad \text{in}~~B_{(r-1)/2}(x),$$
		where $\delta_n:=\|\hat{f}_{j_i,n}\|_{L^{\infty}(B_r)}\to 0$ as $n\to\infty$. It then follows from Lemma \ref{lem:decay-Mu-p} that 
		$$\sup_{y\in B_{(r-1)/4}(x)}\delta^{p_1}a_{j_ij_1} M_{j_{i},j_{1};n}v_{j_i,n}^{p_{j_i}}(y)\leq \frac{C\max_{\overline{B_r}}v_{j_i,n}}{[(r-1)/4]^2}+\delta_n\to 0\quad \text{as}~~n\to\infty,$$
		where $C>0$ is a constant depending only on $N$.
		This implies 
		\begin{align}\label{eq:decay-to-0-Mvp}
			\sup_{x\in B_{(1+r)/2}}M_{j_{i},j_{1};n}v_{j_i,n}^{p_{j_i}}(x)\to 0\quad \text{as}~~n\to\infty,\quad \forall 2\leq i\leq l.
		\end{align}
		As a consequence, we deduce from the equation for $v_{1,n}$ that $v_1(0)=1$ and $|\nabla v_1(0)|=1$.
		Moreover, it follows from Lemma \ref{lem:pre-property-v-n} (6) that $|\nabla v_1|\leq 1$ in $B_{(1+r)/2}$. Take $\theta\in \mathbb{S}^{N-1}$ such that $\partial_{\theta}v_1(0)=1$. Then $\partial_{\theta}v_1$ is harmonic in $B_{(1+r)/2}$
		having a maximum point $0$. By the strong maximum principle, $\partial_{\theta}v_1\equiv 1$ in $B_{(1+r)/2}$. This with the fact $v_1(0)=1$ then implies that $v_1(x)=0$ for some $x\in \partial B_1$, in contradiction with $v_1>0$ in $B_2$. Hence, there exists $2\leq i\leq l$ such that $\widetilde{H}(v_{j_i},0,2)>0$. 
		\par 
		By the uniform convergence of $\{v_{j_i,n}\}$, we conclude that $j_0=j_i$ satisfies the desired properties.
		\par 
		\textbf{Case 2.} $v_1(0)=0$. Let 
		$$\mathcal{I}_1:=\{2\leq j\leq k: v_j(0)>0\},\quad \mathcal{I}_2:=\{2\leq j\leq k: v_j(0)=0\}.$$
		Since $v_1(0)=0$, by Lemma \ref{lem:pre-property-v-n} (1),  we have $\mathcal{I}_1\neq\emptyset$. 
		We claim that $\mathcal{I}_2\neq \emptyset$ as well. In fact, if $v_j(0)>0$ for any $2\leq j\leq k$, then as in the proof of Claim (I), we can deduce that $\{\Delta v_{1,n}\}$ is uniformly bounded in $L^{\infty}(B_r)$ for some $r>0$. Hence, up to a subsequence, $v_{1,n}\to v_1$ in $C^1(\overline{B_{r/2}})$ as $n\to\infty$. Combining this with Lemma \ref{lem:pre-property-v-n} (5) yields that $|\nabla v_1(0)|=\lim\limits_{n\to\infty}|\nabla v_{1,n}(0)|=1$. On the other hand, note that $v_1$ is non-negative, we see that  $0$ is a local minimum point of $v_1$ and hence $\nabla v_1(0)=0$. This  contradicts $|\nabla v_1(0)|=1$. The claim is proved. We remark that this in particular implies \textbf{Case 2} does not occur if $k=2$ and hence the main result holds true for $k=2$ (by \textbf{Case 1}).
		\par 
		Next, we focus on the case when $k \geq 3$ (and $v_1(0)=0$) and prove the result by induction on $k$. Let $\widetilde{R}_{n}:=v_{1,n}(0)+\sum_{j\in \mathcal{I}_2}v_{j,n}(0)>0$. Then $\widetilde{R}_{n}\to 0$ as $n\to\infty$. We write $j_1=1$ and $\mathcal{I}_2=\{j_2,\dots,j_l\}$ with $2\leq l\leq k-1$. Define 
		$$w_{j_i,n}(x):=\frac{1}{\widetilde{R}_{n}}v_{j_i,n}(\widetilde{R}_nx),\quad\text{for}~~ x\in \Omega_n/\widetilde{R}_n,\quad \forall 1\leq i\leq l.$$
		Then by definition, $\sum_{i=1}^{l}w_{j_i,n}(0)=1$. It follows from Lemma \ref{lem:pre-property-v-n} that $(w_{j_1,n},\dots,w_{j_l,n})$ is a positive solution of 
		$$-\Delta w_{j_i,n}=\widetilde{g}_{j_i,n}(x)-w_{j_i,n}^{p_{j_i}}\sum_{s\neq i}^{l}a_{j_{i}j_{s}}\widetilde{M}_{j_{i},j_{s};n}w_{j_{s},n}^{p_{j_s}}\quad \text{in}~~\Omega_n/\widetilde{R}_n,\quad \forall 1\leq i\leq l,$$
		where $\widetilde{M}_{j_{i},j_{s};n}:=\widetilde{R}_{n}^{p_{j_i}+p_{j_s}+1}M_{j_{i},j_{s};n}$ and 
		$$\widetilde{g}_{j_i,n}(x):=\widetilde{R}_{n}\hat{f}_{j_{i},n}(\widetilde{R}_{n}x)-\widetilde{R}_{n}[v_{j_i,n}(\widetilde{R}_{n}x)]^{p_{j_i}}\sum_{s\in \mathcal{I}_1}a_{j_{i}s}M_{j_{i},s;n}[v_{s,n}(\widetilde{R}_{n}x)]^{p_{s}}.$$ 
		Hence, $(w_{j_1,n},\dots,w_{j_l,n})$ satisfies \eqref{sys:bar-u-in} with $R_n=\widetilde{R}_n$ and $\overline{M}_{j_{i},j_{s};n}=\widetilde{M}_{j_{i},j_{s};n}$. 
		Note that by the definition of $\mathcal{I}_1$, we have that there exist $r>0$ and $C>0$ such that  $v_{s,n}\geq C$ in $\overline{B_{2r}}$ for any $s\in \mathcal{I}_1$ and $n\geq 1$. Thus as in the proof of \eqref{eq:bdd-M-j1n-v-jn}, we can deduce that $\sum_{s\in \mathcal{I}_1}M_{j_{i},s;n}v_{j_i,n}^{p_{j_i}}\leq C_1$ in $B_r$ for any $n\geq 1$, where $C_1>0$ is independent of $n$. Combining this with $\widetilde{R}_n\to 0$ as $n\to\infty$ yields that 
		$$\widetilde{g}_{j_i,n}\to 0\quad \text{in}~~L^{\infty}_{\text{loc}}(\mathbb{R}^N)\quad \text{as}~~n\to\infty,\quad \forall 1\leq i\leq l.$$
		As a consequence, Points (1)--(4) holds true for $(w_{j_1,n},\dots,w_{j_l,n})$ and $\widetilde{g}_{j_i,n}$.
		\par 
		Let $(w_{j_1},\dots, w_{j_l})$ be the limiting function of $(w_{j_1,n}, \dots, w_{j_l,n})$ as $n\to\infty$. If $w_{j_1}(0)>0$ or $l=2$, then as in the proof of \textbf{Case 1}, the result holds by setting $\bar{k}=l$, $R_n=\widetilde{R}_n$ and $(\bar{u}_{1,n},\dots,\bar{u}_{\bar{k},n})=(w_{j_1,n}, \dots, w_{j_l,n})$. If $w_{j_1}(0)=0$ and $l\geq 3$, then we can proceed analogously to generate new sequences. Note that $l\leq k-1$ and the result holds for $k=2$. The final conclusion can be derived via a finite sequence of operations.
	\end{proof}
	
	\subsection{Applications of the monotonicity formula}
	
	With Lemma \ref{lem:new-u-n-property} in hand, in this subsection, we apply the monotonicity formula (Theorem \ref{thm:ACF-formula}) to complete the proof of Theorem \ref{thm:interior-lip}.
	\par 
	First of all, without loss of generality, we assume $j_0=2$ in Lemma \ref{lem:new-u-n-property} (6). Moreover, to simplify the notation, we may assume that $\bar{k} = k$ and $R_n = 1$ for all $n \geq 1$ in Lemma~\ref{lem:new-u-n-property}; that is, we assume the sequence $\{\bm{v}_n\}$ itself satisfies the properties stated in Lemma~\ref{lem:new-u-n-property}. Then $U_n'=\Omega_n'$ and $U_n=\Omega_n$. We also denote 
	\begin{align}\label{def:D-n}
		D_n:=d_m\frac{\eta(x_n)}{L_n}R_nr_n=d_m\frac{\eta(x_n)}{L_n}r_n,\quad \forall n\geq 1.
	\end{align}
	Since $D_n\to 0$ as $n\to\infty$, by Lemma \ref{lem:new-u-n-property} (6), Lemma \ref{lem:non-decrease-widetilde-H} and H\"{o}lder's inequality, there exist $C_0>0$ and $0<C_0'<\frac{1}{4}C_0^{1/2}\sigma_{N-1}^{-1/2}$ independent of $n$ such that up to a subsequence,
	\begin{align}\label{eq:H-12-big-C}
		H_i(\bm{v}_n,0,r)\geq C_0,\quad \forall r\in [2,\tilde{l}_n],~~\forall i=1,2,
	\end{align}
	where 
	\begin{align}\label{def:tilde-l-n}
		\tilde{l}_n:=\min\left\{\sqrt{\frac{C_0'}{D_n}},\frac{1}{R_nr_n}\right\}=\min\left\{\sqrt{\frac{C_0'}{D_n}},\frac{1}{r_n}\right\}\to +\infty\quad \text{as}~~n\to\infty.
	\end{align}
	\par 
	Next, for $i=1,2$ and $r\geq 2$, we denote
	\begin{align}
		\begin{split}
			J_{i,n}(r)&:=\int_{B_r}\left(|\nabla v_{i,n}|^2+v_{i,n}^{p_i+1}\sum_{j\neq i}^ka_{ij}M_{i,j;n}v_{j,n}^{p_j}-D_nv_{i,n}\right)|x|^{2-N},\\
			\Lambda_{i,n}(r)&:=\frac{r^2\int_{\partial B_r}|\nabla_{\theta}v_{i,n}|^2+v_{i,n}^{p_i+1}\sum_{j\neq i}^ka_{ij}M_{i,j;n}v_{j,n}^{p_j}-D_nv_{i,n}}{\int_{\partial B_r}v_{i,n}^2}\quad \text{if}~~\int_{\partial B_r}v_{i,n}^2>0.
		\end{split}
	\end{align}
	We observe that $J_{i,n}$ and $\Lambda_{i,n}$ admit a uniform positive lower bound at $r=2$.
	
	\begin{lemma}\label{lem:interior-lip-J-Lambda}
		There exists $C>0$ independent of $n$ such that
		$$\Lambda_{i,n}(2)\geq C,\quad J_{i,n}(2)\geq C,\quad i=1,2\quad \text{for $n$ large}.$$
	\end{lemma}
	\begin{proof}
		For $n\geq 1$ and $1\leq i\leq k$, let 
		$$w_{i,n}(x):=v_{i,n}(2x),\quad x\in B_2\subseteq \Omega_n/2.$$
		Then 
		$$J_{i,n}(2)=\int_{B_1}\left(|\nabla w_{i,n}|^2+w_{i,n}^{p_i+1}\sum_{j\neq i}^k4a_{ij}M_{i,j;n}w_{j,n}^{p_j}-4D_nw_{i,n}\right)|x|^{2-N},$$
		and 
		$$\Lambda_{i,n}(2)=\frac{\int_{\partial B_1}|\nabla_{\theta}w_{i,n}|^2+w_{i,n}^{p_i+1}\sum_{j\neq i}^k4a_{ij}M_{i,j;n}w_{j,n}^{p_j}-4D_nw_{i,n}}{\int_{\partial B_1}w_{i,n}^2}.$$
		By \eqref{eq:H-12-big-C}, we have $H_i(\bm{w}_n,0,1)\geq C_0>0$ for $i=1,2$ and $n\geq 1$. 
		Since $D_n\to 0$ as $n\to\infty$, this implies that
		$$\int_{B_1}D_nw_{i,n}|x|^{2-N}\to 0,\quad \frac{\int_{\partial B_1}D_nw_{i,n}}{\int_{\partial B_1}w_{i,n}^2}\to 0,\quad \text{as}~~n\to\infty,\quad i=1,2.$$
		Combining these with the fact that $M_{1,2;n} \geq C > 0$ (by Lemma \ref{lem:new-u-n-property} (6)), we apply Lemma \ref{lem:property-global-lip-sol} (3) to deduce
		\begin{align*}
			\liminf_{n\to\infty}J_{i,n}(2)\geq 	\liminf_{n\to\infty}\int_{B_1}\left(|\nabla w_{i,n}|^2+w_{i,n}^{p_i+1}\sum_{j\neq i}^k4a_{ij}M_{i,j;n}w_{j,n}^{p_j}\right)|x|^{2-N}>0,
		\end{align*}
		and 
		\begin{align*}
			\liminf_{n\to\infty}\Lambda_{i,n}(2)\geq \liminf_{n\to\infty}\frac{\int_{\partial B_1}|\nabla_{\theta}w_{i,n}|^2+w_{i,n}^{p_i+1}\sum_{j\neq i}^k4a_{ij}M_{i,j;n}w_{j,n}^{p_j}}{\int_{\partial B_1}w_{i,n}^2}>0,
		\end{align*}
		where $i=1,2$. The proof is finished.
	\end{proof}
	\par 
	Thirdly, having \eqref{eq:H-12-big-C} and Lemma \ref{lem:interior-lip-J-Lambda} in mind, we define
	\begin{equation}\label{def:tilde-r-n}
		\tilde{r}_n:=\sup\left\{r\in [2,\hat{r}_n):~
		\begin{split}
			&D_ns^2\sigma_{N-1}^{1/2}H_i(\bm{v}_n,0,s)^{-1/2}\leq (N-2)^2/4,~~~\Lambda_{i,n}(s)>0,\\
			&H_i(\bm{v}_n,0,s)\geq C_0,~~~D_n\int_{2}^{s}tH_n(t)^{-1/2}dt\leq (N-2)^2/4,\\
			&\forall~ 2\leq s\leq r,~\forall i=1,2
		\end{split}
		\right\},
	\end{equation}
	where $H_n(t):=\min_{i\in \{1,2\}}H_i(\bm{v}_n,0,t)$ and
	\begin{align}\label{def:hat-r-n}
		\hat{r}_n:=\frac{\bar{r}_n}{R_n}=\bar{r}_n\to+\infty\quad \text{as}~~n\to\infty,
	\end{align}
	(recall $\bar{r}_n$ is defined in \eqref{def:bar-r-n}). Note that by the standard elliptic regularity theorems, $\bm{v}_n\in C^{1,\alpha}_{\text{loc}}(\Omega_n;\mathbb{R}^k)$ for some $\alpha\in (0,1)$ and hence $\Lambda_{i,n}$ and  $H_i(\bm{v}_n,0,\cdot)$ are continuous in $[2,\hat{r}_n)$ for any fixed $n\geq 1$ and $i=1,2$. Note also that $D_n\to 0$ as $n\to\infty$. As a consequence, $\tilde{r}_n\geq 2$ is well-defined. In the following we show that this quantity is almost maximal for deriving the following monotonicity formula in a uniform way.
	\begin{proposition}\label{prop:monotonicity-formula-v}
		There exist $C_1,C_2>0$ independent of $n$ such that for every $n\geq 1$, one has $J_{i,n}(r)\geq C_1$ for any $i=1,2$ and any $r\in [2,\tilde{r}_n)$; moreover, the function
		$$r\mapsto \frac{J_{1,n}(r)J_{2,n}(r)}{r^4}e^{-C_2M_{1,2;n}^{-1/(2p+2)}r^{-1/(p+1)}+C_2D_n\int_{2}^{r}tH_n(t)^{-1/2}dt}$$
		is monotone non-decreasing for $r\in [2,\tilde{r}_n)$, where $p:=\max\{p_1,p_2\}\geq 1$ and $H_n(t)=\min_{i\in \{1,2\}}H_i(\bm{v}_n,0,t)$.
	\end{proposition}
	\begin{proof}
		For $i=1,2$, since by the defintion of $\tilde{r}_n$,
		\begin{align*}
			\frac{d}{dr}J_{i,n}(r)
			&=r^{2-N}\int_{\partial B_r}\left(|\nabla v_{i,n}|^2+v_{i,n}^{p_i+1}\sum_{j\neq i}^ka_{ij}M_{i,j;n}v_{j,n}^{p_j}-D_nv_{i,n}\right)\\
			&\geq r^{-N}\Lambda_{i,n}(r)\int_{\partial B_r}v_{i,n}^2>0
		\end{align*}
		for any $r\in [2,\tilde{r}_n)$, we see that $J_{i,n}$ is monotone increasing in $[2,\tilde{r}_n)$. Combining this with Lemma \ref{lem:interior-lip-J-Lambda} yields that there exists $C_1>0$ independent of $n$ such that $J_{i,n}(r)\geq C_1$ for any $r\in [2,\tilde{r}_n)$. On the other hand, the defintion of $\tilde{r}_n$ allows us to  apply Theorem \ref{thm:ACF-formula} to deduce that there exists $C_2>0$ independent of $n$ such that the function
		$$r\mapsto \frac{J_{1,n}(r)J_{2,n}(r)}{r^4}e^{-C_2M_{1,2;n}^{-1/(2p+2)}r^{-1/(p+1)}+C_2D_n\int_{2}^{r}tH_n(t)^{-1/2}dt}$$
		is monotone non-decreasing for $r\in [2,\tilde{r}_n)$. The proof is finished.
	\end{proof}
	\par 
	Finally, we derive that in this interval $[2,\tilde{r}_n]$, one can find a point $s_n$ such that the quantity $H_1(\bm{v}_n,0,s_n)H_2(\bm{v}_n,0,s_n)s_n^{-4}$ tends to zero as $n\to\infty$. 
	\begin{remark}\label{rmk:uniform-bdd}
		Note that due to the uniform boundedness of $\{\nabla\bm{v}_n\}$ in $\Omega_n'$, we have that $H_i(\bm{v}_n,0,r)r^{-2}$ is uniformly bounded for $r\in [2,\tilde{r}_n]$ and $n\geq 1$, where $i=1,2$.
	\end{remark}
	
	\begin{lemma}\label{lem:H-1-2-divide-s-4-to0}
		There exists $s_n\in [2,\tilde{r}_n]$ such that up to a subsequence, there holds
		$$\lim\limits_{n\to\infty}\frac{H_1(\bm{v}_n,0,s_n)H_2(\bm{v}_n,0,s_n)}{s_n^4}=0.$$
	\end{lemma}
	\begin{proof}
		We argue by contradiction. By Remark \ref{rmk:uniform-bdd}, we can suppose that there exists $C>0$ independent of $n$ such that
		\begin{align}\label{eq:H1-H2-lower-bdd}
			\frac{H_1(\bm{v}_n,0,s)}{s^2}\geq C,\quad \frac{H_2(\bm{v}_n,0,s)}{s^2}\geq C,\quad \forall s\in [2,\tilde{r}_n]
		\end{align}
		for large $n$. We divide  the proof into four steps.
		\par
		\textbf{Step 1.} Up to a subsequence,
		$\lim\limits_{n\to\infty}\frac{\tilde{r}_n}{\hat{r}_n}=0$, where $\hat{r}_n>0$ is defined in \eqref{def:hat-r-n}. In particular, $\tilde{r}_n<\frac{1}{2}\hat{r}_n$ for any $n\geq 1$.
		\par
		Indeed, suppose by contradiction that there exists $\varepsilon_0\in (0,1)$ such that $\frac{\tilde{r}_n}{\hat{r}_n}\geq \varepsilon_0$ for any $n\geq 1$. Choose $s_n:=\varepsilon_0\hat{r}_n\in [2,\tilde{r}_n]$. Then
		\begin{align*}
			&\frac{H_1(\bm{v}_n,0,s_n)+H_2(\bm{v}_n,0,s_n)}{(s_n)^2}=\frac{H_1(\bm{v}_n,0,\varepsilon_0\bar{r}_n)+H_2(\bm{v}_n,0,\varepsilon_0\bar{r}_n)}{(\varepsilon_0\bar{r}_n)^2}\\
			=&\eta(x_n)^2\frac{H_1(u_{i,n},x_n,\varepsilon_0s_n'/2)+H_2(u_{i,n},x_n,\varepsilon_0s_n'/2)}{L_n^2(\varepsilon_0s_n'/2)^2}\\
			\leq& \frac{\bar{l}^2}{L_n^2(\varepsilon_0/2)^2}\left[H_1(u_{i,n},x_n,\varepsilon_0s_n'/2)+H_2(u_{i,n},x_n,\varepsilon_0s_n'/2)\right]\to 0\quad \text{as}~~n\to\infty,
		\end{align*}
		where $s_n':=\text{dist}(x_n,\partial B_2)$ and $\bar{l}=\|\nabla\eta\|_{L^{\infty}(B_3)}$. This contradicts \eqref{eq:H1-H2-lower-bdd}. \textbf{Step 1} is proved.
		\par
		\textbf{Step 2.} There holds  
		$$\lim\limits_{n\to\infty}D_n\tilde{r}_n^2H_i(\bm{v}_n,0,\tilde{r}_n)^{-1/2}=0,\quad \lim\limits_{n\to\infty}D_n\int_{2}^{\tilde{r}_n}tH_i(\bm{v}_n,0,t)^{-1/2}dt=0,\quad \forall i=1,2.$$
		\par 
		Indeed, since $D_n\tilde{r}_n\leq D_n/r_n\to 0$ as $n\to\infty$, by \eqref{eq:H1-H2-lower-bdd}, we deduce that
		$$D_n\tilde{r}_n^2H_i(\bm{v}_n,0,\tilde{r}_n)^{-1/2}\leq \frac{D_n}{r_n}C^{-1/2}\to 0,$$
		and 
		$$D_n\int_{2}^{\tilde{r}_n}tH_i(\bm{v}_n,0,t)^{-1/2}dt\leq D_nC^{-1/2}(\tilde{r}_n-2)\leq D_nC^{-1/2}(1/r_n-2)\to 0$$
		as $n\to\infty$, where $i=1,2$.
		\par 
		\textbf{Step 3.} There exists $\varepsilon_0'>0$ independent of $n$ such that
		\begin{align*}
			\Lambda_{1,n}(\tilde{r}_n)\geq \varepsilon_0',\quad \Lambda_{2,n}(\tilde{r}_n)\geq \varepsilon_0'\quad \text{for $n$ large.}
		\end{align*}
		\par
		Indeed, let us define
		\begin{align*}
			\tilde{v}_{1,n}(x):=\frac{1}{\sqrt{H_1(\bm{v}_n,0,\tilde{r}_n)}}v_{1,n}(\tilde{r}_nx),\quad \tilde{v}_{2,n}(x):=\frac{1}{\sqrt{H_2(\bm{v}_n,0,\tilde{r}_n)}}v_{2,n}(\tilde{r}_nx),
		\end{align*}
		where $x\in \frac{1}{\tilde{r}_n}\Omega_n$. By Lemma \ref{lem:new-u-n-property}, \textbf{Step 1} and the fact that $B_{\hat{r}_n}\subseteq \Omega_n$, we see that $(\tilde{v}_{1,n},\tilde{v}_{2,n})$ is well-defined in $B_2$ and is a weak solution to
		\begin{align*}
			\left\{
			\begin{array}{ll}
				-\Delta \tilde{v}_{1,n}&\leq D_n\frac{\tilde{r}_n^2}{\sqrt{H_1(\bm{v}_n,0,\tilde{r}_n)}}- 	\widetilde{M}_{1,2;n}\tilde{v}_{1,n}^{p_1}\tilde{v}_{2,n}^{p_2}\quad\text{in}~~B_2,\\
				-\Delta \tilde{v}_{2,n}&\leq D_n\frac{\tilde{r}_n^2}{\sqrt{H_2(\bm{v}_n,0,\tilde{r}_n)}}-	\widetilde{M}_{2,1;n}\tilde{v}_{2,n}^{p_2}\tilde{v}_{1,n}^{p_1}\quad\text{in}~~B_2,
			\end{array}\right.
		\end{align*}
		where 
		\begin{align*}
			\widetilde{M}_{1,2;n}&:=a_{12}M_{1,2;n}\tilde{r}_n^2H_1(\bm{v}_n,0,\tilde{r}_n)^{(p_1-1)/2}H_2(\bm{v}_n,0,\tilde{r}_n)^{p_2/2},\\
			\widetilde{M}_{2,1;n}&:=a_{21}M_{1,2;n}\tilde{r}_n^2H_2(\bm{v}_n,0,\tilde{r}_n)^{(p_2-1)/2}H_1(\bm{v}_n,0,\tilde{r}_n)^{p_1/2}.
		\end{align*}
		On the one hand, by \textbf{Step 2}, we have 
		$$D_n\frac{\tilde{r}_n^2}{\sqrt{H_i(\bm{v}_n,0,\tilde{r}_n)}}\to 0\quad \text{as}~~n\to\infty,\quad \forall i=1,2.$$
		Moreover, by the continuity of $H_i(\bm{v}_n,0,\cdot)$, the definition of $\tilde{r}_n$ and Lemma \ref{lem:new-u-n-property} (6), it follows that 
		\begin{align*}
			\widetilde{M}_{1,2;n}\geq C,\quad 	\widetilde{M}_{2,1;n}\geq C
		\end{align*}
		for some constant $C>0$ independent of $n$. On the other hand, by Lemma \ref{lem:new-u-n-property} (1), (3) and \eqref{eq:H1-H2-lower-bdd}, it is easy to see that $\{(\tilde{v}_{1,n},\tilde{v}_{2,n})\}$ is bounded in $[W^{1,\infty}(B_2)]^2$. Combining these with the fact $\int_{\partial B_1}\tilde{v}_{i,n}^2=1(i=1,2)$, we deduce from Lemma \ref{lem:property-global-lip-sol} (3) that
		$$\liminf_{n\to\infty}\int_{\partial B_1}|\nabla_{\theta}\tilde{v}_{1,n}|^2+	\widetilde{M}_{1,2;n} \tilde{v}_{1,n}^{p_1+1}\tilde{v}_{2,n}^{p_2}>0,\quad \liminf_{n\to\infty}\int_{\partial B_1}|\nabla_{\theta}\tilde{v}_{2,n}|^2+	\widetilde{M}_{2,1;n} \tilde{v}_{2,n}^{p_2+1}\tilde{v}_{1,n}^{p_1}>0.$$
		As a consequence of H\"{o}lder inequality and \textbf{Step 2}, it follows that
		\begin{align*}
			\liminf_{n\to\infty}\Lambda_{1,n}(\tilde{r}_n)
			&=\liminf_{n\to\infty}\frac{\tilde{r}_n^2\int_{\partial B_{\tilde{r}_n}}|\nabla_{\theta}v_{1,n}|^2+v_{1,n}^{p_1+1}\sum_{j\neq i}^ka_{1j}M_{1,j;n}v_{j,n}^{p_j}-D_nv_{1,n}}{\int_{\partial B_{\tilde{r}_n}}v_{1,n}^2}\\
			&\geq \liminf_{n\to\infty}\frac{\tilde{r}_n^2\int_{\partial B_{\tilde{r}_n}}|\nabla_{\theta}v_{1,n}|^2+v_{1,n}^{p_1+1}a_{12}M_{1,2;n}v_{2,n}^{p_2}-D_nv_{1,n}}{\int_{\partial B_{\tilde{r}_n}}v_{1,n}^2}\\
			&\geq \liminf_{n\to\infty}\int_{\partial B_1}|\nabla_{\theta}\tilde{v}_{1,n}|^2+	\widetilde{M}_{1,2;n} \tilde{v}_{1,n}^{p_1+1}\tilde{v}_{2,n}^{p_2}-\limsup_{n\to\infty}\sigma_{N-1}^{1/2}D_n\frac{\tilde{r}_n^2}{\sqrt{H_1(\bm{v}_n,0,\tilde{r}_n)}}\\
			&>0.
		\end{align*}
		Similarly, $\liminf_{n\to\infty}\Lambda_{2,n}(\tilde{r}_n)>0$. Thus, \textbf{Step 3} is proved.
		\par 
		\textbf{Step 4.} There holds $\lim\limits_{n\to\infty}\tilde{r}_n=+\infty$. In particular, it follows from \eqref{eq:H1-H2-lower-bdd} that $\lim\limits_{n\to\infty}H_i(\bm{v}_n,0,\tilde{r}_n)=+\infty$ for $i=1,2$.
		\par 
		Once this is proved, then by \textbf{Steps 1--3} and the continuity of $\Lambda_{i,n}$ and $H_i(\bm{v}_n,0,\cdot)$, we have for large $n\geq 1$, there exists $\varepsilon_n>0$ such that $\tilde{r}_n+\varepsilon_n<\hat{r}_n$, and $\Lambda_{i,n}(r)>0$, $H_i(\bm{v}_n,0,r)\geq C_0$,
		$$D_n\int_{2}^{r}tH_n(t)^{-1/2}dt\leq (N-2)^2/4,\quad D_nr^2\sigma_{N-1}^{1/2}H_i(\bm{v}_n,0,r)^{-1/2}\leq (N-2)^2/4$$
		for any $r\in [2,\tilde{r}_n+\varepsilon_n]$ and $i=1,2$, contradicting to the definition of $\tilde{r}_n$.
		\par 
		It remains to prove  \textbf{Step 4}. We argue by contradiction. Suppose that up to a subsequence, there exists $C>0$ such that $\tilde{r}_n\leq C$ for any $n\geq 1$. Recall that $\tilde{l}_n\to +\infty$ and $\hat{r}_n\to +\infty$ as $n\to\infty$, by \eqref{eq:H-12-big-C}, we can deduce that up to a subsequence, 
		$$\tilde{r}_n+1<\hat{r}_n,\quad H_i(\bm{v}_n,0,r)\geq C_0,\quad \forall r\in [2,\tilde{r}_n+1],\quad\forall n\geq 1,\forall i=1,2.$$
		This together with \textbf{Steps 2,3} yields that for any $n\geq 1$, there exists $\varepsilon_n>0$ such that $\tilde{r}_n+\varepsilon_n<\hat{r}_n$, and $\Lambda_{i,n}(r)>0$, $H_i(\bm{v}_n,0,r)\geq C_0$,
		$$D_n\int_{2}^{r}tH_n(t)^{-1/2}dt\leq (N-2)^2/4,\quad D_nr^2\sigma_{N-1}^{1/2}H_i(\bm{v}_n,0,r)^{-1/2}\leq (N-2)^2/4$$
		for any $r\in [2,\tilde{r}_n+\varepsilon_n]$ and $i=1,2$, contradicting to the definition of $\tilde{r}_n$. Hence, 	\textbf{Step 4} is proved.
		\par 
		We have finished the proof of Lemma \ref{lem:H-1-2-divide-s-4-to0}.
	\end{proof}
	\par 
	Now we are in position to complete the proof of Theorem \ref{thm:interior-lip}.
	\begin{proof}[Conclusion of the proof of Theorem \ref{thm:interior-lip}]
		Observe that by Lemma \ref{lem:new-u-n-property} and the definition of $\tilde{r}_n$, there exist $0<C_1'<C_2'<+\infty$ independent of $n$ such that
		$$C_1'\leq e^{-C_2M_{1,2;n}^{-1/(2p+2)}r^{-1/(p+1)}+C_2D_n\int_{2}^{r}tH_n(t)^{-1/2}dt}\leq C_2',\quad \forall r\in [2,\tilde{r}_n],\quad \forall n\geq 1.$$
		On the one hand, by Lemma \ref{lem:interior-lip-J-Lambda}, we have 
		$$\frac{J_{1,n}(2)J_{2,n}(2)}{2^4}e^{-C_2M_{1,2;n}^{-1/(2p+2)}2^{-1/(p+1)}+C_2D_n\int_{2}^{2}tH_n(t)^{-1/2}dt}\geq C'',\quad \forall n\geq 1,$$
		where $C''>0$  is a constant independent of $n$.
		On the other hand, it follows from Lemmas \ref{lem:new-u-n-property} (3), \ref{lem:est:J-i(r)} (2) and \ref{lem:H-1-2-divide-s-4-to0} that
		\begin{align*}
			&\frac{J_{1,n}(s_n)J_{2,n}(s_n)}{s_n^4}e^{-C_2M_{1,2;n}^{-1/(2p+2)}s_n^{-1/(p+1)}+C_2D_n\int_{2}^{s_n}tH_n(t)^{-1/2}dt}
			\leq C_2'\frac{J_{1,n}(s_n)J_{2,n}(s_n)}{s_n^4}\\
			\leq& C_2'\left\{128\sigma_{N-1}^{1/2}\left[\frac{H_1(\bm{v}_n,0,s_n)}{s_n^2}\right]^{1/2}+\frac{N-2}{2}\frac{H_1(\bm{v}_n,0,s_n)}{s_n^2}\right\}\\
			&\cdot\left\{128\sigma_{N-1}^{1/2}\left[\frac{H_2(\bm{v}_n,0,s_n)}{s_n^2}\right]^{1/2}+\frac{N-2}{2}\frac{H_2(\bm{v}_n,0,s_n)}{s_n^2}\right\}\rightarrow 0\quad \text{as}~~n\to\infty.
		\end{align*}
		Therefore, we deduce from Proposition \ref{prop:monotonicity-formula-v} that $0<C''\to 0$ as $n\to\infty$, which is a contradiction. This finishes the proof of Theorem \ref{thm:interior-lip}.
	\end{proof}
	
	\section{Uniform global Lipschitz bounds}\label{sec:5}
	\setcounter{section}{5}
	\setcounter{equation}{0}
	
	In this section, we prove Theorem \ref{thm:global-lip}. The proof relies on blow-up analysis. Specifically, we demonstrate how to apply Theorem \ref{thm:interior-lip} to derive such a uniform global result. To this end, we first need to construct a suitable blow-up sequence and then carefully analyze the behavior of this sequence near the boundary of the domain.
	\par 
	Before proving Theorem \ref{thm:global-lip}, we need to establish the following result regarding the boundary behavior of solutions to system \eqref{Sys:main-interior-general}.
	\begin{lemma}\label{lem:boundary-est}
		Under the assumptions of Theorem \ref{thm:global-lip}, there exists $M_1>0$ independent of $n$ such that for any $x\in \overline{\Omega}$ and $y\in \partial\Omega$ with $x\neq y$, one has 
		$$\frac{|u_{i,n}(x)-u_{i,n}(y)|}{|x-y|}\leq M_1,\quad \forall 1\leq i\leq k.$$
	\end{lemma}
	\begin{proof}
		For $1\leq i\leq k$, we denote by $\Phi_{i,n}\in H^1_{\text{loc}}(\Omega)\cap C(\overline{\Omega})$ the solution of 
		\begin{align*}
			-\Delta w=f_{i,n}(x,u_{i,n})\quad \text{in}~~\Omega,\quad w=\varphi_{i,n}\quad \text{on}~~\partial\Omega.
		\end{align*}
		Then by (H1), (H2), and the standard elliptic regularity theorems (see e.g. \cite{Gilbarg-Trudinger2001-book,Borsuk1998}), we have that $\Phi_{i,n}\in C^1(\overline{\Omega})$; moreover, there exists a constant $C_1>0$ independent of $n$ such that $\|\Phi_{i,n}\|_{C^1(\overline{\Omega})}\leq C_1$. Since $-\Delta u_{i,n}\leq f_{i,n}(x,u_{i,n})$ in $\Omega$ and $u_{i,n}=\varphi_{i,n}$ on $\partial\Omega$, the weak comparison principle yields that $\Phi_{i,n}\geq u_{i,n}$ in $\Omega$. On the other hand, note that 
		$$-\Delta \widehat{u}_{i,n}\geq f_{i,n}(x,u_{i,n})-\sum_{j\neq i}\frac{a_{ij}}{a_{ji}}f_{j,n}(x,u_{j,n})=-\Delta\Psi_{i,n}\quad \text{in}~~\Omega,$$
		where $\widehat{u}_{i,n}=u_{i,n}-\sum_{j\neq i}\frac{a_{ij}}{a_{ji}}u_{j,n}$ and $\Psi_{i,n}:=\Phi_{i,n}-\sum_{j\neq i}\frac{a_{ij}}{a_{ji}}\Phi_{j,n}$. By the comparison principle again, it follows that $\Psi_{i,n}\leq \widehat{u}_{i,n}\leq u_{i,n}$ in $\Omega$.
		\par 
		Now, we denote $M_1:=C_1+\sum_{j\neq i}\frac{a_{ij}}{a_{ji}}C_1>0$. Let $x\in \overline{\Omega}$ and $y\in \partial\Omega$ with $x\neq y$. If $\varphi_{i,n}(y)=0$, then 
		$$\frac{|u_{i,n}(x)-u_{i,n}(y)|}{|x-y|}=\frac{u_{i,n}(x)}{|x-y|}\leq \frac{\Phi_{i,n}(x)-\Phi_{i,n}(y)}{|x-y|}\leq C_1\leq M_1;$$
		if $\varphi_{i,n}(y)>0$, then by \eqref{Condi:boundary-disjoint}, $\Psi_{i,n}(y)=\varphi_{i,n}(y)-\sum_{j\neq i}\frac{a_{ij}}{a_{ji}}\varphi_{j,n}(y)=\varphi_{i,n}(y)$, hence,
		\begin{align*}
			\frac{\Psi_{i,n}(x)-\Psi_{i,n}(y)}{|x-y|}\leq \frac{u_{i,n}(x)-u_{i,n}(y)}{|x-y|}\leq \frac{\Phi_{i,n}(x)-\Phi_{i,n}(y)}{|x-y|}.
		\end{align*}
		As a consequence, we derive that $\frac{|u_{i,n}(x)-u_{i,n}(y)|}{|x-y|}\leq M_1$. The proof is finished.
	\end{proof}
	\par 
	We aim to prove Theorem \ref{thm:global-lip} via blow-up analysis. Consider a sequence of weak solutions $\{\bm{u}_n\}$ to system \eqref{Sys:main-interior-general}-\eqref{Boundary:main} that satisfies the assumptions of Theorem \ref{thm:global-lip}. Assume, by contradiction, that the conclusion of Theorem \ref{thm:global-lip} fails for $\{\bm{u}_n\}$. Then there exists a subsequence of $\{\bm{u}_n\}$ (still denoted by $\{\bm{u}_n\}$) such that
	\begin{align*}
		W_n := \max_{1 \leq i \leq k} \sup_{x \in \Omega} |\nabla u_{i,n}(x)| \to +\infty \quad \text{as } n \to \infty.
	\end{align*}
	Up to relabelling, we may assume that there exists a sequence $\{x_n\} \subseteq \Omega$ such that
	\begin{align*}
		L_n := |\nabla u_{1,n}(x_n)| \in \left[ \frac{n}{n+1} W_n, W_n \right] \quad \text{for all } n \geq 1.
	\end{align*}
	It then follows that $L_n \to +\infty$ as $n \to \infty$.
	\par 
	For any $n\geq 1$, let us define 
	\begin{align}
		v_{i,n}(x):=\frac{1}{\bar{r}_n}u_{i,n}(x_n+\bar{r}_nx),\quad \forall x\in \overline{\Omega_n},\quad \forall 1\leq i\leq k,
	\end{align}
	where $\Omega_n:=\frac{\Omega-x_n}{\bar{r}_n}$ and 
	\begin{align}\label{def:bar-r-n-bdd}
		\bar{r}_n:=\text{dist}(x_n,\partial\Omega)>0;
	\end{align}
	moreover, since $\partial\Omega\neq \emptyset$, we can also fix a point $y_n\in \partial\Omega$ such that $|y_n-x_n|=\bar{r}_n$. Then by \eqref{Sys:main-interior-general}, $\{\bm{v}_n\}$ solves the equations
	\begin{align}\label{Sys:main-interior-general-1}
		-\Delta v_{i,n}=\bar{r}_n f_{i,n}(x_n+\bar{r}_nx, u_{i,n}(x_n+\bar{r}_nx)) - v_{i,n}^{p_i} \sum_{j\neq i}^k a_{ij} \beta_{ij;n}\bar{r}_n^{1+p_i+p_j}v_{j,n}^{p_j},\quad  v_{i,n}>0\quad \text{in}~~ \Omega_n,
	\end{align}
	where $1\leq i\leq k$ and $n\geq 1$.
	\par
	First of all, by Lemma \ref{lem:boundary-est} and Theorem \ref{thm:interior-lip}, we deduce that  $\sum_{i=1}^{k}v_{i,n}(0)$ tends to infinity as $n\to\infty$.
	\begin{lemma}\label{lem:bahaviour-global}
		Under the previous notation, we have 
		$$\sum_{i=1}^{k}v_{i,n}(0)=\frac{\sum_{i=1}^{k}u_{i,n}(x_n)}{\bar{r}_n}\to +\infty\quad \text{as}~~n\to\infty.$$
	\end{lemma}
	\begin{proof}
		Suppose by contradiction that up to a subsequence, $\sum_{i=1}^{k}v_{i,n}(0)\leq C$ for any $n\geq 1$, where $C>0$  is a constant independent of $n$. Notice that $\overline{B_{\bar{r}_n}(x_n)}\subseteq \overline{\Omega}$, we have $\overline{B_1}\subseteq \overline{\Omega_n}$. Then for any $x\in\overline{B_1}$, we deduce from Lemma \ref{lem:boundary-est} that
		\begin{align*}
			|v_{i,n}(x)|
			&\leq |v_{i,n}(x)-v_{i,n}((y_n-x_n)/\bar{r}_n)|+|v_{i,n}(0)-v_{i,n}((y_n-x_n)/\bar{r}_n)|+|v_{i,n}(0)|\\
			&\leq 3M_1+C,\quad \forall 1\leq i\leq k,~~\forall n\geq 1,
		\end{align*}
		where $M_1>0$ is independent of $n$. This implies that $\{\bm{v}_n\}$ is uniformly bounded in $L^{\infty}(B_1;\mathbb{R}^k)$. Combining this with \eqref{Sys:main-interior-general-1}, it follows from Theorem \ref{thm:interior-lip} that $\{|\nabla \bm{v}_n(0)|\}$ is bounded in $\mathbb{R}$, that is, $\{|\nabla \bm{u}_n(x_n)|\}$ is bounded in $\mathbb{R}$, contradicting to the fact that $L_n = |\nabla u_{1,n}(x_n)|\to +\infty$ as $n\to\infty$. The proof is finished.
	\end{proof}
	\par 
	Next, note that Lemma \ref{lem:bahaviour-global} implies that the sequence $\{\bm{v}_n\}$ diverges to infinity at $x = 0$ and that $\bar{r}_n \to 0$ as $n \to \infty$. Using these results together with Lemmas \ref{lem:bahaviour-global} and \ref{lem:decay-Mu-p}, we can further analyze the behavior of $\{\bm{v}_n\}$.
	\begin{lemma}\label{lem:bahaviour-global-1}
		Under the previous notation, up to a subsequence, the following assertions hold:\\
		(1) There exists $1\leq i_0\leq k$ such that $v_{i_0,n}(0)\to +\infty$ as $n\to\infty$; moreover, 
		$$0<\frac{1}{2}v_{i_0,n}(0)\leq v_{i_0,n}(x)\leq \frac{3}{2}v_{i_0,n}(0),\quad \forall x\in \overline{B_5}\cap \overline{\Omega_n},\quad \forall n\geq 1.$$
		(2) There exists a constant $C>0$ independent of $n$ such that for any $1\leq j\leq k$ with $j\neq i_0$, one has 
		$$0\leq v_{j,n}(x)\leq C,\quad \forall x\in \overline{B_5}\cap \overline{\Omega_n},\quad\text{and}\quad v_{j,n}(x)=0,\quad \forall x\in \overline{B_5}\cap \partial\Omega_n.$$
		(3) There exists a constant $C>0$ independent of $n$ such that 
		$$v_{i_0,n}^{p_{i_0}} \sum_{j \neq i_0}^k \beta_{ji_0;n}\bar{r}_n^{1+p_{i_0}+p_j}v_{j,n}^{p_j}=v_{i_0,n}^{p_{i_0}} \sum_{j \neq i_0}^k \beta_{i_0j;n}\bar{r}_n^{1+p_{i_0}+p_j}v_{j,n}^{p_j}\leq C\quad \text{in}~~B_1.$$
	\end{lemma}
	\begin{proof}
		(1) The existence of $i_0$ such that $\limsup\limits_{n\to\infty}v_{i_0,n}(0)=+\infty$ follows immediately from Lemma \ref{lem:bahaviour-global}. On the other hand, by Lemma \ref{lem:boundary-est}, we deduce that $\forall x\in \overline{B_5}\cap \overline{\Omega_n}$,
		\begin{align*}
			|v_{i_0,n}(x)-v_{i_0,n}(0)|
			&\leq \frac{1}{\bar{r}_n}|u_{i_0,n}(x_n+\bar{r}_nx)-u_{i_0,n}(y_n)|+\frac{1}{\bar{r}_n}|u_{i_0,n}(x_n)-u_{i_0,n}(y_n)|\\
			&\leq 6M_1+M_1=7M_1,
		\end{align*}
		where $M_1>0$ is independent of $n$. As a consequence, it follows that 
		$$\sup_{x\in \overline{B_5}\cap \overline{\Omega_n}}|v_{i_0,n}(x)-v_{i_0,n}(0)|\leq \frac{1}{2}v_{i_0,n}(0)\quad \text{for $n$ large},$$ 
		as desired.
		\par 
		(2) Let $1\leq j\leq k$ with $j\neq i_0$ and $n\geq 1$. By \eqref{Condi:boundary-disjoint} and Point (1), we deduce that $v_{j,n}(x)=0$ for any $x\in \overline{B_5}\cap \partial\Omega_n$. Combining this with Lemma \ref{lem:boundary-est} yields that 
		$$0\leq v_{j,n}(x)=|v_{j,n}(x)-v_{j,n}((y_n-x_n)/\bar{r}_n)|\leq 6M_1=C,\quad \forall x\in \overline{B_5}\cap \overline{\Omega_n}.$$
		\par 
		(3) Let $n\geq 1$. We extend $\bm{v}_n$ to $\overline{B_2}$ by defining 
		$$v_{j,n}\equiv 0\quad\text{in}~~\overline{B_2}\setminus\Omega_n\quad \forall j\neq i_0,\quad \text{and}\quad v_{i_0,n}\equiv v_{i_0,n}(0)\quad\text{in}~~\overline{B_2}\setminus\Omega_n.$$
		\par 
		Let $j \neq i_0$. By Point (2), we have that $v_{j,n} \in C(\overline{B_2}) \cap H^1(B_2)$. Moreover, since $v_{j,n} \equiv 0$ on $B_2 \cap \partial\Omega_n$ and $v_{j,n} \geq 0$ in $B_2 \cap \Omega_n$, we deduce that $\partial_{\nu} v_{j,n} \leq 0$ on $B_2 \cap \partial\Omega_n$, where $\nu(x)$ is the unit outer normal vector to $\Omega_n$ at a point $x \in \partial\Omega_n$.
		\par 
		We claim that for $j\neq i_0$, $v_{j,n}$ weakly solves 
		$$-\Delta v_{j,n}\leq d_m\bar{r}_n-[v_{i_0,n}(0)/2]^{p_{i_0}} a_{ji_0} \beta_{ji_0;n}\bar{r}_n^{1+p_{i_0}+p_j}v_{j,n}^{p_j},\quad v_{j,n}\geq 0,\quad \text{in}~~B_2.$$
		\par 
		Indeed, for any $\eta\in C_{c}^{\infty}(B_2)$ with $\eta\geq 0$ in $B_2$, by \eqref{Sys:main-interior-general-1}, (H1) and Point (1), we derive 
		\begin{align*}
			\int_{B_2}\nabla v_{j,n}\cdot \nabla \eta
			&=	\int_{B_2\cap \Omega_n}\nabla v_{j,n}\cdot \nabla \eta\leq \int_{B_2\cap \Omega_n}\nabla v_{j,n}\cdot \nabla \eta-\int_{B_2\cap \partial\Omega_n}\eta\partial_{\nu}v_{j,n}\\
			&\leq \int_{B_2\cap \Omega_n}\left\{d_m\bar{r}_n-[v_{i_0,n}(0)/2]^{p_{i_0}} a_{ji_0} \beta_{ji_0;n}\bar{r}_n^{1+p_{i_0}+p_j}v_{j,n}^{p_j}\right\}\eta\\
			&\leq \int_{B_2}\left\{d_m\bar{r}_n-[v_{i_0,n}(0)/2]^{p_{i_0}} a_{ji_0} \beta_{ji_0;n}\bar{r}_n^{1+p_{i_0}+p_j}v_{j,n}^{p_j}\right\}\eta,
		\end{align*}
		as desired. The claim is proved.
		\par 
		With the claim and Point (2) in hand, it follows from Lemma \ref{lem:decay-Mu-p} that 
		$$[v_{i_0,n}(0)/2]^{p_{i_0}} a_{ji_0} \beta_{ji_0;n}\bar{r}_n^{1+p_{i_0}+p_j}v_{j,n}^{p_j}\leq C\quad \text{in}~~B_1,$$
		where $j\neq i_0$ and $C>0$ is independent of $n$. Since $\beta_{ji_0;n}=\beta_{i_0j;n}$, we then deduce from Point (1) that
		$$v_{i_0,n}^{p_{i_0}} \sum_{j\neq i_0}^k \beta_{i_0j;n}\bar{r}_n^{1+p_{i_0}+p_j}v_{j,n}^{p_j}\leq [3/2v_{i_0,n}(0)]^{p_{i_0}} \sum_{j\neq i_0}^k \beta_{ji_0;n}\bar{r}_n^{1+p_{i_0}+p_j}v_{j,n}^{p_j}\leq C\quad \text{in}~~B_1,$$
		where  $C>0$ is independent of $n$.  The proof is finished.
	\end{proof}
	Now we are ready to complete the proof of Theorem \ref{thm:global-lip}.
	\begin{proof}[Conclusion of the proof of Theorem \ref{thm:global-lip}]
		On the one hand, by Lemma \ref{lem:bahaviour-global-1} (3) and \eqref{Sys:main-interior-general-1}, we see that $\{\Delta [v_{i_0,n}-v_{i_0,n}(0)]\}$ is bounded in $L^{\infty}(B_1)$. Moreover, it can be easily deduced from Lemma \ref{lem:boundary-est} that $\{[v_{i_0,n}-v_{i_0,n}(0)]\}$ is  bounded in $L^{\infty}(B_1)$ as well. By the standard elliptic estimates, it then follows that $\{\nabla u_{i_0,n}(x_n)\}$ is bounded in $\mathbb{R}$. On the other hand, for any $i\neq i_0$, we know from \eqref{Sys:main-interior-general-1} that $v_{i,n}$ is a solution of 
		$$-\Delta v_{i,n}=\bar{f}_{i,n}(x)-v_{i,n}^{p_i} \sum_{j \notin\{i,i_0\}}^k a_{ij} \beta_{ij;n}\bar{r}_n^{1+p_i+p_j}v_{j,n}^{p_j}\quad \text{in}~~B_1,$$
		where 
		$$\bar{f}_{i,n}(x):=\bar{r}_n f_{i,n}(x_n+\bar{r}_nx, u_{i,n}(x_n+\bar{r}_nx))-v_{i,n}^{p_i} a_{ii_0} \beta_{ii_0;n}\bar{r}_n^{1+p_i+p_{i_0}}v_{i_0,n}^{p_{i_0}}.$$
		By Lemma \ref{lem:bahaviour-global-1} (2) and (3), we have that $\{v_{i,n}\}$ and $\{\bar{f}_{i,n}\}$ are uniformly bounded in $L^{\infty}(B_1)$ for any $i\neq i_0$. This allows us to use Theorem \ref{thm:interior-lip} to derive $\{\nabla u_{i,n}(x_n)\}$ is bounded in $\mathbb{R}$ for any $i\neq i_0$. As a consequence, $\{\nabla\bm{u}_n(x_n)\}$ is bounded in $\mathbb{R}$, contradicting to $L_n=|\nabla u_{1,n}(x_n)|\to +\infty$ as $n\to\infty$. The proof is complete.
	\end{proof}

	\section{Asymptotic behavior}\label{sec:6}
	\setcounter{section}{6}
	\setcounter{equation}{0}
	
	This section is devoted to analyzing the asymptotic behavior of solutions to system \eqref{Sys:main-interior} as $\beta_n \to +\infty$. Propositions \ref{prop:uniform-upper-bdd}, \ref{prop:existence-xn}, and Theorems \ref{thm:ab-near-regular}, \ref{thm:ap-near-gamma-n} are established herein.
	
	\subsection{Uniform upper bound}
	
	First, we use Theorem \ref{thm:interior-lip} to prove Proposition \ref{prop:uniform-upper-bdd}. The proof can be modified along the lines of \cite[Theorem 1.1]{Soave-Z2017Poincare}. However, we need to account for the nonhomogeneity of $\bm{\hat{g}}$. 
	\begin{proof}[Proof of Proposition \ref{prop:uniform-upper-bdd}]
		Suppose, for contradiction, that there exist a sequence $\{x_n\} \subseteq K$ and indices $i \neq j$ such that 
		\begin{align}\label{eq:Ab-2}
			\beta_{ij;n} u_{i,n}(x_n)^{p_i + 1} u_{j,n}(x_n)^{p_j} \to +\infty \quad \text{as } n \to \infty.
		\end{align}
		Without loss of generality, we assume $i = 1$ and $j = 2$. Furthermore, by the compactness of $K$, we may pass to a subsequence (still denoted by $\{x_n\}$ for simplicity) and assume $x_n \to x_0$ for some $x_0 \in \Omega$ as $n \to \infty$.
		
		We claim that $\varepsilon_n := u_{1,n}(x_n) + u_{2,n}(x_n) \to 0$ as $n \to \infty$.
		
		Indeed, if this were not the case, we could assume without loss of generality that $u_{1,n}(x_n) \geq c > 0$ for all $n \geq 1$, where $c > 0$ is a constant independent of $n$. By Theorem \ref{thm:interior-lip} and the Ascoli-Arzelà theorem, there exists $\delta > 0$ independent of $n$ such that $u_{1,n} \geq c/2$ in $B_{2\delta}(x_0)$. As a consequence of \eqref{Sys:main-interior-general}, we have 
		\[
		-\Delta u_{2,n} \leq d_m - \left(\frac{c}{2}\right)^{p_1} a_{21} \beta_{21;n} u_{2,n}^{p_2} \quad \text{in } B_{2\delta}(x_0).
		\]
		By Lemma \ref{lem:decay-Mu-p}, it follows that $\beta_{21;n} u_{2,n}^{p_2} \leq C$ in $B_{\delta}(x_0)$ for some constant $C > 0$ independent of $n$. Since $\beta_{21;n} = \beta_{12;n}$ and $\{\bm{u}_n\}$ is uniformly bounded in $L^{\infty}(\Omega; \mathbb{R}^k)$, this implies that $\{\beta_{12;n} u_{1,n}(x_n)^{p_1 + 1} u_{2,n}(x_n)^{p_2}\}$ is bounded in $\mathbb{R}$, contradicting to \eqref{eq:Ab-2}. The claim is thus proved.
		
		Now let $0 < r < \text{dist}(K, \partial\Omega)$ and define 
		\[
		v_{i,n}(x) := \frac{1}{\varepsilon_n} u_{i,n}(x_n + \varepsilon_n x) \quad \forall x \in B_{r/\varepsilon_n}, ~\forall 1 \leq i \leq k.
		\]
		By Theorem \ref{thm:interior-lip} and the Ascoli-Arzelà theorem, we may pass to a subsequence (still denoted by $\{\bm{v}_n\}$) such that $\{(v_{1,n}, v_{2,n})\}$ converges to a function $(v_1, v_2)$ in $C_{\text{loc}}(\mathbb{R}^N; \mathbb{R}^2)$ as $n \to \infty$. Moreover, $\{\bm{v}_n\}$ satisfies the system
		\begin{align}\label{Eq:Ab}
			-\Delta v_{i,n} = \varepsilon_n f_{i,n}\left(x_n + \varepsilon_n x, u_{i,n}(x_n + \varepsilon_n x)\right) - v_{i,n}^{p_i} \sum_{j \neq i} a_{ij} \beta_{ij;n} \varepsilon_n^{p_i + p_j + 1} v_{j,n}^{p_j}
		\end{align}
		for all $1 \leq i \leq k$ and in $B_{r/\varepsilon_n}$. Since $\sum_{i=1}^2 v_{i,n}(0) = 1$ for all $n$, we may assume without loss of generality that $v_1(0) > 0$. Combining this with the fact that $\beta_{12;n} \varepsilon_n^{p_1 + p_2 + 1} \geq \beta_{12;n} u_{1,n}(x_n)^{p_1 + 1} u_{2,n}(x_n)^{p_2} \to +\infty$ as $n \to \infty$, we can argue as in the proof of the above claim to deduce that $\{\beta_{12;n} \varepsilon_n^{p_1 + p_2 + 1} v_{2,n}(0)^{p_2}\}$ is bounded in $\mathbb{R}$. By the definitions of $\varepsilon_n$ and $\bm{v}_n$, this implies that $\{\beta_{12;n} u_{1,n}(x_n)^{p_1 + 1} u_{2,n}(x_n)^{p_2}\}$ is bounded in $\mathbb{R}$, a contradiction. The proof is therefore completed.
	\end{proof}
	
	\subsection{Asymptotic behavior near regular points}
	Next, we investigate the asymptotic behavior of solutions $\{\bm{u}_n\}$ to system \eqref{Sys:main-interior} near regular points. We borrow ideas from \cite{Caffarelli-K-L2009fixpoint, Dancer-WZ2012Trans} to prove Theorem \ref{thm:ab-near-regular}. The key observation is the following decay estimate, which constitutes a stronger version of Lemma \ref{lem:decay-Mu-p}.
	\begin{lemma}\label{lem:decay-Mu-p-general}
		Let $N\geq 1$, $x_0\in \mathbb{R}^N$ and $M,\rho>0$. Let  $u\in H^1(B_{2\rho}(x_0))\cap C(\overline{B_{2\rho}(x_0)})$ be a non-negative subsolution to $-\Delta u\leq -Mu^p$ in $B_{2\rho}(x_0)$ for some $p>1$. Then there exists a constant $C>0$ depending only on $N$ and $p$ such that 
		$$u(x)\leq CM^{-\frac{1}{p-1}}\rho^{-\frac{2}{p-1}},\quad \forall x\in B_{\rho}(x_0).$$
	\end{lemma}
	\begin{proof}
		Without loss of generality, we assume $x_0=0$. Let 
		$$v(x):=u(M^{-1/2}x),\quad \forall x\in B_{2M^{1/2}\rho}.$$
		Then $v$ is a non-negative subsolution to $-\Delta v+v^p\leq 0$ in $B_{2M^{1/2}\rho}$.
		On the other hand, we define a comparison function (as in \cite{Brezis-AMO1984}):
		$$V(x):=\frac{C(N,p)R^{\frac{2}{p-1}}}{(R^2-|x|^2)^{\frac{2}{p-1}}},\quad \forall x\in B_R,$$
		where $R=2M^{1/2}\rho$ and $C(N,p):=\left(\frac{4}{p-1}\max\{N,2\frac{p+1}{p-1}\}\right)^{1/(p-1)}>0$. Then $V$ is a positive supsolution to
		$$-\Delta V+V^p\geq 0\quad \text{in}~~B_{R}.$$
		Note that $\lim\limits_{|x|\to R^-}V(x)=+\infty$. By Kato's inequality and the maximum principle, one can derive $(v-V)^+\equiv 0$ in $B_R$. This then implies that 
		$v\leq V$ in $B_R$. In particular, 
		$$v(x)\leq V(x)\leq \left(\frac{4}{3}\right)^{\frac{2}{p-1}}C(N,p)R^{-\frac{2}{p-1}},\quad \forall x\in B_{R/2}.$$
		Therefore, $u\leq CM^{-\frac{1}{p-1}}\rho^{-\frac{2}{p-1}}$ in $B_{\rho}$, where $C>0$ is a constant depending only on $N$ and $p$. The proof is finished.
	\end{proof}
	
	\begin{remark}\label{rmk:bdd-Mu}
		Recall (see \cite[Lemma 4.4]{Conti-Terracini-Verzini-2005adv}) that if $u\in H^1(B_{2\rho}(x_0))\cap C(\overline{B_{2\rho}(x_0)})$ is a non-negative subsolution to $-\Delta u\leq -Mu$ in $B_{2\rho}(x_0)$, then there exists a constant $C > 0$ depending only on $N$ such that
		\[
		u(x) \leq C A e^{-\frac{\sqrt{M} \rho}{2}} \quad \forall x \in B_{\rho}(x_0),
		\]
		where $A:= \max_{\partial B_{2\rho}(x_0)}u\geq 0$.
	\end{remark}
	\par 
	Equipped with this decay property, we can adapt the proof of \cite[Theorem 5.8]{Dancer-WZ2012Trans} to establish Theorem \ref{thm:ab-near-regular}. We present only the necessary modifications.
	\par 
	Assume that the assumptions of Theorem \ref{thm:ab-near-regular} hold, and let $\{\bm{u}_n\}$ and $\bm{u}$ be as in Theorem \ref{thm:ab-near-regular}. Recall that $x_0 \in \mathcal{R}_{\bm{u}}$, and that $i_1, i_2$ are the only two indices for which $u_{i_1}, u_{i_2} \not\equiv 0$ in a neighborhood of $x_0$. Without loss of generality, we assume $x_0 = 0$, $i_1 = 1$, and $i_2 = 2$.
	\par 
	Fix a small constant $h_0 > 0$, which will be determined later.
	\par 
	First, since $0\in \mathcal{R}_{\bm{u}}$, by \cite[Lemmas 9 and 10]{Caffarelli-K-L2009fixpoint}, up to a rescaling, we may assume that $B_2\subseteq \Omega$ and 
	$$-\Delta\left(u_1 - \frac{a_{12}}{a_{21}}u_2\right) = 0\quad \text{in }B_1,\quad \max_{\overline{B_1}}\left|u_1 - \frac{a_{12}}{a_{21}}u_2-e_1\cdot x\right|\leq h_0,\quad \nabla \left(u_1 - \frac{a_{12}}{a_{21}}u_2\right)(0)=e_1,$$
	where $e_1=(1,0,\dots,0)\in \partial B_1$. Note that after rescaling, $\bm{u}_n$ is a weak solution to system \eqref{Sys:main-interior-general} with $\bm{f}_n\equiv 0$ in $B_2$ and $\beta_{i,j;n}=a_0^{p_i+p_j+1}b_0^{p_i+p_j-1}\beta_n$ for all $i\neq j$, where $a_0, b_0$ are positive constants independent of $n$ (but depending on $h_0$).
	\par 
	Next, we define the iteration as follows:
	\[
	h_{l+1} = h_l^2, \quad R_0 = 1, \quad R_{l+1} = R_l - h_l^{1/2} \quad \text{for all } l \geq 0.
	\]
	If $h_0 > 0$ is sufficiently small, then $\{h_l\}$ decays to zero rapidly and $\lim_{l \to +\infty} R_l \geq 1/2$. For each $l \geq 0$, we decompose 
	\[
	u_{1,n} - \frac{a_{12}}{a_{21}} u_{2,n} = v_l + w_l \quad \text{in } B_{\widetilde{R}_l},
	\]
	where $v_l$ satisfies the boundary value problem
	\begin{align*}
		\begin{cases}
			-\Delta v_l = 0 & \text{in } B_{\widetilde{R}_l}, \\
			v_l = u_{1,n} - \frac{a_{12}}{a_{21}} u_{2,n} & \text{on } \partial B_{\widetilde{R}_l}.
		\end{cases}
	\end{align*}
	Here, $\widetilde{R}_0 := R_0 = 1$ and $\widetilde{R}_l := R_{l-1} - \frac{2}{3} h_{l-1}^{1/2}$ for $l \geq 1$.
	By induction, we can then establish the following lemma.
	\begin{lemma}\label{lem:iteration}
		Let $N\geq 2$, $p=\max_{1 \leq i \leq k}\{p_i\}>1$, $l\in \mathbb{N}$ and $h_0\in (0,1)$ small enough. Then there exist positive constants $c_0, c_1, c_2, c_3, \lambda$ independent of $n$ and $l$ such that if 
		\begin{align*}
			h_l \geq 
			\begin{cases}
				c_3\max_{p_i>1,\atop i\neq 1,2}\beta_n^{-\frac{1}{2p_i+p}}\quad &\text{if }\{3\leq i\leq k:p_i>1\}\neq \emptyset,\\
				\beta_n^{-\frac{1}{p(p+2)}}\quad &\text{otherwise},
			\end{cases}
		\end{align*} 
		then for $n$ large, we have:
		\begin{enumerate}
			\item [(1)] \( |v_l - (u_{1,n} -  \frac{a_{12}}{a_{21}}u_{2,n})| \leq 2h_l \) in $B_{R_l}$; moreover, if $l\geq 1$, then 
			\begin{align*}
				&\left|v_l - \left(u_{1,n} -  \frac{a_{12}}{a_{21}}u_{2,n}\right)\right| \\
				\leq& c_0 \sum_{p_i>1,\atop i\neq 1,2}h_{l-1}^{-\frac{p+2}{p_i-1}} \beta_n^{-\frac{1}{p_i-1}}+c_0\sum_{p_i=1,\atop i\neq 1,2}h_{l-1}e^{-c_1h_{l-1}^{(p+2)/2}\sqrt{\beta_n}}+ c_2 h_{l-1}e^{-\frac{\ln 2}{6\lambda h_{l-1}^{1/2}}}\quad\text{in }B_{R_l}.
			\end{align*}
			\item [(2)] $\sum_{i\neq 1,2} u_{i,n}\leq 2h_l$ in $B_{R_l}$; moreover, if $l\geq 1$ and $i\neq 1,2$,  then 
			\begin{align*}
				u_{i,n}\leq 
				\begin{cases}
					c_0h_{l-1}^{-\frac{p+2}{p_i-1}} \beta_n^{-\frac{1}{p_i-1}}+4h_{l-1} e^{-\frac{\ln 2}{6\lambda h_{l-1}^{1/2}}} \quad &\text{if } p_i>1,\\
					c_0h_{l-1}e^{-c_1h_{l-1}^{(p+2)/2}\sqrt{\beta_n}}+4h_{l-1}e^{-\frac{\ln 2}{6\lambda h_{l-1}^{1/2}}} \quad &\text{if } p_i=1
				\end{cases}
				\quad \text{in } B_{R_{l}}.
			\end{align*}
			\item [(3)] If $l\geq 1$, then \( |\nabla(v_{l} - v_{l-1})| \leq 7Nh_{l-1}^{\frac{1}{2}} \)  in $B_{R_{l}}$.
			\item [(4)] \( |\nabla v_l - e_1| \leq 7Nh_{0}^{\frac{1}{2}}+7N\sum_{j=1}^l h_{j-1}^{\frac{1}{2}} \leq \frac{1}{4} \) in $B_{R_{l}-h_l^{1/2}/3}$.
			\item [(5)] $\{v_l=0\}\cap B_{R_{l}-h_l^{1/2}/3}$ is a Lipschitz surface with a Lipschitz constant less than or equal to 3.
		\end{enumerate}
	\end{lemma}
	\begin{proof}
		It is easy to see that the result holds for $l=0$. Assuming that the result holds for any $l \geq 0$, we aim to prove that it holds for $l+1$.
		\par 
		First, we estimate $u_{i,n}$ in $B_{R_{l} - 2h_l^{1/2}/3}$, where $i\neq 1,2$. Following the proof of \cite[Lemma 5.9]{Dancer-WZ2012Trans}, we partition the ball $B_{R_{l} - 2h_l^{1/2}/3}$ into two subsets: $\{|v_l| > 8h_l\} \cap B_{R_{l} - 2h_l^{1/2}/3}$ and $\{|v_l| \leq 8h_l\} \cap B_{R_{l} - 2h_l^{1/2}/3}$.
		
		For any $x \in \{|v_l| > 8h_l\} \cap B_{R_{l} - h_l^{1/2}/2}$, it follows from Point (4) that $B_{2h_l}(x) \subseteq \{|v_l| \geq 4h_l\}$. By Point (1), we further deduce that
		\[
		u_{1,n} + \frac{a_{12}}{a_{21}}u_{2,n} \geq \left|u_{1,n} - \frac{a_{12}}{a_{21}}u_{2,n}\right| \geq |v_l| - 2h_l \geq 2h_l \quad \text{in } B_{2h_l}(x).
		\]
		Combining this inequality with system \eqref{Sys:main-interior} and Point (2), we obtain
		\[
		-\Delta u_{i,n} \leq -c_0\beta_n u_{i,n}^{p_i}\left(a_{i1}u_{1,n}^{p_1} + a_{i2}u_{2,n}^{p_2}\right) \leq -c_0\beta_n u_{i,n}^{p_i}h_l^p,\quad u_{i,n}\leq 2h_l \quad \text{in } B_{2h_l}(x) \quad \forall i \neq 1,2,
		\]
		where $c_0 > 0$ is a constant depending only on $\bm{p} := (p_1, \dots, p_k)$, $(a_{ij})$, $a_0$, and $b_0$. This enables us to apply Lemma \ref{lem:decay-Mu-p-general} or Remark \ref{rmk:bdd-Mu}, yielding
		\begin{align}\label{eq:ap-regular-1}
			u_{i,n}\leq 
			\begin{cases}
				c_0h_l^{-\frac{p+2}{p_i-1}} \beta_n^{-\frac{1}{p_i-1}}\quad &\text{if } p_i>1,\\
				c_0h_le^{-c_1h_l^{(p+2)/2}\sqrt{\beta_n}}\quad &\text{if } p_i=1
			\end{cases}
			\quad \text{in }\{|v_l| > 8h_l\} \cap B_{R_{l} - h_l^{1/2}/2},
		\end{align}
		where $c_0,c_1 > 0$ are constants independent of $n$ and $l$.
		\par 
		On the other hand, using Points (4) and (5), one can deduce that the set $\{|v_l| \leq 8h_l\} \cap B_{R_l - h_l^{1/2}/2}$ is $(\lambda h_l)$-narrow in the sense of \cite[Lemma 12]{Caffarelli-K-L2009fixpoint}, where $\lambda > 0$ is a constant depending only on $N$. By combining this property with \eqref{eq:ap-regular-1}, we can follow the argument in the proof of \cite[Lemma 5.10]{Dancer-WZ2012Trans} to derive 
		\begin{align}\label{eq:ap-regular-2}
			u_{i,n}\leq 
			\begin{cases}
				2c_0h_l^{-\frac{p+2}{p_i-1}} \beta_n^{-\frac{1}{p_i-1}}+4h_l e^{-\frac{\ln 2}{6\lambda h_l^{1/2}}} \quad &\text{if } p_i>1,\\
				2c_0h_le^{-c_1h_l^{(p+2)/2}\sqrt{\beta_n}}+4h_l e^{-\frac{\ln 2}{6\lambda h_l^{1/2}}} \quad &\text{if } p_i=1
			\end{cases}
			\quad \text{in }\{|v_l| \leq 8h_l\} \cap B_{R_{l} - 2h_l^{1/2}/3}.
		\end{align}
		As a consequence of \eqref{eq:ap-regular-1} and \eqref{eq:ap-regular-2}, we thus obtain
		\begin{align}\label{eq:ap-regular-3}
			u_{i,n}\leq 
			\begin{cases}
				2c_0h_l^{-\frac{p+2}{p_i-1}} \beta_n^{-\frac{1}{p_i-1}}+4h_l e^{-\frac{\ln 2}{6\lambda h_l^{1/2}}} \quad &\text{if } p_i>1,\\
				2c_0h_le^{-c_1h_l^{(p+2)/2}\sqrt{\beta_n}}+4h_l e^{-\frac{\ln 2}{6\lambda h_l^{1/2}}} \quad &\text{if } p_i=1
			\end{cases}
			\quad \text{in } B_{R_{l} - 2h_l^{1/2}/3}.
		\end{align}
		\par 
		Next, we estimate $w_{l+1}$. Note that $-\Delta \hat{u}_{j,n} \geq 0$ in $B_{R_l - 2h_l^{1/2}/3}$ for $j = 1,2$, where the hat operator is defined in \eqref{def:hat-operator}. By the definition of $w_{l+1}$, the weak comparison principle can be applied to deduce that 
		\[
		|w_{l+1}| \leq c_0' \sum_{i \neq 1,2} \sup_{B_{R_l - 2h_l^{1/2}/3}} u_{i,n} \quad \text{in } B_{R_l - 2h_l^{1/2}/3},
		\]
		where $c_0' > 1$ depends only on $(a_{ij})$. It then follows from \eqref{eq:ap-regular-3} that 
		\begin{align*}
			|w_{l+1}| \leq c_0 \sum_{p_i>1,\atop i\neq 1,2}h_l^{-\frac{p+2}{p_i-1}} \beta_n^{-\frac{1}{p_i-1}}+c_0\sum_{p_i=1,\atop i\neq 1,2}h_le^{-c_1h_l^{(p+2)/2}\sqrt{\beta_n}}+ c_2 h_l e^{-\frac{\ln 2}{6\lambda h_l^{1/2}}} \quad \text{in } B_{R_l - 2h_l^{1/2}/3},
		\end{align*}
		where $c_0>0$ is a constant independent of $n$ and $l$, and $c_2>0$ depends only on $k$ and $(a_{ij})$.
		\par 
		Now, if $h_0 > 0$ is sufficiently small but fixed, and
		\begin{align*}
			h_l \geq 
			\begin{cases}
				c_3\max_{p_i>1,\atop i\neq 1,2}\beta_n^{-\frac{1}{2p_i+p}}\quad &\text{if }\{3\leq i\leq k:p_i>1\}\neq \emptyset,\\
				\beta_n^{-\frac{1}{p(p+2)}}\quad &\text{otherwise},
			\end{cases}
		\end{align*} 
		where $c_3:=\max_{p_j>1,\atop j\neq 1,2}(kc_0)^{\frac{p_j-1}{2p_j+p}}>0$, 
		then for $n$ large, one has
		\begin{align}\label{eq:ap-regular-4}
			|w_{l+1}| \leq \frac{1}{k}\sum_{p_i>1,\atop i\neq 1,2}h_l^2+\frac{1}{k}\sum_{p_i=1,\atop i\neq 1,2}h_l^2+h_l^2 \leq 2h_l^2 \quad \text{in } B_{R_l - 2h_l^{1/2}/3} \supseteq B_{R_{l+1}}.
		\end{align}
		Hence, Points (1) and (2) hold for $l+1$; Point (3) follows from \eqref{eq:ap-regular-4} together with the interior gradient estimates for harmonic functions; Point (4) is a direct consequence of Point (3), and Point (5) follows in turn from Point (4). The proof is thus completed.
	\end{proof}
	\par 
	\begin{proof}[Conclusion of the proof of Theorem \ref{thm:ab-near-regular}]
		If $N \geq 2$, we first choose $h_0 > 0$ to be sufficiently small. Let $3\leq i\leq k$. For $n\gg 1$ large, we then select an integer $l_n \in \mathbb{N}$ such that
		\begin{align*}
			\begin{cases}
				h_{l_n - 1} \geq \max_{p_i>1,\atop i\neq 1,2}c_0^{\frac{p_i-1}{p+2}} \beta_n^{-\frac{1}{2p_i(p+2)}} > h_{l_n - 1}^2\quad &\text{if }\{3\leq i\leq k:p_i>1\}\neq \emptyset,\\
				h_{l_n - 1}\geq \beta_n^{-\frac{1}{2p(p+2)}}>h_{l_n - 1}^2\quad &\text{otherwise}.
			\end{cases}
		\end{align*}
		\par 
		If $\{3\leq i\leq k:p_i>1\}\neq \emptyset$, then 
		\[
		h_{l_n} = h_{l_n - 1}^2 \geq  \max_{p_i>1,\atop i\neq 1,2}c_0^{2 \cdot \frac{p_i-1}{p+2}} \beta_n^{-\frac{1}{p_i(p+2)}} > c_3\max_{p_i>1,\atop i\neq 1,2}\beta_n^{-\frac{1}{2p_i+p}}.
		\]
		By Part (2) of Lemma \ref{lem:iteration}, we obtain
		\begin{align*}
			u_{i,n}\leq 
			\begin{cases}
				\beta_n^{-\frac{1}{p_i}-\frac{1}{2p_i(p_i-1)}}+2h_{l_n-1} e^{-\frac{\ln 2}{6\lambda h_{l_n-1}^{1/2}}}\leq 2\beta_n^{-\frac{1}{p_i}-\frac{1}{2p_i(p_i-1)}} \quad &\text{if } p_i>1,\\
				c_0h_{l_n-1}e^{-c_1h_{l_n-1}^{-(p+2)/2}}+2h_{l_n-1}e^{-\frac{\ln 2}{6\lambda h_{l_n-1}^{1/2}}}\leq e^{-\beta_n^{\frac{1}{16p(p+2)}}} \quad &\text{if } p_i=1
			\end{cases}
			\quad \text{in } B_{1/2}.
		\end{align*}
		\par 
		If $\{3\leq i\leq k:p_i>1\}=\emptyset$, then similarly, we deduce that 
		$$u_{i,n}\leq c_0h_{l_n-1}e^{-c_1h_{l_n-1}^{(p+2)/2}\sqrt{\beta_n}}+4h_{l_n-1}e^{-\frac{\ln 2}{6\lambda h_{l_n-1}^{1/2}}}\leq e^{-\beta_n^{\frac{1}{16p(p+2)}}} \quad\text{in }B_{1/2}.$$
		\par 
		For $N = 1$, the desired result can be obtained by adapting the argument in \cite[below Remark 5.11, p. 996]{Dancer-WZ2012Trans}. This completes the proof.
	\end{proof}
	
	\subsection{Sharp estimates near the interface $\Gamma_n$}
	
	Thirdly, we establish Theorem \ref{thm:ap-near-gamma-n}. The proof is divided into two lemmas.
	\par 
	First, we consider the case where $x_0 \in \mathcal{R}_{\bm{u}}$, and prove part (1) of Theorem \ref{thm:ap-near-gamma-n} by combining Theorem \ref{thm:ab-near-regular} with Proposition \ref{prop:uniform-upper-bdd}.
	\begin{lemma}\label{lem:ap-1}
		Under the assumptions of Theorem \ref{thm:ap-near-gamma-n}, and assuming moreover that $x_0 \in \mathcal{R}_{\bm{u}}$, we have
		$$\limsup_{n \to \infty} \beta_n \sum_{\substack{i,j=1 \\ i \neq j}}^k u_{i,n}(x_n)^{p_i + 1} u_{j,n}(x_n)^{p_j} \in (0, +\infty).$$
	\end{lemma}
	\begin{proof}
		Since $x_n \in \Gamma_n$ and $x_n \to x_0$ in $\Omega$ as $n \to \infty$, Proposition \ref{prop:uniform-upper-bdd} implies the existence of a constant $C > 0$ (independent of $n$) such that 
		$$\beta_n \sum_{\substack{i,j=1 \\ i \neq j}}^k u_{i,n}(x_n)^{p_i + 1} u_{j,n}(x_n)^{p_j} \leq C.$$
		It therefore remains to show that 
		$$\limsup_{n \to \infty} \beta_n \sum_{\substack{i,j=1 \\ i \neq j}}^k u_{i,n}(x_n)^{p_i + 1} u_{j,n}(x_n)^{p_j} > 0.$$
		
		Let $i_1, i_2$ denote the only two indices for which $u_{i_1}, u_{i_2} \not\equiv 0$ in a neighborhood of $x_0$. Without loss of generality, we may assume $x_0 = 0 \in B_2 \subseteq \Omega$, $i_1 = 1$, and $i_2 = 2$. By Theorem \ref{thm:ab-near-regular}, there exists $r \in (0,1)$ (independent of $n$) such that $\{\bm{u}_n\}$ satisfies the properties stated in Theorem \ref{thm:ab-near-regular} within $B_r$. In particular, 
		$$\beta_n u_{j,n}^{p_j} \to 0 \quad \text{in } L^{\infty}(B_r) \quad \text{as } n \to \infty, \quad \forall 3 \leq j \leq k.$$
		Since $x_n \to x_0 = 0$ in $\Omega$ as $n \to \infty$, this further implies 
		$$\lim_{n \to \infty} \beta_n \sum_{\substack{\{i,j\} \neq \{1,2\} \\ i \neq j}}^k u_{i,n}(x_n)^{p_i + 1} u_{j,n}(x_n)^{p_j} = 0.$$
		Thus, it suffices to prove 
		\begin{align}\label{eq:ap-near-Gamma-n-0}
			\limsup_{n \to \infty} \beta_n \left[ u_{1,n}(x_n)^{p_1 + 1} u_{2,n}(x_n)^{p_2} + u_{2,n}(x_n)^{p_2 + 1} u_{1,n}(x_n)^{p_1} \right] > 0.
		\end{align}
		
		To this end, we first claim that, up to a subsequence, there exist $s \in \{1,2\}$ and $c > 0$ such that 
		\begin{align}\label{eq:ap-near-Gamma-n-1}
			|\nabla u_{s,n}(x_n)| \geq c \quad \forall n \geq 1.
		\end{align} 
		
		Indeed, Theorem \ref{thm:ab-near-regular} yields 
		$$a_{21}u_{1,n} - a_{12}u_{2,n} \to a_{21}u_1 - a_{12}u_2 \quad \text{in } C_{\text{loc}}^1(B_r) \quad \text{as } n \to \infty.$$
		Since $|\nabla (a_{21}u_1 - a_{12}u_2)(0)| > 0$ and $x_n \to x_0 = 0$ in $B_r$ as $n \to \infty$, the claim follows immediately.
		
		Without loss of generality, we assume $s = 1$.
		
		Next, let $\varepsilon_n := u_{1,n}(x_n) + u_{2,n}(x_n) > 0$ and define 
		$$v_{i,n}(x) := \frac{1}{\varepsilon_n} u_{i,n}(x_n + \varepsilon_n x) \quad \forall x \in B_{r/(2\varepsilon_n)}, \ \forall i = 1,2.$$
		By Theorem \ref{thm:interior-lip} and the Ascoli–Arzelà theorem, we may pass to a subsequence (still denoted by $\{(v_{1,n}, v_{2,n})\}$) such that $\{(v_{1,n}, v_{2,n})\}$ converges to a function $(v_1, v_2)$ in $C_{\text{loc}}(\mathbb{R}^N; \mathbb{R}^2)$ as $n \to \infty$. Since $v_{1,n}(0) + v_{2,n}(0) = 1$ for all $n$, we have $v_1(0) + v_2(0) = 1$. Moreover, Theorem \ref{thm:ab-near-regular} further implies that $(v_{1,n}, v_{2,n})$ is a positive solution to 
		\begin{align}\label{eq:ap-near-Gamma-n-2}
			\begin{cases}
				-\Delta v_{1,n} = -a_{12}\beta_n \varepsilon_n^{p_1 + p_2 + 1} v_{1,n}^{p_1} v_{2,n}^{p_2} + o(1) & \text{in } B_{r/(2\varepsilon_n)}, \\
				-\Delta v_{2,n} = -a_{21}\beta_n \varepsilon_n^{p_1 + p_2 + 1} v_{1,n}^{p_1} v_{2,n}^{p_2} + o(1) & \text{in } B_{r/(2\varepsilon_n)},
			\end{cases}
		\end{align}
		for all $n \geq 1$, where $o(1)$ denotes a uniformly bounded function that converges to $0$ in $L^{\infty}_{\text{loc}}(\mathbb{R}^N)$ as $n \to \infty$.
		
		We now claim that $\limsup_{n \to \infty} \beta_n \varepsilon_n^{p_1 + p_2 + 1} > 0$.
		
		Suppose, for contradiction, that $\limsup_{n \to \infty} \beta_n \varepsilon_n^{p_1 + p_2 + 1} = 0$. Then, by \eqref{eq:ap-near-Gamma-n-2} and standard elliptic estimates, we may pass to a subsequence such that $v_{1,n} \to v_1$ in $C^1_{\text{loc}}(\mathbb{R}^N)$ as $n \to \infty$. It follows from \eqref{eq:ap-near-Gamma-n-1} that 
		$$|\nabla v_1(0)| = \lim_{n \to \infty} |\nabla v_{1,n}(0)| = \lim_{n \to \infty} |\nabla u_{1,n}(x_n)| \geq c > 0.$$
		On the other hand, \eqref{eq:ap-near-Gamma-n-2} implies that $v_1$ is a non-negative harmonic function in $\mathbb{R}^N$, and hence $v_1$ is constant in $\mathbb{R}^N$. This contradicts $|\nabla v_1(0)| > 0$, proving the claim.
		
		Finally, recalling the definition of $\Gamma_n$ (see \eqref{def:gamma-n}) and noting that $x_n \in \Gamma_n$, we deduce from the above claim that \eqref{eq:ap-near-Gamma-n-0} holds. The proof is complete.
	\end{proof}
	\par 
	Next, we establish part (2) of Theorem \ref{thm:ap-near-gamma-n} by applying Proposition \ref{prop:uniform-upper-bdd}, Theorem \ref{thm:interior-lip} and Theorem \ref{thm:ACF-formula}.
	\begin{lemma}\label{lem:ap-2}
		Under the assumptions of Theorem \ref{thm:ap-near-gamma-n}, and assuming moreover that $x_0 \in \Gamma_{\bm{u}}\setminus \mathcal{R}_{\bm{u}}$, we have
		$$\limsup_{n \to \infty} \beta_n \sum_{\substack{i,j=1 \\ i \neq j}}^k u_{i,n}(x_n)^{p_i + 1} u_{j,n}(x_n)^{p_j}=0.$$
	\end{lemma}
	\begin{proof}
		We argue by contradiction. Suppose that the conclusion fails. Since $x_n \in \Gamma_n$, by passing to a subsequence if necessary, we may assume that $u_{1,n}(x_n) = u_{2,n}(x_n)$ and 
		$$[u_{1,n}(x_n) + u_{2,n}(x_n)]^{p_1 + p_2 + 1} \geq [u_{l,n}(x_n) + u_{s,n}(x_n)]^{p_l + p_s + 1} \quad \text{for all } l \neq s.$$
		Moreover, by Proposition \ref{prop:uniform-upper-bdd}, we may assume that 
		\begin{align}\label{eq:ap-near-Gamma-n-3}
			\beta_n u_{1,n}(x_n)^{p_1 + 1} u_{2,n}(x_n)^{p_2} = \beta_n u_{1,n}(x_n)^{p_1 + p_2 + 1} \to c_0 > 0 \quad \text{as } n \to \infty,
		\end{align}
		where $c_0 > 0$ is a constant. Since $x_n \to x_0$ in $\Omega$ as $n \to \infty$, we may further assume that $B_{r_0/2}(x_n) \subseteq B_{r_0}(x_0)\subseteq \Omega$ for all $n \geq 1$, where $0 < r_0 < \text{dist}(x_0, \partial\Omega)$ is a fixed constant. 
		\par 
		Let $\varepsilon_n := u_{1,n}(x_n) + u_{2,n}(x_n) > 0$ and define 
		$$v_{i,n}(x) := \frac{1}{\varepsilon_n} u_{i,n}(x_n + \varepsilon_n x) \quad \forall x \in B_{r_0/(2\varepsilon_n)}, \ \forall i = 1,2.$$
		Then $v_{1,n}(0) = v_{2,n}(0) = 1/2$ for all $n \geq 1$. Moreover, by \eqref{Sys:main-interior}, $(v_{1,n}, v_{2,n})$ is a weak subsolution to 
		\begin{align*}
			\begin{cases}
				-\Delta v_{1,n} \leq -a_{12} \beta_n \varepsilon_n^{p_1 + p_2 + 1} v_{1,n}^{p_1} v_{2,n}^{p_2}, \quad v_{1,n} > 0 \quad & \text{in } B_{r_0/(2\varepsilon_n)}, \\
				-\Delta v_{2,n} \leq -a_{21} \beta_n \varepsilon_n^{p_1 + p_2 + 1} v_{1,n}^{p_1} v_{2,n}^{p_2}, \quad v_{2,n} > 0 \quad & \text{in } B_{r_0/(2\varepsilon_n)}.
			\end{cases}
		\end{align*}
		On the one hand, by Theorem \ref{thm:interior-lip} and noting that $v_{1,n}(0) = v_{2,n}(0) = 1/2$ for all $n \geq 1$, we can find a constant $r_1 > 0$ independent of $n$ such that, up to a subsequence, 
		\begin{align}\label{eq:ap-near-Gamma-n-4}
			v_{i,n} \geq 1/4 \quad \text{in } \overline{B_{r_1}}, \quad \forall i = 1,2.
		\end{align}
		Since $v_{i,n}$ is subharmonic, it follows that 
		$$H_i(\bm{v}_n, 0, r) \geq \frac{\sigma_{N-1}}{16} > 0, \quad \forall r \in [r_1, r_0/(4\varepsilon_n)], \quad \forall i = 1,2,$$
		where we denote $\bm{v}_n := (v_{1,n}, v_{2,n}, 0, \dots, 0)$. On the other hand, since $v_{i,n} > 0$ in $B_{r_0/(2\varepsilon_n)}$, we have 
		$$J_i(\bm{v}_n, 0, r) > 0, \quad \Lambda_i(\bm{v}_n, 0, r) > 0, \quad \forall r \in [r_1, r_0/(4\varepsilon_n)], \quad \forall i = 1,2,$$
		where $J_i$ and $\Lambda_i$ are defined in \eqref{def:H-J-Lambda} with $d' = 0$, and 
		$$M_{12} = M_{21} = \beta_n \varepsilon_n^{p_1 + p_2 + 1}, \quad M_{ij} = 1 \quad \forall 1 \leq i,j \leq k \text{ with } i \neq j \text{ and } \{i,j\} \neq \{1,2\}.$$
		As a consequence, Theorem \ref{thm:ACF-formula} is applicable, yielding that the function 
		$$J_{12;n}(r) := \frac{J_1(\bm{v}_n, 0, r) J_2(\bm{v}_n, 0, r)}{r^4} \exp\left\{ -C \left( \beta_n \varepsilon_n^{p_1 + p_2 + 1} \cdot r^2 \right)^{-1/(2p + 2)} \right\}$$
		is non-decreasing in $[r_1, r_0/(4\varepsilon_n)]$, where $p := \max\{p_1, p_2\} \geq 1$ and $C > 0$ is a constant independent of $n$. In particular, for any $0 < r < r_0/4$, we have 
		\begin{align*}
			J_{12;n}(r/\varepsilon_n) \geq J_{12;n}(r_1) \quad \text{for sufficiently large } n.
		\end{align*}
		By \eqref{eq:ap-near-Gamma-n-3}, we observe that 
		$$\beta_n \varepsilon_n^{p_1 + p_2 + 1} \to 2^{p_1 + p_2 + 1} c_0 \in (0, +\infty) \quad \text{as } n \to \infty.$$
		This, together with \eqref{eq:ap-near-Gamma-n-4}, implies that $J_{12;n}(r_1) \geq C > 0$ for some constant $C > 0$ independent of $n$ and $r$. Hence, for sufficiently large $n$, we deduce that 
		\begin{align*}
			C \leq J_{12;n}(r/\varepsilon_n)
			\leq& C_1 \cdot \frac{J_1(\bm{v}_n, 0, r/\varepsilon_n) J_2(\bm{v}_n, 0, r/\varepsilon_n)}{(r/\varepsilon_n)^4} \\
			=& C_1 \cdot \frac{1}{r^4} \int_{B_r(x_n)} \left( |\nabla u_{1,n}|^2 + \beta_n a_{12} u_{1,n}^{p_1 + 1} u_{2,n}^{p_2} \right) |x - x_n|^{2 - N} dx \\
			&\cdot \int_{B_r(x_n)} \left( |\nabla u_{2,n}|^2 + \beta_n a_{21} u_{2,n}^{p_2 + 1} u_{1,n}^{p_1} \right) |x - x_n|^{2 - N} dx,
		\end{align*}
		where $C_1 > 0$ is a constant independent of $n$ and $r$. Passing to the limit as $n \to \infty$ (up to a subsequence), it follows that 
		$$J_{ij}(\bm{u}, x_0, r) \geq C > 0 \quad \text{for a.e. } 0 < r < r_0/4,$$
		where $J_{ij}$ is defined in \eqref{def:J-ij} and $C > 0$ is a constant independent of $r$. Here, we have used the fact that, by \eqref{Sys:main-interior}, $\bm{u} \in \mathscr{S}(\Omega)$ and $\bm{u}_n\to \bm{u}$ in $H^1_{\text{loc}}(\Omega;\mathbb{R}^k)$ as $n\to\infty$, one can deduce that 
		\begin{align*}
			\beta_n \int_{B_r(x_n)} u_{i,n}^{p_i + 1} \sum_{j \neq i} a_{ij} u_{j,n}^{p_j}
			&= \int_{\partial B_r(x_n)} u_{i,n} \partial_{\nu} u_{i,n} - \int_{B_r(x_n)} |\nabla u_{i,n}|^2 \\
			&\to \int_{\partial B_r(x_0)} u_i \partial_{\nu} u_i - \int_{B_r(x_0)} |\nabla u_i|^2 = 0 \quad \text{as } n \to \infty,
		\end{align*}
		for a.e. $r \in (0, r_0/4)$ and all $1 \leq i \leq k$.
		This implies that $x_0\in \mathcal{R}_{\bm{u}}$, contradicting to $x_0\in \Gamma_{\bm{u}}\setminus \mathcal{R}_{\bm{u}}$. The proof is finished.
	\end{proof}
	\par 
	\begin{proof}[Proof of Theorem \ref{thm:ap-near-gamma-n}]
		The result follows from Lemmas \ref{lem:ap-1} and \ref{lem:ap-2}. 
	\end{proof}
	
	\subsection{Existence of convergent point sequences}
	
	Finally, we use Lemma \ref{lem:decay-Mu-p-general} to prove Proposition \ref{prop:existence-xn}.
	\begin{proof}[Proof of Proposition \ref{prop:existence-xn}]
		Suppose, for contradiction, that $\text{dist}(x_0, \Gamma_n) \geq 2\delta > 0$, where $\delta > 0$ is a constant independent of $n$. Without loss of generality, we may pass to a subsequence (still denoted by $\{\Gamma_n\}$ and $\{\bm{u}_{n}\}$ for simplicity) such that
		\begin{align}\label{Eq:Ab-existence}
			u_{1,n} > u_{i,n} \quad \text{in } B_{\delta}(x_0), \quad \forall 2 \leq i \leq k, \quad \forall n \geq 1.
		\end{align}
		Combining this inequality with system \eqref{Sys:main-interior}, we obtain
		\[
		-\Delta u_{i,n} \leq -\beta_n a_{i1} u_{i,n}^2 \quad \text{in } B_{\delta}(x_0), \quad \forall 2 \leq i \leq k.
		\]
		By Lemma \ref{lem:decay-Mu-p-general}, we then deduce that
		\[
		\sum_{i=2}^k a_{1i} \beta_n u_{i,n} \leq C\delta^{-2} \quad \text{in } B_{\delta/2}(x_0),
		\]
		where $C > 0$ is a constant depending only on $N$ and the matrix $(a_{ij})$. Using this estimate together with the equation satisfied by $u_{1,n}$, we get
		\[
		-\Delta u_{1,n} \geq -C\delta^{-2} u_{1,n} \quad \text{in } B_{\delta/2}(x_0).
		\]
		Passing to the limit as $n \to \infty$, we find that $u_1$ is a non-negative supersolution to $-\Delta u_1 \geq -C\delta^{-2} u_1$ in $B_{\delta/2}(x_0)$ with $u_1(x_0) = 0$. Applying the Strong Maximum Principle, we conclude that $u_1 \equiv 0$ in $B_{\delta/2}(x_0)$. Combining this with \eqref{Eq:Ab-existence}, it follows that $\bm{u} \equiv 0$ in $B_{\delta/2}(x_0)$, which implies $m(x_0) = 0$. This contradicts the fact that $m(x_0) \neq 0$. The proof is thus completed.
	\end{proof}

	\appendix
	\section{Properties of elements of $\mathscr{S}(\Omega)$ in one dimensional case}
	
	In the appendix, we first show that $\mathcal{L}_{1,\bm{u}} = \emptyset$ for any $\bm{u} \in \mathscr{S}(\Omega)$. This result was essentially established in \cite{Terracini-V-Z2019CPAM} for $N=2$, but the proof extends to arbitrary dimensions. For the reader's convenience, we present the proof below.
	
	\begin{lemma}\label{lem:multiplicity-one=empty}
		Let $N\geq 1$ and $\Omega$ be a domain in $\mathbb{R}^N$. For any $\bm{u}\in  \mathscr{S}(\Omega)$, one has $\mathcal{L}_{1,\bm{u}}=\emptyset$.
	\end{lemma}
	\begin{proof}
		Suppose by contradiction that  there exists $x_0\in \mathcal{L}_{1,\bm{u}}$. By definition, there exist $r>0$ and $1\leq i\leq k$ such that $B_r(x_0)\subseteq \Omega$ and 
		$$|\{u_i>0\}\cap B_r(x_0)|>0,\quad |\{u_j>0\}\cap B_r(x_0)|=0,~~\forall j\neq i.$$
		Since $-\Delta u_i\leq 0$ in $B_r(x_0)$ and $-\Delta \widehat{u}_i\geq 0$ in $B_r(x_0)$, we deduce that $-\Delta u_i=0$ in $B_r(x_0)$. Note that $u_i\geq 0$ and $|\{u_i>0\}\cap B_r(x_0)|>0$. The strong maximum principle then yields that $u_i>0$ in $B_r(x_0)$, contradicting to the fact that $u_i(x_0)=0$.
	\end{proof}
	Now, we can utilize Lemma \ref{lem:multiplicity-one=empty} to analyze the structure of the free boundaries of elements in $\mathscr{S}(\Omega)$ in one dimension.
	\begin{proposition}
		Let $N=1$, $\Omega$ be a domain in $\mathbb{R}^N$ and $\bm{u}\in  \mathscr{S}(\Omega)$. Then $\bm{u}$ has at most two nontrivial components in $\Omega$. Moreover, $\#\Gamma_{\bm{u}}\leq 1$ and  $\Gamma_{\bm{u}}=\mathcal{L}_{2,\bm{u}}$. 
	\end{proposition}
	\begin{proof}
		Since $N=1$, by definition, we see that $u_i$ is non-negative and convex in $\Omega:=(a,b)$ for any $1\leq i\leq k$. This implies that if $a<x_1<x_2<b$ and $u_i(x_1)=u_i(x_2)=0$ for some $i$, then $u_i\equiv 0$ in $[x_1,x_2]$. 
		Combining this with $u_i\cdot u_j\equiv 0$ for $i\neq j$ then yields that either $\#\Gamma_{\bm{u}}\leq 1$, or $\#\Gamma_{\bm{u}}\geq 2$ and $\Gamma_{\bm{u}}=\mathcal{L}_{0,\bm{u}}\cup \mathcal{L}_{1,\bm{u}}$. We show that the latter case is impossible. Indeed, if $\#\Gamma_{\bm{u}}\geq 2$ and $\Gamma_{\bm{u}}=\mathcal{L}_{0,\bm{u}}\cup \mathcal{L}_{1,\bm{u}}$, then by Lemma \ref{lem:multiplicity-one=empty}, $\Gamma_{\bm{u}}=\mathcal{L}_{0,\bm{u}}\neq\emptyset$. Since $\bm{u}\not\equiv 0$ in $\Omega$, we have that either $y_1:=\inf\{x\in (a,b):x\in \mathcal{L}_{0,\bm{u}}\}>a$ or $y_2:=\sup\{x\in (a,b):x\in \mathcal{L}_{0,\bm{u}}\}<b$. Without loss of generality, we assume that $y_2<b$. Then it is not difficult to verify that $y_2\in \mathcal{L}_{1,\bm{u}}$, in contradiction with Lemma \ref{lem:multiplicity-one=empty}.
		\par 
		Now we have $\#\Gamma_{\bm{u}}\leq 1$. We split the remaining proof into two cases.
		\par 
		\textbf{Case 1.} $\#\Gamma_{\bm{u}}=0$. In this case, $\Gamma_{\bm{u}}=\emptyset$ and hence $\bm{u}$ has exactly one  nontrivial component which is positive in $\Omega$. 
		\par 
		\textbf{Case 2.} $\#\Gamma_{\bm{u}}=1$. We denote $\Gamma_{\bm{u}} = \{x_0\}$ and claim that $m_{\bm{u}}(x_0) = 2$. Indeed, suppose for contradiction that this is not the case; then by Lemma \ref{lem:multiplicity-one=empty}, $m_{\bm{u}}(x_0) \geq 3$. By definition, for any $r > 0$, there exist points $x_1, x_2, x_3 \in B_r(x_0)$ with $x_1 < x_2 < x_3$ and three distinct indices $i_1, i_2, i_3 \in \{1, 2, \dots, k\}$ such that $u_{i_j}(x_j) > 0$ for $1 \leq j \leq 3$. Since $\bm{u}$ is segregated, this implies $u_{i_2}(x_1) = u_{i_2}(x_3) = 0$; hence, by convexity of $u_{i_2}$, we have $u_{i_2} \equiv 0$ on $[x_1, x_3]$, contradicting to $u_{i_2}(x_2) > 0$. As a consequence, $\bm{u}$ has exactly two nontrivial components in $\Omega$. This follows from the fact that if the number of nontrivial components of $\bm{u}$ exceeds two, then $\#\Gamma_{\bm{u}} \geq 2$, which contradicts $\#\Gamma_{\bm{u}} = 1$. The proof is complete.
	\end{proof}

\end{document}